\documentclass[a4paper,11pt]{article}
\usepackage{color}

\usepackage{amsmath,amsthm,amssymb,mathrsfs,latexsym,amsfonts}
\usepackage{graphicx,psfrag,epsfig}
\usepackage{bbm}
\usepackage{a4wide}
\usepackage[english]{babel}
\usepackage[latin1]{inputenc}
\usepackage{authblk}

\numberwithin{equation}{section}

\newtheorem{theorem}{Theorem}[section]
\newtheorem{lemma}[theorem]{Lemma}
\newtheorem{proposition}[theorem]{Proposition}
\newtheorem{problem}{Problem}
\newtheorem{question}{Question}

\newtheorem{definition}{Definition\rm}
\newtheorem{conjecture}{Conjecture\rm}
\newtheorem{remark}[theorem]{\it Remark}

\newcounter{paraga}[section]

\newcommand{\RRR}{{R}}
\newcommand{\DC}{{\rm DC}}
\newcommand{\BNF}{{\rm BNF}}
\newcommand{\N}{\mathbb{N}}
\newcommand{\Z}{\mathbb{Z}}
\newcommand{\Q}{\mathbb{Q}}
\newcommand{\R}{\mathbb{R}}
\newcommand{\C}{\mathbb{C}}

\newcommand{\cN}{\mathcal{N}}

\renewcommand{\a}{\alpha}

\begin{document}

\def\MP{\,{<\hspace{-.5em}\cdot}\,}
\def\SP{\,{>\hspace{-.3em}\cdot}\,}
\def\PM{\,{\cdot\hspace{-.3em}<}\,}
\def\PS{\,{\cdot\hspace{-.3em}>}\,}
\def\EP{\,{=\hspace{-.2em}\cdot}\,}
\def\PP{\,{+\hspace{-.1em}\cdot}\,}
\def\PE{\,{\cdot\hspace{-.2em}=}\,}
\def\N{\mathbb N}
\def\C{\mathbb C}
\def\Q{\mathbb Q}
\def\R{\mathbb R}
\def\T{\mathbb T}
\def\A{\mathbb A}
\def\Z{\mathbb Z}
\def\demi{\frac{1}{2}}

\newtheorem{Main}{Theorem}
\newtheorem{Coro}{Corollary}
\newtheorem*{Principal}{Theorem}

\renewcommand{\theMain}{\Alph{Main}}
\renewcommand{\theCoro}{\Alph{Coro}}

\setcounter{tocdepth}{3}

\begin{titlepage}
\author{Abed Bounemoura\footnote{CNRS - CEREMADE - IMCCE/ASD, abedbou@gmail.com} {} and Bassam Fayad\footnote{CNRS - IMJ-PRG and Centro Ennio De Giorgi, bassam.fayad@imj-prg.fr} {} and Laurent Niederman\footnote{Laboratoire Math\'ematiques d'Orsay \& IMCCE/ASD, Laurent.Niederman@math.u-psud.fr}}
\title{\LARGE{\textbf{Double exponential stability for generic real-analytic elliptic equilibrium points}}}
\end{titlepage}

\maketitle

\begin{abstract}
We consider the dynamics in a neighborhood of an elliptic equilibrium point with a Diophantine frequency of a symplectic real analytic vector field and we prove the following result of effective stability. Generically, both in a topological and measure-theoretical sense, any solution starting sufficiently close to the equilibrium point remains close to it for an interval of time which is doubly exponentially large with respect to the inverse of the distance to the equilibrium point. We actually prove a more general statement assuming the frequency is only non-resonant. This improves previous results where much stronger non-generic assumptions were required. 
\end{abstract}
 
\tableofcontents 
\section{Introduction}\label{s1}

The aim of this paper is to study the effective stability of elliptic equilibrium points in Hamiltonian systems. Our main result will be that the flow of a real analytic Hamiltonian $H$  in $n$ degrees of freedom having a Diophantine equilibrium point at the origin is doubly exponentially stable at the origin under an open and dense condition of full Lebesgue measure which only involves the part of the power expansion of $H$ that contains the terms of degree between $3$ and $[\frac{n^2+4}{2}]$. This result will be derived from a more general effective stability result for non-resonant elliptic equilibrium points. Before stating the exact results, let us start by describing the general setting.

\subsection{Stability of elliptic equilibrium points}\label{s11}

We consider a symplectic manifold $(M,\Omega)$ of dimension $2n$, $n\in \N$, where $\Omega$ is an everywhere non-degenerate closed $2$-form, a smooth symplectic vector field $X$ on $M$ (meaning that the one-form $i_X \Omega$ is closed, or, equivalently, that the Lie derivative $\mathcal{L}_X \Omega$ vanishes identically) and an equilibrium point $p^* \in M$, that is $X(p^*)=0$. We are interested in studying whether $p^*$ is \emph{stable} in the following sense (in the sense of Lyapounov): given any neighborhood $U$ of $p^*$, there exists a smaller neighborhood $V$ of $p^*$ such that for any point $p_0 \in V$, the unique solution $p(t)$ of $X$ starting at $p_0$ (that is, the unique curve $p(t)$ satisfying $\dot{p}(t)=X(p(t))$ and $p(0)=p_0$) is defined and contained in $U$ for all time $t\in \R$. 

The problem being local, there are some obvious simplifications. First, by the classical theorem of Darboux, we may assume without loss of generality that $(M,\Omega)=(\R^{2n},\Omega_0)$ where $\Omega_0$ is the canonical symplectic structure of $\R^{2n}$, and that $p^*=0 \in \R^{2n}$. Then, we may also assume that the one-form $i_X \Omega_0$ is in fact exact, meaning that $X$ is \emph{Hamiltonian}: given a primitive $H$ of $i_X \Omega_0$ and letting $J_0$ be the canonical complex structure of $\R^{2n}$, the vector field can be simply written $X=X_H=J_0\nabla H$, where the gradient is taken with respect to the canonical Euclidean structure of $\R^{2n}$. Therefore $0$ is an equilibrium point of $X_H$ if and only if it is a critical point of $H$, that is $\nabla H(0)=0$. Moreover, the Hamiltonian function $H$ being defined only modulo a constant, it is not a restriction to impose that $H(0)=0$.  

Let $(x,y)=(x_1,\dots,x_n,y_1,\dots,y_n)$ be symplectic coordinates defined in a neighborhood of the origin $0 \in \R^{2n}$ so that $(\dot{x}(t),\dot{y}(t))=X_H(x(t),y(t))$ is equivalent to the system 
\[ \dot{x}(t)=\partial_y H(x(t),y(t)), \quad \dot{y}(t)=- \partial_x H(x(t),y(t)). \]
Since $H(0)=0$ but also $\nabla H(0)=0$, the Taylor expansion of $H$ at the origin is of the form
\[ H(x,y)=H_2(x,y) + O_3(x,y) \]
where $H_2$ is the quadratic part of $H$ at the origin and where $O_3(x,y)$ contains terms of order at least $3$ in $(x,y)$. We can now define the \emph{linearized} Hamiltonian vector field at the origin to be the Hamiltonian vector field associated to $H_2$:
\[X_{H_2}=J_0\nabla H_2=J_0A\] 
where $A$ is the symmetric $2n \times 2n$ matrix (corresponding, up to a factor $2$, to the Hessian of $H$ at the origin) such that $H_2(x,y)=A(x,y)\cdot (x,y)$. In order to study the stability of the equilibrium point, it is useful to first study its \emph{linear stability}, that is, the stability of the origin for the linearized vector field (the latter is obviously equivalent to the boundedness of all its solutions). The matrix $J_0A$ possesses symmetries which imply, in particular, that if $\lambda$ is an eigenvalue then so is its complex conjugate $\bar{\lambda}$. It follows that if $J_0A$ has an eigenvalue with a non zero real part, it also has an eigenvalue with positive  real part and in this case one can find solutions of the linear system that converges to infinity at an exponential rate: this implies linear instability but also instability in the sense of Lyapounov. We will say that the equilibrium point is \emph{elliptic} if the spectrum of the matrix $J_0A$ is both purely imaginary and simple. This implies linear stability, while linear stability is equivalent to $J_0A$ being semi-simple and its spectrum purely imaginary (but the assumption that the spectrum is simple, which is already a non-resonance assumption, will be important for us in the sequel). Note that if we only assumed the spectrum to be purely imaginary, then, if the matrix $J_0A$ has a non-trivial Jordan block, one can find solutions for the linearized vector field converging to infinity at a polynomial rate, implying linear instability (but not necessarily instability in the sense of Lyapounov).

So from now on, $0 \in \R^{2n}$ is assumed to be an elliptic equilibrium point of the Hamiltonian system defined by a smooth function $H$. Since the spectrum of the matrix $JA$ is invariant by complex conjugation, it has necessarily the form $\{\pm i \alpha_1, \dots, \pm i \alpha_n\}$ for some vector $\alpha=(\alpha_1,\dots,\alpha_n) \in \R^n$ with distinct components: this is usually called the \emph{frequency vector}. By a result of linear symplectic algebra (a simple case of a theorem due to Williamson, see \cite{AKN97}) one can find a linear symplectic map which puts the quadratic part into diagonal form (this result requires the components of $\alpha$ to be distinct): hence we can assume that $H$ is of the form
\begin{equation} \label{Hamintro} H(x,y)= \sum_{j=1}^n \alpha_j(x_j^2+y_j^2)/2 + O_3(x,y), \end{equation}
where our standing assumption from now on is that the Hamiltonian $H$ is \emph{real-analytic}, hence it can be extended as a holomorphic function on some complex neighborhood of the origin. Also,  we will always assume that the frequency vector $\alpha$ is \emph{non-resonant}, that is for any non-zero $k\in \Z^n$, the Euclidean scalar product $k\cdot \alpha$ is non-zero.

Note that fixing such coordinates imposes a sign on the components of the vector $\alpha \in \R^n$. Given a point $(x,y)\in \R^{2n}$, let us define $I(x,y) \in \R^n_+$ by
\[ I(x,y)=(I_1(x_1,y_1),\dots,I_n(x_n,y_n)), \quad I_j(x_j,y_j)=(x_j^2+y_j^2)/2, \quad 1 \leq j \leq n \]
so that $H$ can be written again as
\begin{equation*}
H(x,y)=\alpha \cdot I(x,y)+O_3(x,y):=h_1(I(x,y))+O_3(x,y) \end{equation*}

The linearized vector field, associated to $h_1(I(x,y))=\alpha \cdot I(x,y)$, is easily integrated: given an initial condition $(x_0,y_0)$, the corresponding solution $(x(t),y(t))$ is quasi-periodic. More precisely, letting $I_0=I(x_0,y_0) \in \R^n_+$, one obviously has $I(x(t),y(t))=I_0$ for all time $t\in\R$ and so the set $T(I_0)=\{(x,y) \in \R^{2n} \; | \; I(x,y)=I_0\}$ is an invariant torus, the dimension of which equals the number of strictly positive components of $I_0$, and on which the flow is just a flow of translation. The same holds true in fact for an arbitrary Hamiltonian depending only on the quantity $I(x,y)$, and such Hamiltonians will be called here \emph{integrable}. 

A central question in Hamiltonian dynamics is then the following. 
\begin{problem}
For a Hamiltonian $H$ as in~\eqref{Hamintro}, is the origin stable or unstable?
\end{problem}
By stable we mean Lyapunov stable in the sense that points near the origin remain in a  neighborhood of the origin. Other notions of stability may also be addressed as we will see below.

\subsection{Perturbation of completely integrable systems.} 
If $H$ is integrable, the origin is obviously stable. Now in general $H$ is, in a small neighborhood of the origin, a small perturbation of the integrable Hamiltonian $h_1$ and thus classical techniques from perturbation theory (such as KAM theory, Aubry-Mather theory, Nekhoroshev estimates or Arnold diffusion) may be used to tackle the problem. However, this setting of singular perturbation theory is quite different from the usual context of a perturbation of an integrable Hamiltonian system in action-angle coordinates, that is, a Hamiltonian of the form $h(I)+\varepsilon f(\theta,I)$, where $\varepsilon$ is the small parameter and $(\theta,I) \in \T^n \times \R^n$. 

A first obvious difference is that for a Hamiltonian $H$ as in~\eqref{Hamintro}, one cannot introduce action-angle coordinates on a full neighborhood of the origin: indeed, if we let $I_j=I_j(x_j,y_j)$, then the symplectic polar coordinates
\[ x_j=\sqrt{2I_j}\cos\theta_j, \quad y_j=\sqrt{2I_j}\sin\theta_j, \quad 1 \leq j \leq n  \]
are analytically well-defined only away from the axes $I_j=0$. This amounts to the fact that for a Hamiltonian integrable in a neighborhood of an elliptic equilibrium point, the foliation by invariant tori is singular in the sense that the dimension of each leaf is non-constant (it varies from $0$ to $n$), whereas in action-angle coordinates this foliation is regular. 

A second difference lies in the fact that for Hamiltonians of the form $h(I)+\varepsilon f(\theta,I)$ the perturbation $f$ is usually considered as arbitrary whereas in~\eqref{Hamintro} the perturbation is more restricted as it is given by the higher order terms $O_3(x,y)$.

Finally, a third difference is that, under the assumption that $\alpha$ is non-resonant, a Hamiltonian $H$ as in~\eqref{Hamintro} possesses infinitely many integrable approximations $h^m$, for any integer $m \geq 2$ (given by the Birkhoff normal form, see below for more details) which are uniquely determined (once the vector $\alpha$ is fixed). This is in sharp contrast with a Hamiltonian of the form $h(I)+\varepsilon f(\theta,I)$ which does not have, in general, further integrable approximations.  

As we will see below, these differences have the following general effect: in a neighborhood of an elliptic equilibrium point, as opposed to a perturbation of an integrable system in action-angle coordinates, stability properties are stronger and instability properties are harder to exhibit.  

\subsection{KAM stability}\label{s12}

Due to the classical KAM (Kolmogorov-Arnold-Moser) theory, one can prove, for any number of degrees of freedom and assuming some non-degeneracy assumption (on the higher order terms $O_3(x,y)$), that the elliptic equilibrium point is \emph{KAM stable}: in any sufficiently small neighborhood of the origin, there exist a positive measure set of Lagrangian invariant tori, on which the dynamics is conjugated to a linear flow, having the origin as a Lebesgue density point. In general,  KAM stability does not have direct implications on Lyapounov stability.

There are however two cases for which one knows that stability holds true for a Hamiltonian $H$ as in~\eqref{Hamintro}. 

The first case is when the quadratic part $H_2$ is sign-definite, or, equivalently, when the components of the vector $\alpha \in \R^{n}$ have the same sign (and this includes, as a trivial instance, the case $n=1$). Indeed, the Hamiltonian function has then a strict minimum (or maximum) at the origin, and as this function is constant along the flow (it is in particular a Lyapounov function) one can construct, using standard arguments, a basis of neighborhoods of the origin which are invariant, and the latter property is obviously equivalent to stability.

The second case is when $n=2$ and when the so called Arnold iso-energetic non-degeneracy condition is satisfied. Then, KAM stability occurs in every energy level passing sufficiently close to the origin,  implying Lyapounov stability as the two-dimensional tori disconnect each three-dimensional energy level (see for instance \cite{Arn61} and \cite{Mos62}). It is easy to see that the Arnold iso-energetic non-degeneracy condition is generic in measure and topology as a function of the coefficients of the $O_4(x,y)$ part of the Taylor expansion of $H$ around the origin. 

Related to the results that we will expose in the following sections, let us mention that it is sometimes possible to replace the non-degeneracy assumption in the study of stability by arithmetic conditions on the frequency vector $\a$ of the linear part of the flow at the equilibrium. 
Indeed, in the analytic setting, Herman conjectured the KAM  stability (without the Lebesgue density requirement) of Diophantine equilibria without any non-degeneracy assumption. In (\cite{Her98}) he made the following conjecture (in the slightly different context of symplectic maps).

\begin{conjecture}[Herman]
Assuming that $\alpha$ is Diophantine, in any sufficiently small neighborhood of the origin there exists a set of positive Lebesgue measure of Lagrangian invariant tori. 
\end{conjecture} 

Recall that  $\a \in \R^n$ is said to be Diophantine if for some constant $\gamma>0$ and exponent $\tau \geq n-1$ it holds that $|k\cdot \alpha| \geq \gamma|k|_1^{-\tau}$ for all $k=(k_1,\dots,k_n) \in \Z^{n} \setminus \{0\}$, where $|k|_1:=|k_1|+\cdots+|k_n|$. We then use the notation $\a \in \DC(\tau,\gamma)$.

Herman's conjecture is true for $n=2$, even in the smooth category, as it was proved by R{\"u}ssmann (see for instance \cite{Rus02} and \cite{FK09} in the discrete case, for respectively real-analytic and smooth maps, and \cite{EFK13} or \cite[Section 7.1]{EFKduke} in the continuous case) but unknown in general (see \cite{EFK13,EFKduke} for partial results).  Note that KAM stability of a Diophantine equilibrium for a Hamiltonian in the case $n=2$ does not imply {\it a priori} Lyapunov stability.

Observe also that this KAM stability phenomenon without any non-degeneracy condition has no counterpart for perturbed integrable system in action-angle coordinates, since any integrable system that does not satisfy  the so-called R{\"u}ssmann non-degeneracy condition can be simply perturbed so that no invariant torus survives (see \cite{Sev03}).

\subsection{Arnold's diffusion  conjecture}\label{s13}

Arnold conjectured that apart from these two cases (the case of a sign-definite quadratic part, and generically for $n=2$), an elliptic equilibrium point is generically unstable. More precisely, in \cite{Arn94} one can find the following conjecture.

\begin{conjecture}[Arnold] An elliptic equilibrium point of a generic analytic Hamiltonian system is Lyapounov unstable, provided $n \geq 3$ and the quadratic part of the Hamiltonian function at the equilibrium point is not sign-definite.
\end{conjecture}  

This conjecture is wide open, to such an extent that under our standing assumptions (real-analyticity of the Hamiltonian and a non-resonance condition on the frequency vector) not a single example is known.

If the frequency vector is resonant, it is quite trivial to construct an example of unstable elliptic equilibrium point (see \cite{Mos60}). The genericity is, however, still open (see \cite{KMV04} for an announcement on some partial results). 

If the Hamiltonian is smooth non-analytic, examples have been constructed by Douady-Le Calvez (\cite{DLC83}) for $n=3$ and by Douady (\cite{Dou88}) for any $n\geq 3$, but here also, genericity seems out of reach. 

\subsection{Effective stability}\label{s14}

The aim of this paper is to investigate the so called effective stability of an elliptic equilibrium point. More precisely, given $r$ sufficiently small and any initial condition $(x_0,y_0)$ at a distance at most $r$ from the origin, we are interested in the largest positive time $T(r)$ for which the solution $(x(t),y(t))$, starting at $(x_0,y_0)$, stays at a distance at most $2r$ from the origin, for all $|t| \leq T(r)$. Arnold's conjecture states that for $n\geq 3$,  it holds generically that $T(r)<\infty$. At the moment there is no other conjectural upper bound on $T(r)$. In this paper, we will be interested in lower bound on $T(r)$. Let us  first recall some previous results.

First, without any assumptions, it is easily seen from the equations of motion that $T(r)$ is at least of order $r^{-1}$. Then, given an integer $K \geq 4$, with the assumption that $H$ is smooth and $\alpha$ is non-resonant up to order $K$, that is
\[   k \in \Z^n, \quad 0 < |k|_1 \leq K \Longrightarrow k\cdot\alpha \neq 0\]
the following statement can be proved (see \cite{Bir66} or \cite{Dou88}): there exists a symplectic transformation $\Phi^K$, well-defined in a neighborhood of the origin, such that 
\begin{equation} H\circ \Phi^K(x,y)=\alpha\cdot I(x,y)+h^m(I(x,y))+f^K(x,y)  \label{BNFK} \tag{BNF} \end{equation}
where $h^m$ is a polynomial of degree $m=[K/2]$ (the integer part of $K/2$) in $n$ variables, with vanishing constant and linear terms, and $f_K$ is of higher order $O_{K+1}(x,y)$. 
The polynomial $\alpha\cdot I(x,y)+h^m(I(x,y))$ is usually called the Birkhoff normal form of $H$ of order $K$. Since the term $\alpha\cdot I(x,y)$ will be fixed in the sequel we will denote $h^m(I(x,y))$ by $\BNF_K(H)$. The polynomial $\BNF_K(H)$ is uniquely defined, but, in general, this is not the case for the coordinate change function $\Phi^K$ (although there is a distinguished choice of a generating function for $\Phi^K$). An obvious consequence of \eqref{BNFK}  is that, in this case, $T(r)$ is at least of order $r^{-K+1}$ at the origin (naturally, the neighborhood in which the effective stability holds depends on $K$ and may be very small depending in particular on the arithmetics of $\a$). Thus if $\alpha$ is non-resonant and $H$ is of class $C^\infty$, $T(r)$ becomes larger near the origin than any power of $r^{-1}$. 
Observe that if $\alpha$ is non-resonant, one can find a formal symplectic transformation $\Phi^{\infty}$ and a unique formal series $h^{\infty}$ in $n$ variables such that $H\circ \Phi^{\infty}(x,y)=h^{\infty}(I(x,y))$. However, the formal transformation $\Phi^{\infty}$ is in general divergent (see \cite{Sie41}), and the convergence problem for the formal series $h^{\infty}$ is still an open problem (see \cite{PM03} for some results). 

Now with the assumption that the Hamiltonian $H$ is real-analytic, exponentially large lower bounds for $T(r)$ have been obtained in two different contexts.

First, if $\alpha$ is Diophantine, $\a \in \DC(\tau,\gamma)$, one can prove that $T(r)$ is at least of order $\exp\left((\gamma r^{-1})^{\frac{1}{\tau+1}}\right)$. This is obtained by estimating the size of the remainder term $f^K$ in the Birkhoff normal form of order $K$, and then choosing $K=K(r)$ as large as possible in terms of $r$ (see \cite{GDFGS} or \cite{DG96b} for slightly better estimates). One should point out here that actually for any non-resonant $\alpha$ one can associate a function $\Delta_\alpha(r)$ and prove that $T(r)$ is at least of order $\exp\left(\Delta_\alpha(r^{-1})\right)$ (see Section \ref{s21} below for the definition of this function $\Delta_\alpha$).  In the Diophantine case one has $\Delta_\alpha(x)\geq (\gamma x)^{\frac{1}{\tau+1}}$ and the classical result is thus recovered.

Then, in a different direction, assuming only that $\alpha$ is non-resonant up to order $K$, for some $K \geq 4$, but requiring that the quadratic form $h^2$ is positive definite (which implies that $h_1+h^2$, and then $h_1+h^m$ for any $m\geq 2$, is convex in a neighborhood of the origin), it has been proved that $T(r)$ is at least of order $\exp\left(r^{-\frac{K-3}{2n}}\right)$: this was established independently by Niederman (\cite{Nie98}) and Fasso-Guzzo-Benettin (\cite{BFG98}) and later clarified by Pöschel (\cite{Pos99a}). The proof is based on the implementation of Nekhoroshev's estimates (\cite{Nek77}, \cite{Nek79}): observe that in the absence of action-angle coordinates, this implementation is not straightforward and it was only conjectured by Nekhoroshev.  

It is a remarkable fact that both exponential stability results under one of the two hypothesis : 1) $\alpha$ is Diophantine or 2) $h^2$ is positive definite, can be combined into a double exponential stability result if both 1) and 2) hold. This was first done by Giorgilli and Morbidelli in \cite{MG95} in the context of a quasi-periodic invariant Lagrangian torus. In our context of an elliptic equilibrium, the result of \cite{MG95} would amount to double exponential stability of a Diophantine equilibrium provided $h^2$ is positive definite, or more precisely that $T(r)$ is at least of order $\exp\left((\exp((\gamma r^{-1})^{\frac{1}{1+\tau}}))^{\frac{1}{2n}}\right)$. Even though the condition that $h^2$ is positive definite is open, it is far from being generic in any sense and recently some efforts have been made to improve this result, especially in \cite{Bou11etds} and \cite{Nie13}. In \cite{Bou11etds}, using results from \cite{Nie07} and \cite{BN12}, it was proved that under a certain condition on the formal Birkhoff series $h_\infty$, the double exponential stability holds true. This condition, which includes the condition that $h^2$ is positive definite as a particular case, was proved to be prevalent (a possible generalization of ``full measure" in infinite dimensional spaces) in the space of all formal series. This result has at least two drawbacks. First, although this condition can be termed generic in a measure-theoretical sense, it is far from being generic in a topological sense. Secondly, this condition was only formulated in the space of formal series, and it was unclear whether prevalent Hamiltonians have formal Birkhoff series satisfying this condition. This second issue was partially solved in \cite{Nie13}: it is proved there that a prevalent Hamiltonian has a formal Birkhoff series satisfying a condition close to the one introduced in \cite{Bou11etds}, yielding a result which is only intermediate between exponential and double exponential stability.

The aim of this paper is to improve those results by establishing that generically, and in a strong sense, the double exponential stability holds true.

\subsection{Main results} \label{s21} 

We start by some reminders and notations that will be useful in our statements. Let $H$ be a real analytic Hamiltonian on $\R^{2n}$ having an elliptic equilibrium point at the origin  with a non-resonant  frequency vector $\alpha$, that is $H$ is as in \eqref{Hamintro}.

\begin{itemize}

\item For vectors in $\C^{2n}$, $\|\,.\,\|$ denotes the norm defined as 
\begin{equation}\label{norme}
\|z\|:=\max_{1 \leq j \leq n} \sqrt{|z_j|^2+|z_{n+j}|^2}, \quad z=(z_1,\dots,z_n,z_{n+1}, \dots, z_{2n}) 
\end{equation}
and for vectors in $\C^n$, $\|\,.\,\|$ denotes the usual Euclidean norm
\begin{equation}\label{norme2}
\|I\|:=\sqrt{|I_1|^2+\cdots+|I_n|^2}, \quad I=(I_1,\dots,I_n,).
\end{equation}
It will be more convenient to use these different norms for vectors in $\C^{2n}$ or in $\C^n$, and we hope that this abuse of notations will not confuse the reader.

\item We suppose that the radius of  convergence of $H$  is strictly larger than some $\RRR>0$ and let $\|H\|_\RRR$ be the sup norm of $H$ in the open complex ball in $\C^{2n}$ centered at the origin of radius $R$ that we denote by
\begin{equation}\label{bouleR}
\mathcal{B}_{R}:=\{z \in \C^{2n} \; | \; \|z\|<R\}.
\end{equation} 
We also define the real ball $B_R:=\mathcal{B}_{R} \cap \R^{2n}$.

\item We denote by $P(n,m)$ the set  of polynomials of degree $m$ in $n$ variables. We let $P_2(n,m)\subset P(n,m)$ be the subspace of polynomials with a vanishing affine part, and $P_3(n,m) \subset P(n,m)$ the subset of polynomials that have a vanishing affine and quadratic part. 

\item We denote by $\tilde{H}_m \in P_3(2n,m)$ the part of the power expansion of $H$ that contains the terms of degree between $3$ and $m$ included.  

\item Having fixed the number of degrees of freedom $n$, in all the sequel, we let
$$K_0=K_0(n):= n^2+4, \quad m_0=m_0(n):=[K_0(n)/2].$$

\item The vector $\alpha$ is supposed to be non-resonant: this means that for any integer $K \geq 1$,
\begin{equation}\label{fonctionpsi}
\Psi_\alpha(K)=\max\{ |k\cdot\alpha|^{-1} \; | \; k \in \Z^n, \; 0 < |k|_1=|k_1|+\cdots|k_n|\leq K\} < +\infty. 
\end{equation}
We define, as in \cite{Bou12}, the function
\[ \Delta_\alpha(x)=\sup\{K \geq 1 \; | \; K\Psi_\alpha(K) \leq x\}. \]
Observe that if $\alpha \in \DC(\tau,\gamma)$, then $\Psi_\alpha(K) \leq \gamma^{-1} K^\tau$ and hence 
\begin{equation} \Delta_\alpha(x) \geq (\gamma x)^{\frac{1}{1+\tau}} \label{deltadioph} \end{equation} 

\item Recall that for $H$ as in \eqref{Hamintro}, there exists for every integer $K \geq 4$ a real analytic symplectic transformation $\Phi^K$ defined in the neighborhood of the origin such that 
\[ H\circ \Phi^K(x,y)=\alpha\cdot I(x,y)+h^m(I(x,y))+f^K(x,y)\] 
where $h^m$ is a polynomial of degree $m=[K/2]$ (the integer part of $K/2$) in $n$ variables, with vanishing constant and linear terms, and $f_K$ is of higher order $O_{K+1}(x,y)$. 
We denoted $h^m$ by $\mathrm{BNF}_K(H)$.  By uniqueness of the Birkhoff normal form 
we have for $K=2m \geq 4$, a well defined map 
\[ \begin{array}{lll}
\mathrm{BNF}_K : & P_3(2n,K) &\longrightarrow P_2(n,m) \\
     & \tilde{H}_K &\longmapsto h^m = \mathrm{BNF}_K( \tilde{H}_K )= \mathrm{BNF}_K(H).
\end{array} \]

\end{itemize}

Our main result is the following.

\begin{Main}\label{infthm} 
Let $H$ be a real analytic Hamiltonian on $\R^{2n}$ having an elliptic equilibrium point at the origin  with a non-resonant  frequency vector $\alpha$. There exists an open and dense set of full Lebesgue measure  $ \cN_n(\a) \in P_3(2n,K_0)$ such that if  $\tilde{H}_{K_0} \in \cN_n(\a)$, then there exists $r^*,c,c',c''>0$ that depend only on $n,\RRR,\|H\|_R, \a$ and $\tilde{H}_{K_0}$ such that if $r \leq r^*$, then 
$$T(r) \geq \exp\left( c r^{-2} \exp\left(c' \Delta_\a\left(c'' r^{-1}\right)\right)\right). $$
If $\alpha \in \DC(\tau,\gamma)$, there exists an open and dense set of full Lebesgue measure  $ \cN_n(\a) \in P_3(2n,K_0)$ such that if  $\tilde{H}_{K_0} \in \cN_n(\a)$, then 
 there exists $r^*$ and $C$ that depend only on $n,\RRR,\|H\|_R, \a$, and $\tilde{H}_{K_0}$ such that if $r \leq r^*$, then 
$$T(r) \geq \exp\left(\exp\left( C {r}^{-\frac{1}{\tau+1} } \right)\right).$$

\end{Main}

Observe that since $c'$ and $c''$ will not depend on $\a$ (see \eqref{ccc}), it follows from \eqref{deltadioph} that the constant $C$ that appears under the double exponential in the Diophantine case is actually of the form  $C=\gamma^{\frac{1}{\tau+1} }C'$ where $C'$ does not depend on $\a$. 
Theorem \ref{infthm} improves all previous results contained in \cite{MG95}, \cite{Bou11etds} and \cite{Nie13}. In the course of its proof, we will also have to extend the results on exponential stability contained in \cite{Nie98}, \cite{BFG98} and \cite{Pos99a}.

\begin{remark}
{\rm Observe that even though $\Delta_\a(r^{-1})$ goes to infinity as $r$ goes to zero, the speed of convergence can be arbitrarily slow but the statement implies that $T(r)$ is always at least of order $\exp(cr^{-2})$. From the proof of the theorem, one can easily obtain the following statement: fixing $k \in \N^*$, $k \geq 2$, and allowing the constants $r_k^*$ and $c_k$ to depend also on $k$, one has
$$T(r) \geq \exp\left( c_k r^{-k} \exp\left(c' \Delta_\a\left(c'' r^{-1}\right)\right)\right)$$
which is always at least of order $\exp(c_kr^{-k})$. As a matter of fact, the weaker estimate
$$T(r) \geq \exp\left( c_k r^{-k}\right)$$
can be obtained if one only assumes $\alpha$ to be non-resonant up to a sufficiently high order depending on $k$ and $n$.}
\end{remark}

\begin{remark} {\rm The Diophantine condition $\a \in \DC(\tau,\gamma)$ is sometimes called an asymptotic Diophantine condition. A strictly weaker condition, called uniform Diophantine condition, requires the existence of an increasing  sequence $K_j \in \N$, $K_j \to \infty$, such that $|k\cdot\alpha| \geq \gamma K_j^{-\tau}$ for every $k \in \Z^n \setminus \{0\}$ with $|k|_1\leq K_j$. This gives  $\Psi_\alpha(K_j) \leq \gamma K_j^\tau$ and  Theorem~\ref{infthm} would then imply that   there exists a sequence $r_j \to 0$ such that } 
$$T(r_j) \geq \exp\left(\exp\left(C r_j^{-\frac{1}{\tau+1} } \right)\right).$$
\end{remark}

The notion of \emph{stably steep} polynomials, which can be implicitly found in the work of Nekhoroshev (\cite{Nek73}), will be important in the proof of Theorem \ref{infthm}.

\begin{definition}[Stably steep polynomials] \label{stabsteep} 
A polynomial $P_0 \in P_2(n,m)$ is called stably steep if there exist a neighborhood $V$ of $P_0$ in $P_2(n,m)$ and positive constants $C,\delta$ such that for any integer $l\in [1,n-1]$, any $P \in V$ and any vector subspace $\Lambda \subseteq \R^n$ of dimension $l$, letting $P_\Lambda$ be the restriction of $P$ to $\Lambda$, the inequality
\[ \max_{0 \leq \eta \leq \xi}\;\min_{||x||=\eta, \; x \in \Lambda}||\nabla P_\Lambda(x)||>C\xi^{m-1} \]
holds true for all $0 < \xi \leq\delta$, where $||\,.\,||$ is the usual Euclidean norm defined in~\eqref{norme2}.
\end{definition}

The set of stably steep polynomials in $P_2(n,m)$ will be denoted by $SS(n,m)$.

Theorem \ref{infthm} will clearly follow from the combination of the following two statements, Theorems \ref{mainthm01} and \ref{mainthm02}, with the set $\cN_n(\a)$ being defined as 
$\cN_n(\a):= \mathrm{BNF}_{K_0}^{-1}(SS(n,m_0))$.

Our first statement is that the set of Hamiltonians with stably steep BNF of order $K_0$ have doubly exponentially stable equilibria.  

\begin{Main}\label{mainthm01} 
Let $H$ be a real analytic Hamiltonian on $\R^{2n}$ having an elliptic equilibrium point at the origin  with a non-resonant  frequency vector $\alpha$. 
If \[ \mathrm{BNF}_{K_0}({H}) =h^{m_0} \in SS(n,m_0)\] 
then the conclusions of Theorem \ref{infthm} hold.
\end{Main}

The second statement shows that the condition $\mathrm{BNF}_{K_0}({H})= \mathrm{BNF}_{K_0}(\tilde{H}_{K_0}) \in SS(n,m_0)$ is generic in a strong sense.

\begin{Main}\label{mainthm02} For any non-resonant $\a \in \R^n$, the complement of $\mathrm{BNF}_{K_0}^{-1}(SS(n,m_0))$ in $P_3(2n,K_0)$ is contained in a  semi-algebraic subset of positive codimension. In particular, $\mathrm{BNF}_{K_0}^{-1}(SS(n,m_0))$ is a dense open subset of $P_3(2n,K_0)$ of full Lebesgue measure. 
\end{Main}

\begin{proof}[Proof of Theorem \ref{infthm}] Putting together Theorem \ref{mainthm01} and \ref{mainthm02} immediately yields Theorem \ref{infthm} if we take  $\cN_n(\a)= \mathrm{BNF}_{K_0}^{-1}(SS(n,m_0))$. 
\end{proof}

To prove Theorem  \ref{mainthm02}, 
we will show that the complement of $SS(n,m_0)$ in $P_2(n,m_0)$ is contained in a semi-algebraic subset of codimension at least one. This will be done in Sections  \ref{subsec.steep.generic}, \ref{subsec.steep.BNF} and Appendix \ref{app.steep}.

Theorem \ref{mainthm01} will follow (see Section \ref{nekhoimpliesstability}) from a version of the Nekhoroshev exponential stability result adapted to our singular perturbation setting, that we now present and that will be proven in Section \ref{sec.stability}. 

\begin{itemize}

\item For vectors in $\C^{n}$, it will be convenient to also use the sup norm $|\,.\,|$ defined as
\begin{equation}\label{norme3}
|I|:=\max\{|I_1|,\dots,|I_n|\}, \quad I=(I_1,\dots,I_n).
\end{equation} 
This norm allows an easier comparison between $I(z) \in \C^n$ and $z \in \C^{2n}$: indeed, we have $|I(z)|\leq \|z\|^2/2$ and the equality holds true if $z \in \R^{2n}$.

\item Given $r>0$, we define the domain $\mathcal{D}_r$ to be the open ball centered at the origin in $\C^n$ of radius $r^2/2$ with respect to the norm $|\,.\,|$:
\begin{equation*}
\mathcal{D}_r:=\{ I \in \C^n \; | \; |I|<r^2/2 \}
\end{equation*}  
and we let $D_r:=\mathcal{D}_r \cap \R^n$. This choice is motivated by the fact that if $I: z \in \C^{2n} \mapsto I(z) \in \C^n$, then $I(\mathcal{B}_r) \subseteq \mathcal{D}_r$ and $I(B_r)=D_r \cap \R^n_+$, where $\mathcal{B}_r$ and $B_r$ have been defined in~\eqref{bouleR}. 

\item We define $\|\,.\,\|_r$ to be the sup norm for functions defined on $\mathcal{B}_r$ or on $\mathcal{D}_r$. Extending the norm $\|\,.\,\|$ initially defined for vectors in $\C^n$ and $\C^{2n}$ (respectively in~\eqref{norme} and in~\eqref{norme2}) to tensors in $\C^n$ and $\C^{2n}$, we extend the sup norm $\|\,.\,\|_r$ for tensor-valued functions defined on $\mathcal{B}_r$ or on $\mathcal{D}_r$. The same notation $\|\,.\,\|_r$ will be used also for the real domains $B_r$ and $D_r$: this will not cause confusion as it will be clear from the context if it is the complex or the real domains that are considered.   

\item We consider a  Hamiltonian $H$  of the form
\begin{equation}\label{Ham}\tag{$*$}
H(z)=h(I(z))+f(z), \quad h : \mathcal{D}_r \rightarrow \C, \quad f : \mathcal{B}_r \rightarrow \C 
\end{equation}
which is real analytic and such that \
\begin{equation}\label{condsteep}
\|\nabla h\|_r \leq E, \quad \|\nabla^2 h\|_r \leq F, \quad \|X_f\|_r \leq \varepsilon  
\end{equation}
where $X_f$ is the Hamiltonian vector field associated to $f$

\item The integrable Hamiltonian $h$ is supposed to be \emph{steep} on the domain $D_r$, as defined below.

\begin{definition}\label{funcsteep} 
A differentiable function $h : D_r \rightarrow \R$ is steep if there exist positive constants $C,\delta,p_l$, for any integer $l\in [1,n-1]$, and $\kappa$ such that for all $I \in D_r$, we have $||\nabla h(I)|| \geq \kappa$ and, for all integer $l\in [1,n-1]$, for all vector space $\Lambda \in \R^n$ of dimension $l$, letting $\lambda=I+\Lambda$ the associated affine subspace passing through $I$ and $h_\lambda$ the restriction of $h$ to $\lambda$, the inequality
\[ \max_{0 \leq \eta \leq \xi}\;\min_{||I'-I||=\eta, \; I' \in \lambda \cap D_r}||\nabla h_\lambda(I')-\nabla h_\lambda(I)||>C\xi^{p_l} \]
holds true for all $0 < \xi \leq\delta$. We say that $h$ is $(r,\kappa,C,\delta,(p_l)_{l=1,\ldots,n-1})$-steep and, if all the $p_i=p$, we say that  $h$ is $(r,\kappa,C,\delta,p)$-steep.
\end{definition}
\end{itemize}
 
Let us point out that the definition of steepness that we use is not exactly the one given by Nekhoroshev but it is obviously equivalent to it (see \cite{Nek73} or \cite{Nek77}). Indeed, Nekhoroshev only requires steepness for subspaces $\Lambda$ which are orthogonal to $\nabla h(I)$, in which case $\nabla h_\lambda(I)=0$; for subspaces $\Lambda$ such that $\nabla h_\lambda(I)\neq 0$, the inequality in Definition~\ref{funcsteep} is clearly satisfied (and one may even set $p_l=0$ in this case).

\begin{Main}\label{mainthm03} 

Let $H(z)=h(I(z))+f(z)$ be as in~\eqref{Ham} satisfying~\eqref{condsteep}, such that $h$ is $(r,\kappa,C,\delta,(p_l)_{l=1,\ldots,n-1})$-steep. Then there exist $\tilde{r}^*, \tilde{c}, \tilde{c}'>0$, which depend only on $n$, $E$, $F$, $\kappa$, $C$ and $p_l$ for $1 \leq l \leq n-1$ such that if
\begin{equation}\label{seuilsteep}
r \leq \tilde{r}^*, \quad  r\varepsilon \leq \tilde{c}\min\left\{\delta^{2na},r^{4na}\right\}
\end{equation}
where
\[  a:=1+p_1+p_1p_2+\cdots+p_1p_2\dots p_{n-1},\]
then for any solution $z(t)$ of the Hamiltonian flow \eqref{Ham} with $z(0)=z_0 \in B_{r/2}$ we have 
\[ |I(z(t))-I(z_0)|\leq \tilde{c}'(r\varepsilon)^{\frac{1}{2na}}, \quad |t| \leq \exp\left((r\varepsilon)^{-\frac{1}{2na}}\right).  \]
\end{Main}

\begin{remark}
{\rm We will prove in fact a slightly more general and more precise statement (but whose formulation is also more cumbersome): there exist positive constants $\tilde{c}_1$, $\tilde{c}_2$, $\tilde{c}_3$, $\tilde{c}_4$, $\tilde{c}_5$, $\tilde{c}_6$ and $\tilde{c}_7$, which depend only on $n,E,F$ and on the steepness constants $\kappa, C,p_l$, for $1 \leq l \leq n-1$, such that for any solution $z(t)$ with $z(0)=z_0 \in B_{r/2}$, if
\begin{equation}\label{seuilsteep2}
r\varepsilon \leq \min\left\{\tilde{c}_1,\tilde{c}_2\delta^{2na},\tilde{c}_3r^{4na},\tilde{c}_4r^{\frac{2a}{a-a'}}\right\}
\end{equation}
where $a$ is as above and 
\[ a':=1+p_2+p_2p_3+\cdots+p_2p_3\dots p_{n-1}, \]
then
\[ |I(z(t))-I(z_0)|\leq \tilde{c}_5(r\varepsilon)^{\frac{1}{2na}}, \quad |t| \leq \tilde{c}_6r^{-1}(r\varepsilon)^{-\frac{1}{2na}} \exp\left(\tilde{c}_7r^{-1}(r\varepsilon)^{-\frac{1}{2na}}\right).  \]
This statement obviously implies the statement of Theorem~\ref{mainthm03}. Let us also add that using this more precise statement, one can easily obtain a more precise statement in Theorem~\ref{infthm}.
}
\end{remark}

\subsection{Comments, open questions and prospects}\label{s22}

It is natural to ask whether our main result, Theorem~\ref{infthm}, can be improved, and so we can ask the following two questions.

\begin{question}
Does Theorem~\ref{infthm} remains true without assuming $\mathrm{BNF}_{K_0}(\tilde{H}_{K_0}) \in SS(n,m_0)$?
\end{question}

\begin{question}
Is the estimate on the time $T(r)$ in Theorem~\ref{infthm} essentially optimal?
\end{question}

A main difficulty in these questions is related to the fact that the construction of an unstable elliptic equilibrium point in the analytic category is  a wide open problem as we emphasized in the Introduction. Concerning the second question, let us just mention that it may be possible to give an answer in the Gevrey category (a regularity which is intermediate between smooth and analytic). Indeed, on the one hand, one should expect that the statement of Theorem~\ref{infthm} holds true for Gevrey Hamiltonians, with only different constants. On the other hand, using the methods in \cite{MS02}, it might be possible to construct an unstable elliptic equilibrium point in the Gevrey category, with a time of stability which is a double exponential (the fact that one can construct a Lyapunov unstable elliptic equilibrium point in the Gevrey category follows directly from \cite{Dou88}, but the real difficulty is to get an estimate on the time of instability).

Then it is also natural to ask whether our result holds true for a quasi-periodic invariant Lagrangian torus, or more generally, for a quasi-periodic normally elliptic and reducible invariant torus (which includes both elliptic equilibrium points and quasi-periodic invariant Lagrangian tori as particular cases). This general case is described by a Hamiltonian of the form
\[ H(\theta,J,x,y)=\beta\cdot J+\alpha \cdot I(x,y)+F(\theta,I,x,y) \] 
where $(\theta,J) \in \T^m \times \R^m$ are action-angle coordinates, $(x,y)$ symplectic coordinates around the origin in $\R^{2n}$ and $F$ is at least of order $2$ in $I$ and $3$ in $(x,y)$. The set $\{(J,x,y)=0\}=\{(J,x,y) \; | \; J=0,\,I(x,y)=0\}$ is a normally elliptic torus of dimension $n$ in a $n+m$ degrees of freedom Hamiltonian, and the question is as follows.

\begin{question}
Assuming that the vector $(\beta,\alpha) \in \R^{m+n}$ is Diophantine and $H$ is real-analytic, does Theorem~\ref{infthm} extends to this setting in the following sense: if $(J(0),I(x(0),y(0)))$ is at a distance $r$ of zero in $\R^{n+m}$, with $r$ sufficiently small, is it true that $(J(t),I(x(t),y(t)))$ stays at a distance $2r$ from $0$ for a time $T(r)$ which is doubly exponentially large with respect to $r^{-1/(\tau+1)}$ (where $\tau$ is the exponent of the Diophantine condition on the vector $(\beta,\alpha)$)? 
\end{question}

In a subsequent paper (\cite{BFN15}), we will answer positively the above question in the case of an invariant Lagrangian Diophantine torus that is of particular interest in the study of perturbed integrable systems. Indeed, by KAM theory, it is well-known that invariant Lagrangian Diophantine tori appear for arbitrary small perturbations of generic integrable Hamiltonian systems in action-angle coordinates. Furthermore, these tori are not isolated and appear as a family parametrized by some Cantor set of positive Lebesgue measure (tending to full measure as the size of the perturbation goes to zero). The goal of \cite{BFN15} is to prove that under an additional generic assumption on the integrable Hamiltonian, most of the KAM tori are doubly exponentially stable.

\section{Genericity of steepness and Birkhoff normal forms} \label{sec.steep}

The aim of this section is to give a proof of Theorem \ref{mainthm02} and of the fact that Theorem \ref{mainthm03} implies Theorem \ref{mainthm01}.   
 
\subsection{Genericity of steepness} \label{subsec.steep.generic}

In Appendix \ref{app.steep} we will prove a general result on genericity of stably steep polynomials.

\begin{theorem}\label{proputile2} 
The complement of $SS(n,m_0)$ in $P_2(n,m_0)$ is  contained in a  semi-algebraic subset $ \Upsilon(n,m_0)$ of codimension at least one.  
\end{theorem}

Theorem \ref{proputile2} has an immediate consequence on the genericity of  steep functions as will be shown in the following Theorem \ref{th.taylor}. 

Given $p \in \N$, $p \geq 3$ and $\rho>0$, let $\mathcal{C}^{p}(D_\rho)$ be the set of functions $p$ times continuously differentiable on $D_\rho$, and let
$$\vert\vert \nabla^2 h\vert\vert_{{p,\rho}}=\max_{2 \leq j \leq p} \|\nabla^j h\|_\rho < \infty$$
where $\|\,.\,\|_\rho$ is the sup norm on $D_\rho$ of the tensor-valued function $\nabla^j h$, and where we recall that by definition, $D_\rho$ is the (real) open ball of radius $\rho^2/2$ with respect to the sup norm $|\,.\,|$. Given $h\in\mathcal{C}^{p}(D_\rho)$, we denote by $T_{p-1} h(I) \in P(n,p-1)$ the Taylor expansion of $h$ of order $p-1$ at $I \in D_\rho$ (or the $p-1$-jet at $I$). We have the following statement (that will be used later with the value $p=m_0+1$).

\begin{theorem}  \label{th.taylor} Let $h\in\mathcal{C}^{p}(D_\rho)$ be such that $\|\nabla h(0)\|:=\varpi$ and $P_{p-1}:=T_{p-1} h(0) -T_1 h(0)-T_0h(0) \in SS(n,p-1)$. Then, there exists positive numbers $\mu^*,\delta^*$ and $C$ that depend only on $\varpi$, $P_{p-1}$, $\vert\vert \nabla^2 h\vert\vert_{p,\rho}$ and $n$ such that $h$ is $(\mu,\kappa,C,\delta,p-2)$-steep, with
\[ \mu:=\min\{\rho/2,\mu^*\}, \quad \kappa:=\varpi/2, \quad \delta:=\min\{\rho^2/4,\delta^*\}. \]
\end{theorem}

\begin{proof} [Proof of Theorem \ref{th.taylor}] Let $M:=\vert\vert \nabla^2 h\vert\vert_{p,\rho}$. Observe first that if $\mu^2\leq \varpi/M$ then the condition $\|\nabla h(I)\| \geq \kappa=\varpi/2$ is satisfied for any $I \in D_{\mu}$.     

Fix an arbitrary $I \in D_\mu$, and define $\mathcal{H}_{I}= T_{p-1} h(I)-T_1 h(I)-T_0 h(I) \in P_2(n,p-1)$. Since $\mathcal{H}_{0}=T_{p-1} h(0)-T_1 h(0)-T_0 h(I)=P_{p-1}$ is stably steep, we have the existence of $\tilde{\mu}$ that depends on $M$, $P_{p-1}$, and $n$ such that if $\mu \leq \tilde{\mu}$, $\mathcal{H}_{I}$ is sufficiently close to $P_{p-1}$ so that for all integer $l \in [1,n-1]$, for all vector subspace $\Lambda \subseteq \R^n$ of dimension $l$, letting $\mathcal{H}_{I,\Lambda}$ be the restriction of $\mathcal{H}_{I}$ to $\Lambda$, the inequality
\[ \max_{0 \leq \eta \leq \xi}\;\min_{||x||=\eta, \; x \in \Lambda}||\nabla \mathcal{H}_{I,\Lambda}(x)||>C_0\xi^{p-2} \]
holds true for all $0 < \xi \leq\delta_0$, where $\delta_0$ and $C_0$ are the steepness constant related to $P_{p-1}$. 

Now, we get by the Taylor formula (applied to $\nabla h$ at the order $p-1$) that 
$$ \|\nabla h(I+x) - \nabla h(I)-\nabla \mathcal{H}_{I}(x)\| \leq  M(p-1)! \|x\|^{p-1} $$
provided $I+x \in D_\rho$, which is satisfied if $\mu \leq \rho/2$ and $|x|\leq\|x\| \leq \rho^2/4$. So for $\|x\|\leq \xi \leq\delta$, with $\delta:=\min\{C_0(2M(p-1)!)^{-1},\rho^2/4\}$, we have
$$ \|\nabla h(I+x) - \nabla h(I)-\nabla \mathcal{H}_{I}(x)\| \leq (C_0/2)\xi^{p-2}$$
and then, letting $\lambda=I+\Lambda$,
$$ \|\nabla h_\lambda(I+x) - \nabla h_\lambda(I)-\nabla \mathcal{H}_{I,\Lambda}(x)\| \leq (C_0/2)\xi^{p-2}.$$
From this we eventually obtain
\[ \max_{0 \leq \eta \leq \xi}\;\min_{||x||=\eta, \; x \in \Lambda}||\nabla h_\lambda(I+x) - \nabla h_\lambda(I)||>(C_0/2)\xi^{p-2} \] 
and letting $I'=I+x$, $C:=C_0/2$, $\delta^*:=C_0(2M(p-1)!)^{-1}$ and $\mu^*:=\min\{\tilde{\mu},\sqrt{\varpi/M}\}$, the steepness of $f$ is thus established with the constants given in the statement. \end{proof}

\subsection{Generic steepness of the BNF. } \label{subsec.steep.BNF}

The proof of Theorem~\ref{mainthm02} will be an easy consequence of Theorem \ref{proputile2} and the following two lemmas on the map $\mathrm{BNF}_K$.  

\begin{lemma}\label{lemme1}
The map $\mathrm{BNF}_K$ is algebraic.
\end{lemma}

\begin{proof} This follows by construction of the Birkhoff normal form, and we refer to \cite{PM03} for more details. \end{proof} 

Now given a polynomial $Q=Q_2+\cdots+Q_m \in P_2(n,m)$, where each $Q_j$ is homogeneous of degree $j$, it can be identified to a polynomial $\tilde{Q} \in P_3(2n,K)$ by setting $\tilde{Q}(\xi):=Q(I(\xi))$. For $K \geq 4$, we can define a map by
\[ \begin{array}{lll}
F^K : & P_2(n,m) &\longrightarrow P_2(n,m) \\
     & Q &\longmapsto \mathrm{BNF}_K(\tilde{H}_K+\tilde{Q}).
\end{array} \]

\begin{lemma}\label{lemme2}
The map $F^K$ preserves Lebesgue measure.
\end{lemma}

\begin{proof} This also follows by construction of the Birkhoff normal form. More precisely, it can be shown that decomposing the map $F^K$ as $F^K=(F_2^K,\dots,F_m^K)$, where $F_j^K$ is the component with respect to homogeneous polynomials of degree $j$, then we have $F_2^K(Q)=Q_2+\mathrm{BNF}^4(\tilde{H}_4)=Q_2+h^2$, and for $3 \leq j \leq K$, we have $F_j^K(Q)=Q_j+\mathcal{F}_j^K(\tilde{H}_{2j},Q_2,\dots,Q_{j-1})$ where $\mathcal{F}_j^K$ is an algebraic map (see \cite{Nie13}, where this property has already been used). This expression clearly implies that $F^K$ is smooth with Jacobian one, therefore it preserves Lebesgue measure. \end{proof} 

\begin{proof}[Proof of Theorem~\ref{mainthm02}]
Our aim is to show that the complement of $\mathrm{BNF}_{K_0}^{-1}(SS(n,m_0))$ in $P_3(2n,K_0)$ is contained in a semi-algebraic subset of positive codimension. Since the inverse image of a semi-algebraic subset by an algebraic map is semi-algebraic, from Theorem~\ref{proputile2} and Lemma~\ref{lemme1}, it follows that the complement of $\mathrm{BNF}_{K_0}^{-1}(SS(n,m_0))$ in $P_3(2n,K_0)$ is contained in a semi-algebraic subset. It remains to prove that this set has positive codimension, or equivalently, zero Lebesgue measure in $P_3(2n,K_0)$. By Lemma~\ref{lemme1}, for any $\tilde{H}_{K_0} \in P_3(2n,K_0)$, the Lebesgue measure in $P_2(n,m_0)$ of the set
\[ \{ Q \in P_2(n,m_0) \; | \; \tilde{H}_{K_0}+\tilde{Q} \notin  \mathrm{BNF}_{K_0}^{-1}(SS(n,m_0)) \} \] 
is zero. By Fubini-Tonelli theorem, this implies that the complement of $\mathrm{BNF}_{K_0}^{-1}(SS(n,m_0))$ in $P_3(2n,K_0)$ has zero Lebesgue measure, and this concludes the proof.  
\end{proof}

\subsection{Birkhoff normal forms with estimates} \label{sec.bir}

For a real analytic Hamiltonian with an elliptic equilibrium point, as in \eqref{Hamintro}, it is known that the estimates on the Birkhoff normal form are given by the arithmetic properties of $\a$ and the analytic norm of $H$.  We summarize in the following Proposition \ref{propBirkhoffutile} the estimates on the BNF that will be useful for us in the sequel. The proof of Proposition \ref{propBirkhoffutile} is relatively standard, we include it in Appendix \ref{s3} following \cite{DG96b}. 

Here it will be more convenient to perform a linear change of complex canonical coordinates $z=S(\xi)$, where $S : \C^{2n} \rightarrow \C^{2n}$ is defined by
\[ z_j=\frac{1}{\sqrt{2}}(\xi_j+i\xi_{n+j}), \quad z_{n+j}=\frac{i}{\sqrt{2}}(\xi_j-i\xi_{n+j}). \]
It is easy to check that this linear transformation $S$ and its inverse $S^{-1}$ have unit norm (with respect to the norm $||\,.\,||$ defined in~\eqref{norme}), hence $H$ and $H \circ S$ have the same radius of convergence around the origin and $||H \circ S||_R=||H||_R$. Abusing notations, we will still write $H$ instead of $H \circ S$ to denote the Hamiltonian in these new coordinates. Observe that
\[ H_2(\xi)=h_1(I(\xi))=\alpha\cdot I(\xi)= i\sum_{j=1}^n \alpha_j \xi_j \xi_{n+j} \]
where
\[ I(\xi)=(I_1(\xi),\dots,I_n(\xi)), \quad I_j(\xi)=i\xi_j\xi_{n+j}, \quad 1 \leq j \leq n. \]

Recall the definition of $\Psi_\a$ given in  Section \ref{s21} and define also for any integer $j \geq 3$, 
\[ \psi_\alpha^j:=\prod_{i=3}^{j}\Psi_\alpha(i). \]

For $K \geq 1$, define 
\begin{equation}\label{defrK}
\rho_K:=(548ncd K\Psi(K))^{-1},
\end{equation}
where the positive constants $c$ and $d$ depend only on $n$, $R$ and $||H||_R$ and are defined in~\eqref{cd}.

\begin{proposition}\label{propBirkhoffutile}
Let $H$ be as in \eqref{Hamintro} with 
$\a$ as in \eqref{fonctionpsi}, and fix integers $p \geq 2$, $K \geq 2p$ and $0 \leq q \leq K-4$. There exist constants $b(p)$ and  $\tilde{b}(q)$ that depend respectively on $p,n,R,\|H\|_R, \psi_\a^{2p-1}$ and on $q,n,R,\|H\|_R, \psi_\a^{q+2}$ such that if we assume
\begin{equation}\label{threshold}
0 < \rho \leq \rho_K/e,
\end{equation}
then  there exists a real-analytic symplectic transformation $\Phi^K={\rm Id} +O(\xi^{2})$ defined on $\mathcal{B}_{\rho_K}$ such that  
\[ H \circ \Phi^K(\xi)=\alpha \cdot I(\xi)+h^m(I(\xi))+f^K(\xi),\]
with $f^K=O(\xi^{K+1})$ and the following estimates hold
\begin{equation}\label{estimhK}
||\nabla^2 h^m||_{p,\rho}=\max_{2 \leq j \leq p}||\nabla^j h^m||_{\rho} \leq b(p)
\end{equation}
\begin{equation}\label{estimfK}
||\nabla f^K||_\rho \leq \tilde{b}(q)\rho^q e^{-K}.
\end{equation}
\end{proposition}

\subsection{From Nekhoroshev stability to double exponential stability} \label{nekhoimpliesstability}

In this section we prove that Theorem \ref{mainthm03} implies Theorem \ref{mainthm01}. As a corollary of Proposition \ref{propBirkhoffutile} and Theorem \ref{th.taylor} we get the following

\begin{proposition}\label{proputile1}
Let $H$ be as in~\eqref{Hamintro}  with $\a$ as in \eqref{fonctionpsi}, and such that 
\[ \mathrm{BNF}_{K_0}(\tilde{H}_{K_0})=h^{m_0} \in SS(n,m_0).\]  
There exists $C>0$ and $K^* \geq 4$ that depend only on $n$, $R$, $\|H\|_R$, $h_{m_0}$, $||\alpha||$ and $\psi_\alpha^{2m_0+1}$ such that if 
\[ K \geq K^*, \quad 0<\rho\leq \rho_K/e,\] 
then there exists a real-analytic symplectic transformation $\Phi^K={\rm Id} +O(\xi^{2})$ defined on $\mathcal{B}_{\rho_K}$ such that
\begin{equation}
H \circ \Phi^K(\xi)=\alpha\cdot I(\xi)+h^m(I(\xi))+f^K(\xi):=h(I(\xi))+f^K(\xi) \label{eBNF} 
\end{equation}
with $f^K=O(\xi^{K+1})$ and
\begin{equation}\label{estimhK2}
||\nabla^2 h^m||_{m_0+1,\rho}=\max_{2 \leq j \leq m_0+1}||\nabla^j h^m||_{\rho} \leq b(m_0+1)
\end{equation}
\begin{equation}\label{estimfK2}
||\nabla f^K||_\rho \leq \tilde{b}(q)\rho^q e^{-K},\quad 0 \leq q \leq K-4,
\end{equation}
and such that $h$ is $(\rho/2,||\alpha||/2,C,\rho^2/4,m_0-1)$-steep. 
\end{proposition}

\begin{proof}[Proof of Proposition~\ref{proputile1}] For $K\geq K^*\geq 2(m_0+1)$ apply Proposition \ref{propBirkhoffutile} with $p=m_0+1$ and $q \leq K-4$ and get \eqref{eBNF} with estimates~\eqref{estimhK2} and~\eqref{estimfK2}. We want to apply Theorem~\ref{th.taylor}  with $p=m_0+1$ and $\varpi=||\alpha||$. Observe first that $T_{m_0} h(0) -T_1 h(0)-T_0 h(0)=h^{m_0} \in SS(n,m_0)$. Then observe also that $\nabla^2 h=\nabla^2 h^m$ and that by~\eqref{estimhK2}, we have the bound    
\[ ||\nabla^2 h||_{m_0+1,\rho}=||\nabla^2 h^m||_{m_0+1,\rho} \leq b(m_0+1) \]
which is independent of $\rho$, hence the constants $C$, $\mu^*$ and $\delta^*$  in the statement of Theorem~\ref{th.taylor} do not depend on $\rho$,
and choosing $K^*$ sufficiently large, $\rho_K$ and then $\rho$ become sufficiently small so that $\rho/2 \leq \mu^*$ and $\rho^2/4 \leq \delta^*$ 
therefore $h$ is $(\rho/2,||\alpha||/2,C,\rho^2/4,m_0-1)$-steep.   
\end{proof}

We now use Proposition~\ref{proputile1} and Theorem~\ref{mainthm03}  to give the 

\begin{proof}[Proof of Theorem~\ref{mainthm01}] Let $H$ be as in~\eqref{Hamintro}  with $\a$ as in \eqref{fonctionpsi} and
\[ \mathrm{BNF}_{K_0}(\tilde{H}_{K_0})=h^{m_0} \in SS(n,m_0).\]  

For $r>0$ we define 
\begin{equation*}
K=\Delta_\alpha((1644encdr)^{-1})
\end{equation*} 
so that $\rho_K/e \geq 3r$, and observe that $K\geq K^*$ is satisfied (with $K^*$ given by Proposition~\ref{proputile1}) provided $r\leq r^*$ for some sufficiently small $r^* >0$. Hence we can apply the latter proposition with our choice of $K$ and with $\rho=3r$. 

Next we want to apply Theorem~\ref{mainthm03} to \eqref{eBNF}. First observe that Theorem~\ref{mainthm03} is stated and proved in the $z$ variables whereas the estimate of Proposition~\ref{proputile1} are given in the $\xi$ variables: however since $z=S(\xi)$ with $S$ and $S^{-1}$ of unit norm, Theorem~\ref{mainthm03} also holds true, with the same estimates, if one uses the $\xi$ variables.

From Proposition~\ref{proputile1} and our choice of $\rho$, the function $h$ is $(3r/2,||\alpha||/2,C,9r^2/4,m_0-1)$-steep and~\eqref{condsteep} is satisfied with
\[ E:=3/2||\alpha||, \quad F:=b(2), \quad \varepsilon:=\tilde{b}(q)\rho^qe^{-K} \]
for some $0 \leq q \leq K-4$ yet to be chosen. Up to taking $r^*$ smaller one easily checks that~\eqref{seuilsteep} (with $r$ replaced by $\rho=3r$)  is satisfied provided we choose $q=4na-1$. Thus Theorem~\ref{mainthm03} can be applied and we obtain the following statement: given an arbitrary solution $\tilde{\xi}(t)$ of the system associated to $H \circ \Phi^K$ in~\eqref{eBNF}, if $||\tilde{\xi}(0)||=||\tilde{z}(0)||\leq \rho/2=3r/2$, then 
\[ |I(\tilde{z}(t))-I(\tilde{z}(0))|\leq \tilde{c}'(\rho\varepsilon)^{\frac{1}{2na}}, \quad |t| \leq \exp\left((\rho\varepsilon)^{-\frac{1}{2na}}\right).  \]
For $r$ sufficiently small, this implies in particular that $||\tilde{\xi}(t)||=||\tilde{z}(t)||< 7r/4$ for times
\[ |t|\leq \exp\left((3r\varepsilon)^{-\frac{1}{2na}}\right).  \] 
Recalling the definition of $\varepsilon$ and with our choices of $q$ and $K$, the previous estimate implies that $||\tilde{\xi}(t)||=||\tilde{z}(t)||< 7r/4$ for times
\[ |t| \leq \exp( c r^{-2} \exp(c' \Delta_\a(c'' r^{-1})))\]with
\begin{equation}  c:=3^{-2}\tilde{b}(4na-1)^{-\frac{1}{2na}}, \quad c':=(2na)^{-1}, \quad c'':=1644encd.  \label{ccc} \end{equation}

To conclude, observe that $H$ is related to~\eqref{eBNF} by a symplectic transformation $\Phi^K={\rm Id}+O(\xi^2)$, which can be made close enough to the identity (as well as its inverse) by taking $r$ small enough. Hence, given any solution $\xi(t)$ of the system associated to $H$ with $||\xi(t)||=||z(t)||\leq r$, the corresponding solution $\tilde{\xi}(t)$ of~\eqref{eBNF} satisfy $||\tilde{\xi}(0)||=||\tilde{z}(0)||\leq 3r/2$ for $r$ small enough, and therefore $||\tilde{\xi}(t)||=||\tilde{z}(t)||< 7r/4$, and also $||\xi(t)||=||z(t)||< 2r$, for times
$$|t| \leq \exp( c r^{-2} \exp(c' \Delta_\a(c'' r^{-1}))).$$   
We eventually arrives at the estimate
$$T(r) \geq \exp( c r^{-2} \exp(c' \Delta_\a(c'' r^{-1})))$$
and this concludes the proof of the lower bound on $T(r)$ in the general case.  The estimate in the Diophantine case follows from the general case and from \eqref{deltadioph}. 
\end{proof}

\section{Nekhoroshev exponential stability for an elliptic equilibrium} \label{sec.stability}

The goal of this section is to prove Theorem~\ref{mainthm03}, following the method introduced in \cite{BN12} and \cite{Bou11cmp}. This method, which uses only periodic approximations and compositions of periodic averagings, has the advantage of being directly applicable in a neighborhood of an elliptic equilibrium point where action-angle coordinates cannot be used.
 
Since the proof contains some technical statements, we first give in the next Section \ref{s50} a long and complete heuristic description of the method that would hopefully make the reading of the proof easier. We emphasize that Section \ref{s50} is included only for the convenience of the reader and does not interfere with the proof strictly speaking.

\subsection{Heuristic description and plan of the proof}\label{s50}

Given an arbitrary initial condition $z_0$ and the associated solution $z(t)$ (that is $z(0)=z_0$) of the Hamiltonian $H=h+f$, our goal is to prove that the variation of the action $I(z(t))-I(z_0)$ remains small (as a small power of $\epsilon$) for an interval of time which is exponentially large with respect to the inverse of (some power of) $\varepsilon$, $\varepsilon$ being the size of the perturbation $f$. The proof is based on an algorithm that, for $0 \leq j \leq n-1$, reduces to a space of dimension $n-j-1$ the directions in which a fast drift (before an exponentially long interval of time) may be possible at each step $j$, and that stops therefore after at most $j=n-1$ steps. We now describe the heuristics of this algorithm, which depends on a positive parameter $Q\geq 1$ and an integer parameter $m \geq 1$.

For the step $j=0$, we write $H=H_0$ and given the parameter $Q\geq 1$, we use Dirichlet's box principle to approximate the unperturbed frequency $v_0=\nabla h(I(z_0))$ by a periodic vector $\omega_0$, that is, a vector which is a real multiple of an integer vector (this corresponds to a frequency vector which is maximally resonant, as the set of integer vectors $k$ orthogonal to $\omega_0$ forms a sub-module of maximal rank $n-1$). Letting $T_0$ be the period of $\omega_0$, which is the smallest positive number $t$ such that $t\omega_0 \in \Z^n$, the parameter $Q$ controls the approximation as follows:
\[ ||v_0-\omega_0||=s_0 \lesssim (T_0Q)^{-1}, \quad ||v_0||^{-1} \lesssim T_0 \lesssim ||v_0||^{-1} Q^{n-1}. \]
Then, on some small neighborhood $V_0$ of $z_0$, given the integer parameter $m \geq 1$ and assuming certain compatibility conditions between $m, s_0, T_0$ and $\varepsilon$, it is possible to construct a resonant normal form (with respect to $\omega_0$) up to a remainder which is exponentially small in $m$: more precisely, by a symplectic transformation $\Phi_0$ which is close to the identity, the transformed Hamiltonian $H_0 \circ \Phi_0$ can be written as a perturbation of $h$, but this time the perturbation splits into two parts: a resonant part, which is still of order $\varepsilon$ but has the additional property that its Hamiltonian flow commutes with the linear flow of frequency $\omega_0$, and a non-resonant part which is of order $2^{-m}\varepsilon$. As a consequence, we have the following partial stability result for the solution of $H_0 \circ \Phi_0$ starting at $\Phi_0^{-1}(z_0)$: the variation of the action variables in the (one-dimensional) direction given by $\omega_0$ is small for an exponentially long interval of time of order $2^m$, unless the solution escapes from the domain of the normal form $V_0$ before. In other words, we excluded at this step $j=0$ the direction $\omega_0$ from the directions along which a drift in the actions may appear before an exponentially long interval of time. Since $\Phi_0$ is symplectic and close to the identity, the same holds true for the solution $z(t)=z_0(t)$ of $H=H_0$, and we arrive at the following dichotomy: either the action variables have also small variation in the direction transverse to $\omega_0$ for an exponentially long interval of time, or not. 

In the first case, for an exponentially long interval of time with respect to $m$, the variation of the action variables is small: the stability condition is satisfied and the algorithm stops. Once the algorithm stops, one can determine the parameters $Q$ and $m$ in order to fulfill the compatibility conditions which essentially read as follows: 
\begin{equation*}
s_0 \lesssim 1, \quad mT_0\varepsilon \lesssim s_0, \quad mT_0s_0 \lesssim 1.
\end{equation*}
Since $s_0 \sim (T_0Q)^{-1}$ and $T_0 \lesssim Q^{n-1}$ (as $||\nabla h(I(z_0))||=||v_0||$ is of order one), these conditions are satisfied if we choose $m \sim Q \sim \varepsilon^{-\frac{1}{2n}}$ and $\varepsilon \lesssim 1$, and thus we obtain a result of exponential stability with respect to $m \sim \varepsilon^{-\frac{1}{2n}}$. If $h$ is convex or quasi-convex it is simple to see, due to energy conservation, that this first case is automatic and it is not possible that the action variables drifts transversely to $\omega_0$, hence exponential stability is proved with one step of the algorithm. But in general, the second case is possible and more work is needed to further reduce the drifting possibilities of the actions going from the step $j=0$ to the step $j=1$ of the algorithm.  

In the second case, setting $H_0^+=H_0 \circ \Phi_0$ and denoting by $z_0^+(t)$ the associated solution, we can find a positive time $\tilde{t}_0^+$, which is shorter than $2^{m}$, such that the $I(z_0^+(\tilde{t}_0^+))-I(z_0^+(0))$ has a small drift of order $s_0$ in the direction orthogonal to $\omega_0$. Letting $\Pi_0$ be the projection onto the orthogonal of $\omega_0$, we can define a curve 
\[ \gamma_0(t)=I(z_0^+(0))+\Pi_0(I(z_0^+(t))-I(z_0^+(0)))\] 
which takes values in an affine subspace of dimension $n-1$. One can then exploit the steepness property to find a time $\tilde{t}_0\leq t_0^+$ for which the vector $\nabla h(\gamma_0(\tilde{t}_0))$ is linearly independent from $\omega_0$ in a quantitative way: 
\[ \Pi_0(\nabla h(\gamma_0(\tilde{t}_0))) \gtrsim s_0^{p_{n-1}}\] 
where $p_{n-1}$ is the steepness index in dimension $n-1$. Using Dirichlet's box principle again with the same parameter $Q$, we can approximate the vector $\Pi_0(\nabla h(\gamma_0(\tilde{t}_0)))=v_1$ by another periodic vector $\omega_1$:
\[ ||v_1-\omega_1||=s_1 \lesssim (T_1Q)^{-1}, \quad ||v_1||^{-1} \lesssim T_1 \lesssim ||v_1||^{-1} Q^{n-1} \lesssim s_0^{-p_{n-1}}Q^{n-1}.  \]
First observe that since $v_1$ is orthogonal to $\omega_0$, $\omega_1$ is almost orthogonal to $\omega_0$ and in particular it is linearly independent from $\omega_0$. Then since $||\gamma_0(\tilde{t}_0)-I(z_0^+(\tilde{t}_0))||=||\Pi_0^\perp(I(z_0^+(\tilde{t}_0))-I(z_0^+(0)))||$ is small (as $I(z_0^+(t)-I(z_0^+(0))$ has only small variation in the direction given by $\omega_0$), we also have
\[ ||\Pi_0(\nabla h(I(z_0^+(\tilde{t}_0))))-\omega_1||\sim s_1 \lesssim (T_1Q)^{-1}, \quad  T_1 \lesssim s_0^{-p_{n-1}}Q^{n-1}.  \]
Set $z_1=z_0^+(\tilde{t}_0)$ and $H_1=H_0^+$. On a small neighborhood $V_1$ of $z_1$ (small enough so that $V_1$ is still included in $V_0$) we can then construct, as in the first step, a resonant normal form with respect to $\omega_1$ up to an exponentially small remainder with respect to $m$. Unlike the step $j=0$ in which the perturbation was arbitrary, here the perturbation is given by the non-resonant part with respect to $\omega_0$, and this explains why it is sufficient to have an approximation of $\Pi_0(\nabla h(I(z_1)))$ and not of the full vector $\nabla h(I(z_1))$. Moreover, a careful construction of the new normalizing transformation $\Phi_1$ shows that the resonant part of $H_1 \circ \Phi_1$, whose flow commutes with the linear flow of frequency $\omega_1$, also commutes with the linear flow of frequency $\omega_0$. Exactly like in the first step, we arrive at a dichotomy which determines whether the algorithm stops (the variation of the action of the solution of $H_1 \circ \Phi_1$, and then of $H_1=H_0 \circ \Phi_0$ and $H_0=H$, is small and the theorem is proved) or moves to the next step with the gain that now both the directions $\omega_0$ and $\omega_1$, that are linearly independent, are  excluded from the directions along which a drift in the actions may appear before an exponentially long interval of time. Note that if the algorithm stops, the parameters $Q$ and $m$ have to be chosen according to
\begin{equation*}
s_i \lesssim 1, \quad mT_i\varepsilon \lesssim s_i, \quad mT_is_i \lesssim 1, \quad 0 \leq i \leq 1.
\end{equation*}
Observe that
\[ T_1 \lesssim s_0^{-p_{n-1}}Q^{n-1}\sim (T_0Q)^{p_{n-1}}Q^{n-1} \lesssim Q^{n(1+p_{n-1})-1} \]
and hence the compatibility conditions are satisfied if we choose $m \sim Q \sim \varepsilon^{-\frac{1}{2na_1}}$ and $\varepsilon \lesssim 1$, with $a_1=1+p_{n-1}$, and thus we obtain a result of exponential stability with respect to $m \sim \varepsilon^{-\frac{1}{2na_1}}$. In particular, as $2na_1>2n$, this stability result also holds true if the algorithm stopped at step $j=0$.

We have briefly explained how to pass from the step $j=0$ to the step $j=1$, and how the parameters $Q$ and $m$ are chosen if the algorithm stops at step $j=0$ or $j=1$. But of course, for any $0 \leq j \leq n-2$, one proceeds exactly the same way to go from step $j$ to step $j+1$.  The fact that the algorithm stops after $n$ steps  (if, of course, it didn't stop before) is clear since then $n$ linearly independent directions are  excluded from the directions along which a drift in the actions may appear before an exponentially long interval of time. More formally, in the case $j=n-1$, the resonant part in the normal form $H_{n-1}^+=H_{n-1} \circ \Phi_{n-1}$ consists of a Hamiltonian whose flow commutes with $n$ linearly independent linear flows with frequency $\omega_0, \dots, \omega_{n-1}$; it is plain to see that such a Hamiltonian is integrable, so $H_{n-1}^+$ consist of an exponentially small perturbation of some integrable Hamiltonian: the first case of the dichotomy thus holds, the algorithm stops and the theorem is proved. At each step $j$, the compatibility conditions are given by
\begin{equation*}
s_i \lesssim 1, \quad mT_i\varepsilon \lesssim s_i, \quad mT_is_i \lesssim 1, \quad 0 \leq i \leq j,
\end{equation*}
and using the fact that
\[ T_0 \lesssim Q^{n-1}, \quad T_j \lesssim (T_{j-1}Q)^{p_{n-j}}Q^{n-1}, \quad 1 \leq j \leq n-1, \]
we can choose, if the algorithm stops at step $j$, $m \sim Q \sim \varepsilon^{-\frac{1}{2na_{j}}}$ and $\varepsilon \lesssim 1$, with \[ a_0=1, \quad a_1=1+p_{n-1}, \quad a_j=1+p_{n-j}+\cdots+ p_{n-j}\dots p_{n-1}, \quad j \geq 2, \]
leading to a result of exponential stability with exponent $2na_j$. We have 
\[ a_{n-1}=1+p_1+p_1p_2+\cdots+p_1p_2\dots p_{n-1}=a,\] 
and since $2na_{n-1}>2na_{n-2}>\cdots >2na_0$, independently of the step $j$ at which the algorithm stops (and hence independently of the choice of $Q$ and $m$), we obtain a result of exponential stability with exponent $2na$. 

Let us now describe the plan of the proof. In \S\ref{s51}, our aim will be to obtain a suitable normal form (Proposition~\ref{normal}) for an abstract Hamiltonian $H_j$, where $0 \leq j \leq n-1$, which, as we explained in the heuristic description above, will be later related to our original Hamiltonian $H$ described in~\eqref{Ham} in the following way: $H_0=H$ and for $j \geq 1$, $H_j=H_{j-1} \circ \Phi_{j-1}=H_0 \circ \Phi_{0}\circ \cdots \circ \Phi_{j-1}$ where $\Phi_i$, for $0 \leq i \leq j-1$, is the normalizing transformation with respect to the periodic frequency $\omega_i$. Here the periodic frequencies $\omega_0, \dots, \omega_j$ are assumed to be known, $H_j$ is already normalized with respect to $\omega_0, \dots, \omega_{j-1}$ and our aim is to explain the construction of the transformation $\Phi_j$ which will further normalize $H_j$ with respect to $\omega_j$. The proof being technical, details will be given in Appendix~\ref{app}. In \S\ref{s52}, a partial stability result (in the direction given by the linear span of $\omega_0,\dots,\omega_{j}$, up to times of order $2^m$) will be easily deduced from the normal form Hamiltonian $H_j \circ \Phi_j$. Using this, we will introduce a first version of the algorithm dichotomy in Proposition~\ref{norm}: either stability holds for an exponentially long interval of time and  the algorithm stops, or a small drift in the action does appear in the orthocomplement of  $\omega_0,\dots,\omega_{j}$ and the algorithm should move to the next step. We will also prove in Proposition~\ref{norm} that if  $j=n-1$, then only the first alternative can be true. In \S\ref{s53} and \S\ref{s54}, we will examine the situation where the second alternative holds true (so necessarily $j \leq n-2$), and using the steepness property of $h$ (in \S\ref{s53}) and then Dirichlet's box principle (in \S\ref{s54}) we will prove how to pass from the step $j$ to the step $j+1$. In \S\ref{s55}, we summarize the work done in \S\ref{s51}, \S\ref{s52}, \S\ref{s53} and \S\ref{s54} in Proposition~\ref{algostep} that gives one step of the algorithm, and which clearly shows that at some given step, either the algorithm stops or it yields the hypotheses that allow to apply it again. Finally, in \S\ref{s56} we conclude the proof of Theorem~\ref{mainthm03}, which will follow easily from Proposition~\ref{algostep} by determining the parameters $Q\geq 1$ and $m\geq 1$ in terms of our small parameter $\varepsilon$.

\subsection{Normal form statement}\label{s51}

In this section we fix $0 \leq j \leq n-1$, and we assume the existence of periodic vectors $\omega_0, \dots, \omega_j$, with periods $T_0, \dots T_j$, which are linearly independent. For convenience we set $\omega_{-1}=0 \in \R^n$. We define the complex and real vector space
\[ \tilde{\Lambda}_j := \{ v \in \C^n \; | \; v \cdot \omega_{-1}=v \cdot \omega_{0}=\cdots=v \cdot \omega_{j-1}=0\}, \quad \Lambda_j:=\tilde{\Lambda}_j \cap \R^n, \]
which are of complex (respectively real) dimension $n-j$. Then we consider three positive real numbers $r_j$, $s_j$ and $\xi_j$, a point $z_j \in B_{r_j}$ and we define the complex domain
\begin{equation}\label{domj}
\mathcal{V}_{3s_j,3\xi_j}(z_j):=\{ z \in \C^{2n} \; | \; I(z)-I(z_j) \in \tilde{\Lambda}_j, \;  |I(z)-I(z_j)| < 3s_j, \; ||z||< r_j+3\xi_j \}
\end{equation}
where, for simplicity, the dependence on $r_j$ is omitted. We will simply write $||\,\cdot\,||_{3s_j,3\xi_j}$ for the supremum norm for vector fields defined on $\mathcal{V}_{3s_j,3\xi_j}(z_j)$ and, for $-1 \leq i \leq n-1$, we will denote by $l_{\omega_i}(z):=\omega_i\cdot I(z)$ and $X_{\omega_i}$ its associated Hamiltonian vector field. With our convention, the function $l_{\omega_{-1}}$ and its associated Hamiltonian vector field are identically zero.

Given a real number $0<\varepsilon<1$ and an integer $m\geq 1$, we can define a set of Hamiltonians as follows.

\begin{definition}\label{NF1}
The set $NF_j(\omega_{-1},\dots,\omega_{j-1},z_j,s_j,r_j,\xi_j,F,\varepsilon,m)$, or for short $NF_j$, consists of real-analytic Hamiltonians $H_j$ defined on $\mathcal{V}_{3s_j,3\xi_j}(z_j)$, and of the form
\begin{equation*}
\begin{cases}
H_j(z):=h(I(z))+g_j(z)+f_j(z), \quad z \in \mathcal{V}_{3s_j,3\xi_j}(z_j)  \\
h : \mathcal{D}_r \rightarrow \C, \quad ||\nabla^2 h(I)||_r \leq F, \quad I(\mathcal{V}_{3s_j,3\xi_j}(z_j)) \subseteq \mathcal{D}_r, \\
||X_{g_j}||_{3s_j,3\xi_j} \leq 2^j \varepsilon, \quad ||X_{f_j}||_{3s_j,3\xi_j} \leq  j2^{j-1} 2^{-m}\varepsilon, \\
\{l_{\omega_{-1}},g_j\}=\{l_{\omega_{0}},g_j\}=\cdots=\{l_{\omega_{j-1}},g_j\}=0.
\end{cases}.
\end{equation*}
\end{definition} 

Now let us introduce another definition, taking into account the periodic frequency $\omega_j$.

\begin{definition}\label{NF2}
The set $\widetilde{NF}_j(\omega_0,\dots,\omega_j,z_j,s_j,r_j,\xi_j,F,\varepsilon,m)$, or for short $\widetilde{NF}_j$, consists of real-analytic Hamiltonians $H_j \in NF_j(\omega_{-1},\dots,\omega_{j-1},z_j,s_j,r_j,\xi_j,F,\varepsilon,m)$ which satisfy the following additional conditions: if we denote $\tilde{\Pi}_j$ (respectively $\Pi_j$) the orthogonal projection onto $\tilde{\Lambda}_j$ (respectively $\Lambda_j$), then
\begin{equation}\label{zj}
||\tilde{\Pi}_j \nabla h(I(z_j))-\omega_j||=||\Pi_j \nabla h(I(z_j))-\omega_j|| \leq s_j
\end{equation}
and 
\begin{equation}\label{thr1}
s_j \leq (r_j+2\xi_j)\xi_j, \quad 2^j216(r_j+3\xi_j)mT_j\varepsilon  \leq s_j, \quad 72(3F\sqrt{n}+1)\xi_j^{-1}(r_j+3\xi_j)mT_js_j\leq 1.
\end{equation}
\end{definition}

The interest of the subset $\widetilde{NF}_j \subset {NF}_j$ is that if $H_j \in \widetilde{NF}_j$, then up to a change of coordinates $\Phi_j$ (which is real-analytic, symplectic and close to identity), we get that $H_j \circ \Phi_j \in NF_{j+1}$ which will constitute a main ingredient in our algorithm. Here's the precise statement.

\begin{proposition}\label{normal}
Let $H_j \in \widetilde{NF}_j(\omega_0,\dots,\omega_j,z_j,s_j,r_j,\xi_j,F,\varepsilon,m)$. Then there exist a real-analytic symplectic embedding 
\[ \Phi_j :  \mathcal{V}_{2s_j,2\xi_j}(z_j) \rightarrow \mathcal{V}_{3s_j,3\xi_j}(z_j), \quad \Phi_j\left(\mathcal{V}_{2s_j,2\xi_j}(z_j)\right)  \supseteq \mathcal{V}_{s_j,\xi_j}(z_j),\] 
such that $H_j^+:=H_j \circ \Phi_j = h +g_j^+ +f_j^+$ with
\begin{equation}\label{resonant}
\{l_{\omega_{-1}},g_j^+\}=\{l_{\omega_{0}},g_j^+\}=\cdots=\{l_{\omega_{j}},g_j^+\}=0,
\end{equation} 
and with the estimates
\begin{equation}\label{estimnormal}
||X_{g_j^+}||_{2s_j,2\xi_j} \leq 2^{j+1} \varepsilon, \quad ||X_{f_j^+}||_{2s_j,2\xi_j} \leq (j+1)2^j2^{-m}\varepsilon,  
\end{equation}
\begin{equation}\label{estimdist}
||\Phi_j - \mathrm{Id}||_{2s_j,2\xi_j} \leq 2^{j+1} T_j \varepsilon. 
\end{equation}
In particular, $H_j^+ \in NF_{j+1}(\omega_{-1},\dots,\omega_j,z_{j+1},s_{j+1},r_{j+1},\xi_{j+1},F,\varepsilon,m)$ given any choice of $s_{j+1}$, $r_{j+1}$, $\xi_{j+1}$ and $z_{j+1} \in B_{r_{j+1}}$ for which the inclusion $\mathcal{V}_{3s_{j+1},3\xi_{j+1}}(z_{j+1}) \subseteq \mathcal{V}_{2s_j,2\xi_j}(z_j)$ holds true.\end{proposition}

The proof of Proposition~\ref{normal}, which is technical, is deferred to Appendix~\ref{app}. The second part of Proposition~\ref{normal} follows easily from the first: if we define $H_{j+1}:=H_j^+$, $g_{j+1}:=g_j^+$ and $f_{j+1}:=f_j^+$, then~\eqref{resonant} read
\[ \{l_{\omega_{-1}},g_{j+1}\}=\{l_{\omega_{0}},g_{j+1}\}=\cdots=\{l_{\omega_{j}},g_{j+1}\}=0,\] 
whereas the inclusion $\mathcal{V}_{3s_{j+1},3\xi_{j+1}}(z_{j+1}) \subseteq \mathcal{V}_{2s_j,2\xi_j}(z_j)$ yields
\[ I(\mathcal{V}_{3s_{j+1},3\xi_{j+1}}(z_{j+1})) \subseteq I(\mathcal{V}_{2s_j,2\xi_j}(z_j))  \subseteq I(\mathcal{V}_{3s_j,3\xi_j}(z_j)) \subseteq \mathcal{D}_r \] 
and, together with~\eqref{estimnormal}, the estimates
\begin{equation*}
||X_{g_{j+1}}||_{3s_{j+1},3\xi_{j+1}} \leq ||X_{g_{j+1}}||_{2s_j,2\xi_j} \leq 2^{j+1} \varepsilon,
\end{equation*}
\begin{equation*}
||X_{f_{j+1}}||_{3s_{j+1},3\xi_{j+1}} \leq ||X_{f_{j+1}}||_{2s_j,2\xi_j} \leq (j+1)2^j2^{-m}\varepsilon. 
\end{equation*}
This exactly means that $H_{j+1}=H_j^+ \in NF_{j+1}(\omega_{-1},\dots,\omega_j,z_{j+1},s_{j+1},r_{j+1},\xi_{j+1},F,\varepsilon,m)$.

\subsection{Use of the normal form}\label{s52}

From now on, we will mainly work on the real domains
\[ V_{2s_j,2\xi_j}(z_j):=\mathcal{V}_{2s_j,2\xi_j}(z_j) \cap \R^{2n}, \quad 0 \leq j \leq n-1. \]
The normal form in Proposition~\ref{normal} is used to show that given a solution $z_j^+(t)$ of the Hamiltonian system associated to $H_j^+=H_j \circ \Phi_j$, the curve $I(z_j^+(t))$ has a small variation in the direction spanned by $\omega_0, \dots, \omega_{j}$, which is nothing but $\Lambda_{j+1}^\perp$ (the orthocomplement of $\Lambda_{j+1}$), for times $|t|$ as large as the inverse of $||X_{f_j^+}||_{2s_j,2\xi_j}$. It may well happen that in the direction given by $\Lambda_{j+1}$, the curve $I(z_j^+(t))$ has also a small variation and hence $I(z_j(t))$ where $z_j(t)=\Phi_j (z_j^+(t))$, has small variation, which yields our confinement result. But if not, that is if there is a faster deviation of $I(z_j^+(t))$ from $I(z_j^+(0))$, this has to occur in the direction given by $\Lambda_{j+1}$. Here is a precise statement.

\begin{proposition}\label{norm}
For $0 \leq j \leq n-1$, let $H_j \in \widetilde{NF}_j(\omega_0,\dots,\omega_j,z_j,s_j,r_j,\xi_j,F,\epsilon,m)$ and $\Phi_j :  \mathcal{V}_{2s_j,2\xi_j}(z_j) \rightarrow \mathcal{V}_{3s_j,3\xi_j}(z_j)$ given by Proposition~\ref{normal}, and let $z_j^+(t)$ be the forward solution of the Hamiltonian $H_j^+=H_j \circ \Phi_j$ starting at $z_j^+:=\Phi_j^{-1}(z_j)$. If $0 \leq j \leq n-2$, and if we define
\begin{equation} \label{eqtj} \bar{t}_j:=(r_j+\xi_j)^{-1}(j+1)^{-1}2^{-j}s_j^{-1}2^m,
\end{equation}
then we have the following dichotomy: 
\begin{itemize}
\item[$(1)$] either $z_j^+(t) \in V_{s_j,\xi_j}(z_j)$ for $0 \leq t \leq \bar{t}_j$,
\item[$(2)$] or there exists a positive time $t_j^+ < \bar{t}_j$ such that
\[ |I(z_j^+(t_j^+))-I(z_j)|=s_j/4 \]
and, for $0 \leq t \leq t_j^{+}$,
\begin{equation}\label{bazar}
z_j^+(t) \in V_{s_j,\xi_j}(z_j), \quad |I(z_j^+(t))-I(z_j)| \leq s_j/4, \quad |\Pi_{j+1}^\perp(I(z_j^+(t))-I(z_j))| \leq s_j^{-1}\varepsilon.
\end{equation} 
\end{itemize}
If $j=n-1$, and if we define 
\begin{equation} \label{eqtnn}
 \bar{t}_{n-1}:=(r_{n-1}+\xi_{n-1})^{-1}n^{-1}2^{-(n-1)}s_{n-1}\varepsilon^{-1}2^m, 
\end{equation} 
then $z_{n-1}^+(t) \in V_{s_{n-1},\xi_{n-1}}(z_{n-1})$, for $0 \leq t \leq \bar{t}_{n-1}$.
\end{proposition}

\begin{proof}
First observe that since the image of $\Phi_j$ contains $V_{s_j,\xi_j}(z_j)$, it contains $z_j$ and so $z_j^+=\Phi_j^{-1}(z_j)$ is well-defined, and we have, using~\eqref{estimdist} and the first two inequalities of~\eqref{thr1},
\[ ||z_j^+-z_j||=||z_j^+-\Phi_j(z_j^+)|| \leq 2^{j+1}T_j\varepsilon \leq s_j(108(r_j+3\xi_j))^{-1} \leq \xi_j (r_j+2\xi_j)(108(r_j+3\xi_j))^{-1} \]
which easily implies  
\begin{equation}\label{confin}
||z_j^+-z_j|| < \xi_j/108, \quad |I(z_j^+)-I(z_j)|< s_j/108.
\end{equation}
Observe also that since $z_j$ is real and $H_j$ and $\Phi_j$ are reals, the forward solution $z_j^+(t)$ is real.  

We first consider the case $0 \leq j \leq n-2$. Using~\eqref{confin}, we can now define $t_j^e \in (0,+\infty]$ to be the time of first exit of $z_j^+(t)$ from $V_{s_j,\xi_j}(z_j)$. We claim that the dichotomy of the statement is implied by the following trivial dichotomy: either $\bar{t}_j < t_j^e$ or $t_j^e \leq \bar{t}_j$.

Indeed, in the first case, one obviously have
\[ z^+_j(t) \in V_{r_j,\xi_j}(z_j), \quad 0 \leq t \leq \bar{t}_j. \]  
In the second case, either $|I(z_j^+(t_j^e))-I(z_j)|=s_j$ or $||z_j^+(t_j^e)||=r_j+\xi_j$. But since the solution is real, the second possibility implies that
\[ |I(z_j^+(t_j^e))|=1/2||z_j^+(t_j^e)||^2=1/2(r_j^2+\xi_j^2)+r_j\xi_j, \]
while
\[ |I(z_j)|=1/2||z_j||^2<1/2 r_j^2 \] 
and therefore, using the first inequality of~\eqref{thr1}, we obtain
\[ |I(z_j^+(t_j^e))-I(z_j)| \geq |I(z_j^+(t_j^e))|-|I(z_j)| > \xi_j(1/2\xi_j+r_j) > s_j/4.  \]
So, whether $|I(z_j^+(t_j^e))-I(z_j)|=s_j$ or $||z_j^+(t_j^e)||=r_j+\xi_j$, there exists a positive time $t_j^+ < t_j^e \leq \bar{t}_j$ such that
\[ |I(z_j^+(t_j^+))-I(z_j)| =s_j/4 \]
and
\[ |I(z_j^+(t))-I(z_j)| \leq s_j/4, \quad 0 \leq t \leq t_j^+. \]
Since $t_j^+ < t_j^e$, $z_j^+(t) \in V_{s_j,\xi_j}(z_j)$ for $0 \leq t \leq t_j^+$. It remains to show that
\[ |\Pi_{j+1}^\perp(I(z_j^+(t))-I(z_j))| \leq s_j^{-1}\varepsilon, \quad 0 \leq t \leq t_j^+. \]
Since $h$ is integrable, for $0 \leq s \leq t \leq t_j^+$,
\begin{eqnarray*}
\frac{d}{dt}I(z_j^+(s)) & = & \{I,H_j^+\}(z_j^+(s)) := (\{I_1,H_j^+\}(z_j^+(s)),\dots,\{I_n,H_j^+\}(z_j^+(s))) \\
& = & \{I,h+g_j^++f_j^+\}(z_j^+(s))=\{I,g_j^++f_j^+\}(z_j^+(s)). 
\end{eqnarray*}
Then, for any $z \in V_{s_j,\xi_j}(z_j)$ and any $-1 \leq i \leq j$, using~\eqref{resonant} we obtain
\[ \{l_{\omega_{i}},g_j^+\}(z)=\omega_i\cdot \{I,g_j^+\}(z)=0 \]
which implies that $\{I,g_j^+\}(z) \in \Lambda_{j+1}$. Therefore
 \[ \Pi_{j+1}^\perp\left(\frac{d}{dt}I(z_j^+(s))\right)=\{I,f_j^+\}(z_j^+(s)) \] 
and hence
\[ \Pi_{j+1}^\perp(I(z_j^+(t))-I(z_j))=\int_0^t \Pi_{j+1}^\perp\left(\frac{d}{dt}I(z_j^+(s))\right) ds=\int_0^t \{I,f_j^+\}(z_j^+(s)) ds  \]
and, using the second inequality of~\eqref{estimnormal} and the fact that
\[|\{I,f_j^+\}(z_j^+(s))|\leq ||z_j^+(s)||||X_f(z_j^+(s))||\]
we obtain
\[ |\Pi_{j+1}^\perp(I(z_j^+(t))-I(z_j))| \leq t (r_j+\xi_j) ||X_{f_j^+}||_{2s_j,2\xi_j} \leq t (r_j+\xi_j)(j+1)2^j2^{-m}\varepsilon.  \]
Since $t_j^+ \leq t_j^e \leq \bar{t}_j=(r_j+\xi_j)^{-1}(j+1)^{-1}2^{-j}s_j^{-1}2^m$, we thus obtain
\[ |\Pi_{j+1}^\perp(I(z_j^+(t))-I(z_j))|\leq s_j^{-1}\varepsilon, \quad 0 \leq t \leq t_j^{+}, \]
which concludes the proof for the case $0 \leq j \leq n-2$.

Now for the case $j=n-1$, we have $\Lambda_{n}=\{0\}$ and so $\Lambda_{n}^\perp=\R^n$, hence for $t \leq \bar{t}_{n-1}=(r_{n-1}+\xi_{n-1})^{-1}n^{-1}2^{-(n-1)}s_{n-1}\varepsilon^{-1}2^m$, repeating the last argument we get
\[ |I(z_{n-1}^+(t))-I(z_{n-1})| \leq s_{n-1}, \quad 0 \leq t \leq \bar{t}_{n-1}.  \]
As $z_{n-1}^+(t)$ is real, this implies, using also the first inequality of~\eqref{thr1}, that for $0 \leq t \leq \bar{t}_{n-1}$, 
\begin{eqnarray*}
||z_{n-1}^+(t)||^2 & = & 2 |I(z_{n-1}^+(t))| \\
& \leq & 2|I(z_{n-1}^+(t))-I(z_{n-1})| + 2 |I(z_{n-1})| \\
& \leq & 2s_{n-1}+||z_{n-1}(t)||^2 \\
& \leq & 2(r_{n-1}+2\xi_{n-1})\xi_{n-1}+r_{n-1}^2 \\
& \leq & (r_{n-1}+2\xi_{n-1})^2
\end{eqnarray*}
so $z_{n-1}^+(t) \in V_{s_{n-1},2\xi_{n-1}}(z_{n-1})$ for $0 \leq t \leq \bar{t}_{n-1}$, and this concludes the proof of the proposition.   
\end{proof}

\subsection{Use of the steepness property}\label{s53}

Let us start by giving a geometric interpretation of the steepness property, as its definition is quite abstract. Assume that $h$ is steep on some domain $D$, and consider a curve $\gamma : [0,1] \rightarrow \R^n$ which takes values in $\lambda \cap D$, where $\lambda$ is a proper affine subspace of $\R^n$. It may happen that $\nabla h_\lambda(\gamma(0))=0$ (this is the case if $\gamma(0)$ is a resonant point for $h$, that is, if $k\cdot \nabla h(\gamma(0))$ for some non-zero integer vector $k \in \Z^n$: then $\nabla h_\lambda(\gamma(0))=0$ where $\lambda$ is the real space generated by such integer vectors $k$). If this happens, the steepness property ensures that, for some time $0 <\tilde{t} \leq 1$, $\nabla h_\lambda(\gamma(\tilde{t})) \neq 0$ (informally, in terms of resonances, this means that we do not have ``accumulation of resonances"). Moreover, the longer is the length of the curve $\gamma$, the farther away from zero is the vector $\nabla h_\lambda(\gamma(\tilde{t}))$. Here's a quantitative statement, which is due to Nekhoroshev.

\begin{lemma}[Nekhoroshev]\label{lemmenekho}
Let $h$ be a function which is $(r,\kappa,C,\delta,(p_l)_{l=1,\ldots,n-1})$-steep, and such that
\[ ||\nabla^2 h (I)||_r \leq F. \]
Let $\gamma : [0,t^+] \rightarrow \R^n$ be a continuous curve, $\lambda$ an affine subspace of $\R^n$ of dimension $l$, where $1 \leq l \leq n-1$, and $d$ a positive real number. Assume that
\begin{itemize}
\item[$(i)$] for all $t \in [0,t^+]$, $\gamma(t) \in \lambda
$;
\item[$(ii)$] for all $t \in [0,t^+]$, $||\gamma(0)- \gamma(t)||\leq d$ and $||\gamma(0)-\gamma(t^+)||=d$;
\item[$(iii)$] the ball $\{I \in \R^n \; | \; ||I-\gamma(0)||\leq d\}$ is contained in $D_r$;
\item[$(iv)$] $d < \min\{\delta,(3F)^{-1}\kappa,2(5\kappa(4C)^{-1})^{1/p_l}\}$,
\end{itemize}
then there exists a time $\tilde{t} \in [0,t^+]$ such that 
\[ ||\Pi_\Lambda \nabla h(\gamma(\tilde{t}))||>C/5(d/2)^{p_l}, \]
where $\Lambda$ is the vector space associated to $\lambda$, and $\Pi_\Lambda$ the orthogonal projection onto $\Lambda$.
\end{lemma}

This is a special case of the lemma on ``almost plane curves" of Nekhoroshev, stated in \cite{Nek77} and proved in \cite{Nek79} (our case corresponds to ``plane curves").

Now assume that Alternative $(2)$ of Proposition~\ref{norm} holds true, and let $\gamma_j(t):=I(z_j^+)+\Pi_{j+1}(I(z_j^+(t))-I(z_j^+))$ for $t \in [0,t_j^+]$. Since this curve takes values in a proper affine subspace, the following proposition is a simple consequence of Proposition \ref{norm} and Lemma~\ref{lemmenekho}.

\begin{proposition}\label{steep}
For $0 \leq j \leq n-1$, let $H_j \in \widetilde{NF}_j(\omega_0,\dots,\omega_j,z_j,s_j,r_j,\xi_j,F)$ and $\Phi_j :  \mathcal{V}_{2s_j,2\xi_j}(z_j) \rightarrow \mathcal{V}_{3s_j,3\xi_j}(z_j)$ given by Proposition~\ref{normal}, and let $z_j^+(t)$ be the forward solution of the Hamiltonian $H_j^+=H_j \circ \Phi_j$ starting at $z_j^+=\Phi_j^{-1}(z_j)$. Assume that $h$ is $(r,\kappa,C,\delta,(p_l)_{l=1,\ldots,n-1})$-steep. Then we have the following dichotomy for $j \leq n-2$: 
\begin{itemize}
\item[$(1)$] either $z_j^+(t) \in V_{s_j,\xi_j}(z_j)$ for $0 \leq t \leq \bar{t}_j$,
\item[$(2)$] or there exists a time $\tilde{t}_j \leq t_j^+ < \bar{t}_j$ such that, setting $\gamma_j(\tilde{t}_j):=I(z_j^+)+\Pi_{j+1}(I(z_j^+(\tilde{t}_j))-I(z_j^+))$, then
\begin{equation}\label{munu}
||\Pi_{j+1} \nabla h(\gamma_j(\tilde{t}_j))||> \mu_j s_j^{p_{n-j-1}}, \quad \mu_j:=5^{-1}C16^{-p_{n-j-1}},
\end{equation}
provided that
\begin{equation}\label{thr2}
\begin{cases}
\varepsilon < s_j^2/8, \\ 
s_j < 8\min\{\delta,(3F)^{-1}\kappa,2(5\kappa(4C)^{-1})^{1/p_{n-j-1}}\}.
\end{cases}
\end{equation}
\end{itemize}
If $j=n-1$, then $z_{n-1}^+(t) \in V_{s_{n-1},\xi_{n-1}}(z_{n-1})$, for $0 \leq t \leq \bar{t}_{n-1}$.

\end{proposition}

\begin{proof} We only have to consider the case $j\leq n-2$, and we have to prove that Alternative $(2)$ of Proposition~\ref{norm} implies Alternative $(2)$  of the above proposition. So we assume the existence of a positive time $t_j^+ < \bar{t}_j$ such that
\[ |I(z_j^+(t_j^+))-I(z_j)|=s_j/4 \]
and, for $0 \leq t \leq t_j^{+}$,
\[ z_j^+(t) \in V_{s_j,\xi_j}(z_j), \quad |I(z_j^+(t))-I(z_j)| \leq s_j/4, \quad |\Pi_{j+1}^\perp(I(z_j^+(t))-I(z_j))| \leq s_j^{-1}\varepsilon. \] 
Hence, using the first inequality of~\eqref{thr2}, 
\[ |\Pi_{j+1}(I(z_j^+(t_j^+))-I(z_j))| \geq |I(z_j^+(t_j^+))-I(z_j)|-|\Pi_{j+1}^\perp(I(z_j^+(t))-I(z_j)| \geq s_j/4 - s_j^{-1}\varepsilon \geq s_j/8\]
and in particular
\[ ||\Pi_{j+1}(I(z_j^+(t_j^+))-I(z_j))|| \geq s_j/8. \]
Therefore we can certainly find a positive time $\tilde{t}_j^+ \leq t_j^+$ such that
\[ ||\Pi_{j+1}(I(z_j^+(\tilde{t}_j^+))-I(z_j))|| = s_j/8 \]
and
\begin{equation}\label{estoubli}
||\Pi_{j+1}(I(z_j^+(t))-I(z_j))|| \leq s_j/8, \quad 0 \leq t \leq \tilde{t}_j^+. 
\end{equation}
Now we want to apply Lemma~\ref{lemmenekho} to the curve $\gamma_j(t)=I(z_j^+)+\Pi_{j+1}(I(z_j^+(t))-I(z_j^+))$, for $t \in [0,\tilde{t}_j^+]$, with $d:=s_j/8$ and with the affine subspace $\lambda_{j+1}:=I(z_j^+)+\Lambda_{j+1}$ which has dimension $n-j-1$. The assumptions $(i)$ and $(ii)$ of Lemma~\eqref{lemmenekho} are trivially satisfied, as
\[ \gamma_j(\tilde{t}_j^+)-\gamma_j(0)=\Pi_{j+1}(I(z_j^+(\tilde{t}_j^+)-I(z_j)). \] 
Then $(iii)$ holds true since, by definition of $H_j$, we have $I(V_{2s_j,2\xi_j}(z_j)) \subseteq I(V_{3s_j,3\xi_j}(z_j)) \subseteq D_r$. Eventually, the second inequality of~\eqref{thr2} clearly implies $(iv)$ therefore Lemma~\ref{lemmenekho} can be applied, and there exists a time $\tilde{t}_j \in [0,\tilde{t}_j^+]$ such that 
\[ ||\Pi_{j+1}\nabla h(\gamma_j(\tilde{t}_j))||>5^{-1}C(d/2)^{p_{n-j-1}}=\mu_j s_j^{p_{n-j-1}}.  \]
This concludes the proof.  
\end{proof}

\subsection{Use of periodic approximations}\label{s54}

Let us first state the following simple consequence of Dirichlet's theorem on approximation of real vectors by rational vectors.

\begin{lemma}\label{dirichlet}
Let $v\in \R^n\setminus\{0\}$, and $Q \geq 1$ a real number. Then there exists a $T$-periodic vector $\omega \in \R^n\setminus\{0\}$ such that
\[ ||v-\omega||\leq \sqrt{n-1}(TQ)^{-1}, \quad ||v||^{-1} \leq T \leq \sqrt{n} ||v||^{-1}Q^{n-1}.  \]
\end{lemma}

\begin{proof}
Fix $Q \geq 1$. Up to a re-ordering of its component, we can write $v=|v|(\pm 1, x)$ for some $x \in \R^{n-1}$ and by Dirichlet's approximation theorem, there exists a rational vector $p/q \in \Q^{n-1}$, such that
\[ |qx-p| \leq Q^{-1}, \quad 1 \leq q \leq Q^{n-1}. \]
The vector $\omega=|v|(\pm 1, p/q) \in \R^n$ is then $T$-periodic, for $T=|v|^{-1}q$, and we have
\[ ||v-\omega||\leq T^{-1}||qx-p||, \quad |v|^{-1} \leq T \leq |v|^{-1}Q^{n-1}  \]
which implies
\[ ||v-\omega||\leq  \sqrt{n-1}(TQ)^{-1}, \quad ||v||^{-1} \leq T \leq \sqrt{n}||v||^{-1}Q^{n-1} \]
and this was the statement to prove.
\end{proof}

Now assume that the conclusion of Alternative 2 of Proposition~\ref{steep} holds true, so the vector $\Pi_{j+1} \nabla h(\gamma_j(\tilde{t}_j))$ is non-zero, where $\gamma_j(\tilde{t}_j)=I(z_j^+)+\Pi_{j+1}(I(z_j^+(\tilde{t}_j))-I(z_j^+))$. By Lemma~\ref{dirichlet} this non-zero vector can be approximated by a periodic vector $\omega_{j+1}$, and it will be easy to ensure that this new periodic vector is linearly independent from $\omega_0, \dots, \omega_j$ (as $\omega_{j+1}$ is close to $\Pi_{j+1} \nabla h(\gamma_j(\tilde{t}_j))$, the latter being, obviously, linearly independent from $\omega_0, \dots, \omega_j$ as it is orthogonal to them). Moreover, as $\gamma_j(\tilde{t}_j)$ is close to $I(z^+_j(\tilde{t}_j))$, setting $z_{j+1}:=z^+_j(\tilde{t}_j)$, the vector $\omega_{j+1}$ is also an approximation of $\Pi_{j+1} \nabla h(I(z_{j+1}))$. This leads to the following proposition.

\begin{proposition}\label{periodic}
For $0 \leq j \leq n-1$, let $H_j \in \widetilde{NF}_j(\omega_0,\dots,\omega_j,z_j,s_j,r_j,\xi_j,F,\epsilon,m)$ and $\Phi_j :  \mathcal{V}_{2s_j,2\xi_j}(z_j) \rightarrow \mathcal{V}_{3s_j,3\xi_j}(z_j)$ given by Proposition~\ref{normal}, and let $z_j^+(t)$ be the forward solution of the Hamiltonian $H_j^+=H_j \circ \Phi_j$ starting at $z_j^+=\Phi_j^{-1}(z_j)$. Assume that $h$ is $(r,\kappa,C,\delta,(p_l)_{l=1,\ldots,n-1})$-steep. Then, for $j\leq n-2$, we have the following dichotomy: 
\begin{itemize}
\item[$(1)$] either $z_j^+(t) \in V_{s_j,\xi_j}(z_j)$ for $0 \leq t \leq \bar{t}_j$,
\item[$(2)$] or, given some positive parameter $Q \geq 1$, there exists a $T_{j+1}$-periodic vector $\omega_{j+1}\in \R^n\setminus\{0\}$, linearly independent from $\omega_0, \dots, \omega_j$, with the estimate
\begin{equation} (F'\sqrt{n} s_j)^{-1}< T_{j+1} < \sqrt{n} \mu_j^{-1} s_j^{-p_{n-j-1}} Q^{n-1}, \quad F':=\max\{1,F\}, \label{Tjplus1} \end{equation}
and a time $\tilde{t}_j\leq \leq \bar{t}_j$ such that if we define
\[ z_{j+1}:=z_j^+(\tilde{t}_j), \quad  s_{j+1}:=2\sqrt{n-1}(T_{j+1}Q)^{-1}, \quad r_{j+1}:=r_j+\xi_j, \quad \xi_{j+1}:=\xi_j/3 \] 
then it holds that 
\begin{equation}\label{incl}
\mathcal{V}_{3s_{j+1},3\xi_{j+1}}(z_{j+1}) \subseteq  \mathcal{V}_{2s_{j},2\xi_{j}}(z_{j})
\end{equation}
\begin{equation}\label{Zj}
||\Pi_{j+1}\nabla h(I(z_{j+1}))-\omega_{j+1}|| \leq s_{j+1}, 
\end{equation} 
\begin{equation}\label{apart}
z_j^+(t) \in V_{s_j,\xi_j}(z_j), \quad 0 \leq t \leq \tilde{t}_j, 
\end{equation}
provided that
\begin{equation}\label{thr3}
\begin{cases}
\varepsilon < s_j^2/8, \\ 
s_j < 8\min\{\delta,(3F)^{-1}\kappa,2(5\kappa(4C)^{-1})^{1/p_{n-j-1}}\}, \\
Q \geq 8F'\sqrt{n(n-1)}, \\
\varepsilon \leq (2\sqrt{n}F)^{-1}s_j s_{j+1}.
\end{cases}
\end{equation}
\end{itemize}
If $j=n-1$, then $z_{n-1}^+(t) \in V_{s_{n-1},\xi_{n-1}}(z_{n-1})$, for $0 \leq t \leq \bar{t}_{n-1}$.
\end{proposition}

\begin{proof}
Since~\eqref{thr3} implies in particular~\eqref{thr2}, it is enough to prove that Alternative $(2)$ of Proposition~\ref{steep} implies Alternative $(2)$ of the above proposition. So we assume the existence of a time $\tilde{t}_j \leq t_j^+ < \bar{t}_j$ such that 
\[ ||\Pi_{j+1} \nabla h (\gamma_j(\tilde{t}_j))||> \mu_j s_j^{p_{n-j-1}}, \quad \mu_j=5^{-1}C16^{-p_{n-j-1}}, \]
where $\gamma_j(\tilde{t}_j)=I(z_j^+)+\Pi_{j+1} (I(z_j^+(\tilde{t}_j))-I(z_j^+))$. Let us define $v_{j+1}:=\Pi_{j+1} \nabla h (\gamma_j(\tilde{t}_j))$. We have
\[  ||\gamma_j(\tilde{t_j})-I(z_j)|| \leq ||\Pi_{j+1}(I(z_j^+(\tilde{t}_j))-I(z_j^+))||+||I(z_j^+)-I(z_j)|| \leq s_j/8+\sqrt{n} s_j/108, \]
where we used the estimate~\eqref{estoubli} (as $\tilde{t}_j \leq \tilde{t}_j^+$) and the estimate~\eqref{confin}. Since $n \geq 2$ this implies
\[ ||\gamma_j(\tilde{t_j})-I(z_j)|| \leq (\sqrt{n}-1)s_j \] 
which implies 
\[ ||v_{j+1}-\Pi_{j+1} \nabla h (I(z_j))|| \leq F(\sqrt{n}-1)s_j. \]
Now recall that by definition of $z_j$, we have
\[ ||\Pi_{j}\nabla h (I(z_j))-\omega_j|| \leq s_j\] 
and since $\omega_j \in \Lambda_{j+1}^\perp$ and $\Lambda_{j+1} \subseteq \Lambda_j$, $\Pi_{j+1}\omega_j=0$ and $\Pi_{j+1}=\Pi_{j+1}\Pi_{j}$, and therefore 
\[ ||\Pi_{j+1}\nabla h (I(z_j))||=||\Pi_{j+1}(\Pi_{j}(\nabla h (I(z_j)))-\omega_j)|| \leq s_j \] 
which implies that
\[ ||v_{j+1}|| \leq ||v_{j+1}-\Pi_{j+1} \nabla h (I(z_j))|| + ||\Pi_{j+1}\nabla h (I(z_j))|| \leq F'\sqrt{n}s_j. \]
We just proved that
\begin{equation}\label{estimev}
\mu_j s_j^{p_{n-j-1}} < ||v_{j+1}|| \leq F'\sqrt{n}s_j.
\end{equation}
Now, for $Q\geq 1$, we apply Lemma~\ref{dirichlet} to $v_{j+1}$: there exists a $T_{j+1}$-periodic vector $\omega_{j+1} \in \R^n\setminus\{0\}$ such that
\[ ||v_{j+1}-\omega_{j+1}||\leq \sqrt{n-1}(T_{j+1}Q)^{-1}, \quad ||v_{j+1}||^{-1} \leq T_{j+1} \leq \sqrt{n} ||v_{j+1}||^{-1}Q^{n-1}.  \]
Using~\eqref{estimev}, this implies
\begin{equation}\label{period}
(F'\sqrt{n} s_j)^{-1}\leq ||v_{j+1}||^{-1} \leq T_{j+1} \leq \sqrt{n} ||v_{j+1}||^{-1}Q^{n-1} < \sqrt{n} \mu_j^{-1} s_j^{-p_{n-j-1}} Q^{n-1}
\end{equation}
and also, using the lower bound on $T_{j+1}$,
\begin{equation}\label{approx}
||v_{j+1}-\omega_{j+1}||\leq \sqrt{n-1}(T_{j+1}Q)^{-1} \leq \sqrt{n-1}||v_{j+1}||Q^{-1}\leq F'\sqrt{n-1}\sqrt{n} s_jQ^{-1}.  
\end{equation}
Let us prove that $\omega_{j+1}$ is linearly independent from $\omega_0, \dots, \omega_{j}$, that is $\omega_{j+1}$ does not belong to $\Lambda_{j+1}^\perp$. To do this, it is enough to prove that if $v$ is an arbitrary vector in $\Lambda_{j+1}^\perp$, then $|\omega_{j+1}\cdot v| < ||\omega_{j+1}||||v||$: indeed, otherwise, letting $v=\omega_{j+1}$, one would get a contradiction. On the one hand, we have
\[ |\omega_{j+1}\cdot v|=|(\omega_{j+1}-v_{j+1})\cdot v|\leq ||v_{j+1}-\omega_{j+1}||||v||\leq \sqrt{n-1}Q^{-1}||v_{j+1}||||v||   \]
where we used the fact $v_{j+1} \in \Lambda_{j+1}$ and~\eqref{approx}, while, on the other hand,
\[ ||\omega_{j+1}||||v|| \geq (||v_{j+1}||-||v_{j+1}-\omega_{j+1}||)||v||\geq (1 -\sqrt{n-1}Q^{-1})||v_{j+1}||||v||> \sqrt{n-1}Q^{-1}||v_{j+1}||||v|| \]
where we used the third inequality of~\eqref{thr3} and~\eqref{approx}. These last two inequalities imply that $|\omega_{j+1}\cdot v| < ||\omega_{j+1}||||v||$ for an arbitrary vector $v \in \Lambda_{j+1}^\perp$, and so $\omega_{j+1}$ is linearly independent from $\omega_0, \dots, \omega_{j}$. 

Next we define
\[ z_{j+1} := z_j^+(\tilde{t}_j), \]
and observe that, since $\tilde{t}_j \leq t_j^+$, by~\eqref{bazar}, $z_{j+1} \in V_{s_j,\xi_j}(z_j)$, but also $z_j^+(t) \in V_{s_j,\xi_j}(z_j)$ for $0 \leq t \leq \tilde{t}_j$, which justifies~\eqref{apart}. Moreover, still from~\eqref{bazar},
\[ |I(z_{j+1})-\gamma_j(\tilde{t}_j)|=|\Pi_{j+1}^\perp (I(z_{j+1})-I(z_{j}^+))| \leq s_j^{-1}\varepsilon \]
so
\[ ||I(z_{j+1})-\gamma_j(\tilde{t}_j)|| \leq \sqrt{n} s_j^{-1}\varepsilon  \]
and hence
\[ ||\Pi_{j+1} \nabla h (I(z_{j+1}))-v_{j+1}|| \leq F \sqrt{n} s_j^{-1}\varepsilon. \]
Therefore, using the first inequality of~\eqref{approx}, the definition of $s_{j+1}$ and the last inequality of~\eqref{thr3}, 
\begin{eqnarray*}
||\Pi_{j+1} \nabla h (I(z_{j+1}))-\omega_{j+1}|| & \leq & ||\Pi_{j+1} \nabla h (I(z_{j+1}))-v_{j+1}||+||v_{j+1}-\omega_{j+1}|| \\
 & \leq & F\sqrt{n} s_j^{-1}\varepsilon +s_{j+1}/2 \leq s_{j+1},
\end{eqnarray*} 
which proves~\eqref{Zj}. It remains to check~\eqref{incl}, so let us fix $z \in \mathcal{V}_{3s_{j+1},3\xi_{j+1}}(z_{j+1})$. First, we have
\[ I(z)-I(z_j)=I(z)-I(z_{j+1})+I(z_{j+1})-I(z_{j})=I(z)-I(z_{j+1})+I(z_j^+(\tilde{t}_j))-I(z_{j}) \in \tilde{\Lambda}_j 
\] 
since $I(z)-I(z_{j+1}) \in \tilde{\Lambda}_{j+1} \subseteq \tilde{\Lambda}_j$ and $I(z_j^+(\tilde{t}_j))-I(z_{j}) \in \tilde{\Lambda}_j$. Then,
\[ |I(z)-I(z_j)| \leq |I(z)-I(z_{j+1})|+|I(z_{j+1})-I(z_j)| < 3s_{j+1}+s_j/4 < 2 s_j \]
provided that $s_{j+1}\leq 7s_j/12$: but this inequality (in fact, the stronger inequality $s_{j+1} \leq s_j/2$) follows from the definition of $s_{j+1}$, the third inequality of~\eqref{thr3} and~\eqref{approx}. Eventually,
\[ ||z|| < r_{j+1}+3\xi_{j+1}=r_j+\xi_j+3\xi_{j+1}=r_j+2\xi_j  \]
and so we showed that $\mathcal{V}_{3s_{j+1},3\xi_{j+1}}(z_{j+1}) \subseteq \mathcal{V}_{2s_j,2\xi_j}(z_j)$, which concludes the proof.
\end{proof}

\subsection{One step of the algorithm}\label{s55}

As a straightforward application of Proposition~\ref{normal}, Proposition~\ref{norm} and Proposition~\ref{periodic}, we now describe formally one step of the algorithm that will eventually lead to the proof of Theorem~\ref{mainthm03}. 

\begin{proposition}\label{algostep}
For $0 \leq j \leq n-1$, let $H_j \in \widetilde{NF}_j(\omega_0,\dots,\omega_j,z_j,s_j,r_j,\xi_j,F)$ and $\Phi_j :  \mathcal{V}_{2s_j,2\xi_j}(z_j) \rightarrow \mathcal{V}_{3s_j,3\xi_j}(z_j)$ given by Proposition~\ref{normal}, and let $z_j^+(t)$ be the forward solution of the Hamiltonian $H_j^+=H_j \circ \Phi_j$ starting at $z_j^+=\Phi_j^{-1}(z_j)$. Assume that $h$ is $(r,\kappa,C,\delta,(p_l)_{l=1,\ldots,n-1})$-steep. Then, for $0 \leq j \leq n-2$, we have the following dichotomy: 
\begin{itemize}
\item[$(1)$] either $z_j^+(t) \in V_{s_j,\xi_j}(z_j)$ for $0 \leq t \leq \bar{t}_j$,
\item[$(2)$] or, given a real number $Q\geq 1$, there exists a $T_{j+1}$-periodic vector $\omega_{j+1}\in \R^n\setminus\{0\}$, linearly independent from $\omega_0, \dots, \omega_j$, with the estimate
\begin{equation}\label{borneperiode}
(F'\sqrt{n} s_j)^{-1}< T_{j+1} < \sqrt{n} \mu_j^{-1} s_j^{-p_{n-j-1}} Q^{n-1}, \quad F'=\max\{1,F\}, 
\end{equation}
and there exists a time $\tilde{t}_j  < \bar{t}_j$ such that, 
\[ z_{j+1}=z_j^+(\tilde{t}_j) \in B_{r_{j+1}}, \quad s_{j+1}=2\sqrt{n-1}(T_{j+1}Q)^{-1}, \quad r_{j+1}=r_j+\xi_j, \quad \xi_{j+1}=\xi_j/3 \]  
such that $H_j^+ \in \widetilde{NF}_{j+1}(\omega_0,\dots,\omega_{j+1},z_{j+1},s_{j+1},r_{j+1},\xi_{j+1},F,\varepsilon,m)$ and
\begin{equation}\label{apart2}
z_j^+(t) \in V_{s_j,\xi_j}(z_j), \quad 0 \leq t \leq \tilde{t}_j, 
\end{equation}
provided that
\begin{equation}\label{thr4}
\begin{cases}
s_j < 8\min\{\delta,(3F)^{-1}\kappa,2(5\kappa(4C)^{-1})^{1/p_{n-j-1}}\}, \\
Q \geq 8F'\sqrt{n(n-1)}, \\
\varepsilon \leq (2\sqrt{n}F')^{-1}s_j s_{j+1}, \\
2^{j+1}216(r_{j+1}+3\xi_{j+1})mT_{j+1}\varepsilon  \leq s_{j+1}, \\ 
72(3F\sqrt{n}+1)\xi_{j+1}^{-1}(r_{j+1}+3\xi_{j+1})mT_{j+1}s_{j+1}\leq 1.
\end{cases}
\end{equation}
\end{itemize}
If $j=n-1$, then $z_{n-1}^+(t) \in V_{s_{n-1},2\xi_{n-1}}(z_{n-1})$ for $0 \leq t \leq \bar{t}_{n-1}$.
\end{proposition}

\begin{proof}
The case $j=n-1$ follows directly from the case $j=n-1$ of Proposition~\ref{norm}. Then, we claim that the inequalities~\eqref{thr4} imply the inequalities~\eqref{thr3} and the inequalities~\eqref{thr1} (with $j$ replaced by $j+1$). Assuming this claim, and using Proposition~\ref{periodic}, we have the inclusion of the complex domains of \eqref{incl}, and therefore using the second part of the statement of Proposition~\ref{normal}, we can assert that $H_j^+ \in NF_{j+1}$. Moreover, in view of~\eqref{Zj},and since~\eqref{thr1} is satisfied (with $j$ replaced by $j+1$), we eventually obtain that $H_j^+ \in \widetilde{NF}_{j+1}$, while~\eqref{apart2} is nothing but~\eqref{apart}.

It remains to prove the claim. To do this, observe that~\eqref{thr4} obviously implies~\eqref{thr3} and~\eqref{thr1}, except for the following two inequalities:
\begin{equation}\label{r1}
\varepsilon < s_j^2/8, \quad s_{j+1} \leq (r_{j+1}+2\xi_{j+1})\xi_{j+1}.
\end{equation}
But using the third inequality of~\eqref{thr4} and the fact that $H_j \in \widetilde{NF}_j$, we know that
\begin{equation}\label{r2}
\varepsilon \leq (2\sqrt{n}F')^{-1}s_j s_{j+1}, \quad s_j \leq (r_j+2\xi_j)\xi_j.
\end{equation}
Then, using the second inequality of~\eqref{thr4}, one easily check that $s_{j+1} \leq s_j/4$, and this, together with~\eqref{r2}, imply~\eqref{r1}, and the proof is over.  
\end{proof}

\subsection{Proof of Nekhoroshev exponential stability }\label{s56}

We can finally give the proof of Theorem~\ref{mainthm03}. Recall that we are given a Hamiltonian $H$ as in~\eqref{Ham}, which is defined on $\mathcal{B}_r$, and of the form
\[ H(z)=h(I(z))+f(z), \quad h : \mathcal{D}_r \rightarrow \C, \quad f : \mathcal{B}_r \rightarrow \C \]
and that~\eqref{condsteep} holds true, that is
\[   ||\nabla h ||_r \leq E, \quad ||\nabla^2 h ||_r \leq F, \quad ||X_f||_r \leq \varepsilon. \]
We already defined $F'=\max\{1,F\}$. Recall also that $h$ is $(r,\kappa,C,\delta,(p_l)_{l=1,\ldots,n-1})$-steep. Let us now define additional parameters: for any $0 \leq j \leq n-1$ and $0 \leq  k \leq j$, we set
\[ \pi_j^k := \prod_{n-j \leq i \leq n-j+k-1} p_i, \quad a_j^{k}:=\sum_{0 \leq i \leq k} \pi_j^i, \]
with the convention that the product over the empty set is one, that is, $\pi_j^0=1$. Observe in particular that
\[ a_0^0=1, \quad a_1^1=1+p_{n-1}, \]
and at the other extreme,
\[ a_{n-2}^{n-2}=1+p_2+p_2p_3+\cdots+p_2p_3\dots p_{n-1}=a',\]
\[ a_{n-1}^{n-1}=1+p_1+p_1p_2+\cdots+p_1p_2\dots p_{n-1}=a. \]
For $0 \leq j \leq n-2$, recalling that the numbers $\mu_j$ have been defined in~\eqref{munu}, we define 
\[ \eta:= \min_{0 \leq j \leq n-2}\{(3F)^{-1}\kappa,2(5\kappa(4C)^{-1})^{1/p_{n-j-1}}\}, \quad \nu_j:=\prod_{i=0}^{j-1}\mu_i^{\pi_j^{j-1-i}}.\]
The proof of Theorem~\ref{mainthm03} will be a consequence of the following proposition.

\begin{proposition}\label{algo} 
Let $H(z)=h(I(z))+f(z)$ be as in~\eqref{Ham} satisfying~\eqref{condsteep}, such that $h$ is $(r,\kappa,C,\delta,(p_l)_{l=1,\ldots,n-1})$-steep. Let $z_0$ be an arbitrary point in $B_{r/2}$ and $z(t)$ the forward solution of $H$ starting at $z_0$. Given an integer $m \geq 1$ and a real number $Q \geq 1$, we have 
\[ |I(z(t))-I(z_0)| \leq s:=3E\sqrt{n-1} Q^{-1}, \quad 0 \leq t \leq \bar{t}:=3(2rE\sqrt{n-1})^{-1}Q2^m,  \]
provided that:
\begin{equation}\label{cond}
\begin{cases}\tag{C}
Q \geq (5r^2)^{-1}36E\sqrt{n-1}, \\
Q>E\sqrt{n-1}(8\eta)^{-1}, \\
Q>E\sqrt{n-1}(8\delta)^{-1}, \\
Q \geq 8F'\sqrt{n(n-1)}, \\
2\sqrt{n}F'\sqrt{n}^{a+a'}\sqrt{n-1}^{-(a+a')}\nu_{n-1}^{-1}\nu_{n-2}^{-1}\kappa^{-(\pi_{n-1}^{n-1}+\pi_{n-2}^{n-2})} Q^{n(a+a')}\varepsilon \leq 1, \\
2^{n-1}27(3+3^{-n+1})r(n-1)^{-a}n^{a}\sqrt{n-1}\kappa^{-2\pi_{n-1}^{n-1}}\nu_{n-1}^{-2}mQ^{2na-1}\varepsilon \leq 1, \\
Q \geq m216(3^n+1)(3F\sqrt{n}+1)\sqrt{n-1}.
\end{cases}
\end{equation}
\end{proposition}

Let us first prove this proposition. The fact that this proposition implies Theorem~\ref{mainthm03} simply follows from a suitable choice of $m$ and $Q$ (in terms of our given parameters) and will be detailed later.

\begin{proof}[Proof of Proposition~\ref{algo}]
The proof follows from an algorithm whose inductive step is given by Proposition~\ref{algostep}. But first we need to initiate the algorithm. By assumptions we have 
\[ \kappa \leq ||\nabla h (I(z_0))|| \leq E \] 
and so we can apply Lemma~\ref{dirichlet} to $v_0:=\nabla h (I(z_0))$: there exists a $T_0$-periodic vector $\omega_0 \in \R^n\setminus\{0\}$ such that
\begin{equation}\label{T0}
||v_0-\omega_0||\leq \sqrt{n-1}(T_0Q)^{-1}, \quad E^{-1} \leq T_0 \leq \sqrt{n}\kappa^{-1}Q^{n-1}. 
\end{equation} 
We define
\[ H_0:=H, \quad s_0:=\sqrt{n-1}(T_0Q)^{-1}, \quad r_0:=r/2, \quad \xi_0:=r_0/3=r/6, \]
and observe that $H_0 \in NF_0(\omega_{-1},z_0,s_0,r_0,\xi_0,F,\varepsilon,m)$. Indeed, $\tilde{\lambda}_0=\C^n$, $r_0+3\xi_0=r$ so that $\mathcal{V}_{3s_{0},3\xi_0}(z_0) \subseteq \mathcal{B}_r$, and we can write $H_0=h+f=h+g_0+f_0$, with $g_0:=f$ and $f_0:=0$,
\[ ||X_{g_0}||_{3r_0,3\xi_0}\leq ||X_f||_{s} \leq \varepsilon, \]
as the requirement $\{l_{\omega_{-1}},g_0\}=0$ is void since $\omega_{-1}=0$. In fact, using the first inequality of~\eqref{T0} and assuming that
\begin{equation}\label{cond0}
\begin{cases}\tag{C0}
Q \geq (5r^2)^{-1}36E\sqrt{n-1}, \\
216n\sqrt{n-1}^{-1}\kappa^{-2}rmQ^{2n-1}\varepsilon \leq 1, \\
Q \geq m216.3(3F\sqrt{n}+1)\sqrt{n-1},
\end{cases}
\end{equation}
one easily check that, using the definitions of $s_0$, $ r_0$, $\xi_0$ (which gives in particular $r_0+3\xi_0=r$ and $(r_0+3\xi_0)\xi_0^{-1}=6$) and the second estimate of~\eqref{T0}, that $H_0 \in \widetilde{NF}_0(\omega_0,z_0,s_0,r_0,\xi_0,F,\varepsilon,m)$. So Proposition~\ref{algostep} can be applied.

If Alternative $(1)$ of Proposition~\ref{algostep} holds true, the solution $z_0^+(t)$ of $H_0^+=H_0\circ \Phi_0$ satisfies $z_0^+(t) \in V_{s_0,\xi_0}(z_0)$ for $0 \leq t \leq \bar{t}_0$. As $\Phi_0$ sends $V_{2s_{0},2\xi_0}(z_0)$ into $V_{3s_{0},3\xi_0}(z_0)$ and $\bar{t} \leq \bar{t}_0$, then $\Phi_0(z_0^+(t))=z_0(t)=z(t)$ satisfies in particular
\begin{equation}\label{estfinal}
|I(z(t))-I(z_0)| \leq 3s_0 \leq s, \quad 0 \leq t \leq \bar{t}, 
\end{equation}
the proposition is proved and the algorithm stops.

If Alternative $(2)$ of Proposition~\ref{algostep} holds true, then there exist a $T_{1}$-periodic vector $\omega_{1}\in \R^n\setminus\{0\}$, linearly independent from $\omega_0$ with the estimate
\begin{equation}\label{T1}
(F'\sqrt{n} s_0)^{-1}< T_{1} < \sqrt{n} \mu_0^{-1} s_0^{-p_{n-1}} Q^{n-1}, 
\end{equation}
and
\[ z_{1}=z_0^+(\tilde{t}_0) \in B_{r_1}, \quad s_{1}=2\sqrt{n-1}(T_{1}Q)^{-1}, \quad r_{1}=r_0+\xi_0, \quad \xi_{1}=\xi_0/3 \]  
such that $H_0^+ \in \widetilde{NF}_1(\omega_0,\omega_{1},z_{1},s_{1},r_{1},\xi_{1},F,\varepsilon,m)$ and 
\begin{equation}\label{apart222}
z_0^+(t) \in V_{s_0,\xi_0}(z_0), \quad 0 \leq t \leq \tilde{t}_0, 
\end{equation}  
provided that 
\begin{equation}\label{cond1}
\begin{cases}\tag{C1}
Q>E\sqrt{n-1}(8\eta)^{-1}, \\
Q>E\sqrt{n-1}(8\delta)^{-1}, \\
Q \geq 8F'\sqrt{n(n-1)}, \\
2\sqrt{n}F'\sqrt{n}^{a_1^1+a_0^0}\sqrt{n-1}^{-(a_1^1+a_0^0)}\mu_{0}^{-1}\kappa^{-a_1^1} Q^{n(a_1^1+a_0^0)}\varepsilon \leq 1, \\
180r(n-1)^{-a_1^1}n^{a_1^1}\sqrt{n-1}\kappa^{-2p_{n-1}}\mu_0^{-2}mQ^{2na_1^1-1}\varepsilon \leq 1, \\
Q \geq m216.10(3F\sqrt{n}+1)\sqrt{n-1}.
\end{cases}
\end{equation}
Indeed, using the definitions of $s_0$, $s_1$, $r_0$, $r_1$, $\xi_0$, $\xi_1$ (in particular, we use the facts that $s_1 \leq s_0$, $s_1 \geq \sqrt{n-1}(T_1Q)^{-1}$, $r_1+3\xi_1=5r/6$ and $(r_1+3\xi_1)\xi_1^{-1}=15$) and the estimate~\eqref{T0} and~\eqref{T1} on respectively $T_0$ and $T_1$, one can check that~\eqref{cond1} imply~\eqref{thr4} for $j=1$. Setting $H_1:=H_0^+\in \widetilde{NF}_1(\omega_0,\omega_{1},z_{1},s_{1},r_{1},\xi_{1},F,\varepsilon,m)$, we can apply Proposition~\ref{algostep} again. 

If Alternative $(1)$ holds true, then the solution $z_1^+(t)$ of $H_1^+=H_1\circ \Phi_1=H_0^+\circ \Phi_1=H_0 \circ \Phi_0 \circ \Phi_1$ starting at $z_1^+=\Phi_1^{-1}(z_1)$ satisfies $z_1^+(t) \in V_{s_1,\xi_1}(z_1)$ for $0 \leq t \leq \bar{t}_1$. As $\Phi_1$ sends $V_{2s_{1},2\xi_1}(z_1)$ into $V_{3s_{1},3\xi_1}(z_1)$, then $\Phi_1(z_1^+(t))=z_1(t)$ belongs to $V_{3s_{1},3\xi_1}(z_1)$ for $0 \leq t \leq \bar{t}_1$. By~\eqref{incl}, $V_{3s_{1},3\xi_1}(z_1)$ is contained in $V_{2s_{0},2\xi_0}(z_0)$, and as $\bar{t} \leq \bar{t}_1$, $z_1(t)$ belongs to $V_{2s_{0},2\xi_0}(z_0)$ for $0 \leq t \leq \bar{t}$.  Now observe that since $z_1=z_0^+(\tilde{t}_0)$, by uniqueness of the solutions associated to the system defined by $H_1=H_0^+$, we have the equality $z_1(t)=z_0^+(t+\tilde{t}_0)$ as long as the solution is defined. Using this equality, what we have proved is that
\begin{equation*}
z_0^+(t) \in V_{2s_{0},2\xi_0}(z_0), \quad \tilde{t}_0 \leq t \leq \tilde{t}_0+\bar{t} 
\end{equation*}    
But recall that from~\eqref{apart222}, we know that 
\[ z_0^+(t) \in V_{s_0,\xi_0}(z_0), \quad 0 \leq t \leq \tilde{t}_0,\]
and therefore, since $\bar{t} < \tilde{t}_0+\bar{t}$, we have in particular
\[ z_0^+(t) \in V_{2s_0,2\xi_0}(z_0), \quad 0 \leq t \leq \bar{t}. \] 
As before, using this and the fact that $\Phi_0$ sends $V_{2s_{0},2\xi_0}(z_0)$ into $V_{3s_{0},3\xi_0}(z_0)$ we also arrive at the estimate~\eqref{estfinal}. 

If Alternative $(2)$ holds true, then the algorithm continues. To apply Proposition \ref{algostep} at a step $j$, for $1 \leq j \leq n-1$, it is sufficient to check that 
\begin{equation}\label{condj}
\begin{cases}\tag{C$j$}
Q>E\sqrt{n-1}(8\eta)^{-1}, \\
Q>E\sqrt{n-1}(8\delta)^{-1}, \\
Q \geq 8F'\sqrt{n(n-1)}, \\
2\sqrt{n}F'\sqrt{n}^{a_j^j+a_{j-1}^{j-1}}\sqrt{n-1}^{-(a_j^j+a_{j-1}^{j-1})}\nu_{j}^{-1}\nu_{j-1}^{-1}\kappa^{-(\pi_j^j+\pi_{j-1}^{j-1})} Q^{n(a_j^j+a_{j-1}^{j-1})}\varepsilon \leq 1, \\
2^{j}27(3+3^{-j})r(n-1)^{-a_j^j}n^{a_j^j}\sqrt{n-1}\kappa^{-2\pi_j^j}\nu_j^{-2}mQ^{2na_j^j-1}\varepsilon \leq 1, \\
Q \geq m216(3^{j+1}+1)(3F\sqrt{n}+1)\sqrt{n-1}.
\end{cases}
\end{equation}
Indeed, \eqref{condj} implies \eqref{thr4}, using the definitions of $s_i$, $r_i$ and $\xi_i$ for $0 \leq i \leq j$ (which imply in particular that the $s_i$ are decreasing, $s_i \geq \sqrt{n-1}(T_iQ)^{-1}$, $r_i+3\xi_i=r(3+3^{-i})/4$ and $(r_i+3\xi_i)\xi_i^{-1}=3(3^{i+1}+1)/2$), and the estimates on the period $T_i$ that one obtains at each step using~\eqref{borneperiode}. To conclude, just observe that the conditions~\eqref{cond} imply the conditions~\eqref{cond0} and~\eqref{condj} for any $1 \leq j \leq n-1$.  For $j=n-1$, there is only one possibility in Proposition~\eqref{algostep}, the algorithm stops and the statement is proved.  This ends the proof.
\end{proof}

\begin{proof}[Proof of Theorem~\ref{mainthm03}]
We just need to choose $m$ and $Q$ in Proposition~\ref{algo} in terms of our given parameters. First we choose $m$ in terms of $Q$ as follows:
\[ m:=[b_1Q], \quad b_1=(216(3F\sqrt{n}+1)(3^n+1)\sqrt{n-1})^{-1} \]
where $[\,\cdot\,]$ denotes the integer part. Using this choice, the conditions~\eqref{cond} are implied by
\begin{equation}\label{condd}
Q \geq b_2, \quad Q \geq b_3\delta^{-1}, \quad Q \geq b_4 r^{-2}, \quad rb_5 Q^{2na}\varepsilon \leq 1, \quad b_6Q^{n(a+a')}\varepsilon \leq 1,
\end{equation} 
where
\[ b_2:=\max\{8F'\sqrt{n(n-1)},E\sqrt{n-1}(8\eta)^{-1},b_1^{-1}\} \]
\[ b_3:=E\sqrt{n-1}8^{-1} \]
\[ b_4:=5^{-1}36E\sqrt{n-1} \]
\[ b_5:=2^{n-1}27(3+3^{-n+1})n^{a}(n-1)^{-a}\sqrt{n-1} \nu_{n-1}^{-2} \kappa^{-2\pi_{n-1}^{n-1}} b_1  \] 
\[ b_6:=2\sqrt{n}F'\sqrt{n}^{a+a'}\sqrt{n-1}^{-(a+a')}\nu_{n-1}^{-1}\nu_{n-2}^{-1}\kappa^{-(\pi_{n-1}^{n-1}+\pi_{n-2}^{n-2})}. \]
Then we choose $Q$ as follows:
\[ Q:=(b_5r\varepsilon)^{-\frac{1}{2na}}\]
and observe that~\eqref{condd} becomes
\begin{equation}\label{conddd}
r\varepsilon \leq b_5^{-1}b_2^{-2na}, \quad r\varepsilon \leq b_5^{-1}b_3^{-2na}\delta^{2na}, \quad r\varepsilon \leq b_5^{-1}b_4^{-2na}r^{4na}, \quad r\varepsilon \leq b_6^{-\frac{2a}{a-a'}}b_5^{\frac{a+a'}{a-a'}}r^{\frac{2a}{a-a'}}.
\end{equation} 
With these choices of $m$ and $Q$, since $m>b_1Q-1$ we have
\[ s=3E\sqrt{n-1}b_5^{\frac{1}{2na}}(r\varepsilon)^{\frac{1}{2na}} \]
and 
\[ \bar{t}\geq 3(4E\sqrt{n-1})^{-1}b_5^{-\frac{1}{2na}}r^{-1}(r\varepsilon)^{-\frac{1}{2na}}\exp\left((\ln2) b_1 b_5^{-\frac{1}{2na}}r^{-1}(r\varepsilon)^{-\frac{1}{2na}}\right) \]
so if we define
\[\tilde{c}_1:=b_5^{-1}b_2^{-2na}, \quad \tilde{c}_2:=b_5^{-1}b_3^{-2na}, \quad \tilde{c}_3:=b_5^{-1}b_4^{-2na}, \quad \tilde{c}_4:=b_6^{-\frac{2a}{a-a'}}b_5^{\frac{a+a'}{a-a'}}  \]
and
\[  \tilde{c}_5:=2E\sqrt{n-1}b_5^{\frac{1}{2na}}, \quad \tilde{c}_6:=3(4E\sqrt{n-1})^{-1}b_5^{-\frac{1}{2na}}, \quad \tilde{c}_7:=(\ln2) b_1 b_5^{-\frac{1}{2na}} \]
we eventually obtain that if
\begin{equation}\label{seuil}
r\varepsilon \leq \min\left\{\tilde{c}_1,\tilde{c}_2\delta^{2na},\tilde{c}_3r^{4na},\tilde{c}_4r^{\frac{2a}{a-a'}}\right\}
\end{equation}
then
\[ |I(z(t))-I(z_0)|\leq \tilde{c}_5(r\varepsilon)^{\frac{1}{2na}}, \quad 0 \leq t \leq \tilde{c}_6r^{-1}(r\varepsilon)^{-\frac{1}{2na}} \exp\left(\tilde{c}_7r^{-1}(r\varepsilon)^{-\frac{1}{2na}}\right).  \]
This proves the statement for positive times, but for negative times, the proof is of course the same, so this concludes the proof.
\end{proof}

\appendix

\section{Proof of generic steepness} \label{app.steep}

The aim of this section is to give the proof of Theorem~\ref{proputile2}. The latter will be an immediate consequence of Propositions \ref{propNekho} and \ref{propNekho2} below. 
We shall use in the proof of these propositions basic results concerning semi-algebraic subsets; for proofs and more information we refer to \cite{BCR98}. 
Our main ingredient to prove Theorem~\ref{proputile2} is a result of Nekhoroshev on stably expanding polynomials that we will now state. 

Let us first recall that $P(n,m)$ denotes the space of polynomials of degree $m$ in $n$ variables with real coefficients, and $P_2(n,m)$ the subspace of $P(n,m)$ consisting of polynomials with vanishing homogeneous parts of order zero and one. The following definition, which is related to the definition of stably steep polynomials, is due to Nekhoroshev (\cite{Nek73}).

\begin{definition}\label{stabexp}
Let $1 \leq l \leq n-1$. A polynomial $Q_0 \in P_2(l,m)$ is called stably expanding if there exist a neighborhood $U_l$ of $Q_0$ in $P_2(l,m)$ and positive constants $C_l',\delta_l'$ such that for any $Q \in U_l$, the inequality
\[ \max_{0 \leq \eta \leq \xi}\;\min_{||y||=\eta}||\nabla Q(y)||>C_l'\xi^{m-1} \]
holds true for all $0 < \xi \leq\delta_l'$. \end{definition}

The set of stably expanding polynomials in $P_2(l,m)$ will be denoted by $SE(l,m)$. 

\begin{theorem}[Nekhoroshev]\label{thmNekho}
Let $1 \leq l \leq n-1$. The complement of $SE(l,m)$ in $P_2(l,m)$ is contained in a closed semi-algebraic subset $\Sigma(l,m)$ of codimension $[m/2]$. 
\end{theorem}  

Let us denote by $L(n,l)$ the space of rectangular matrices with $n$ rows and $l$ columns, with real coefficients, and by $L_1(n,l)$ the open subset of $L(n,l)$ consisting of matrices of maximal rank. Any $A \in L(n,l)$ induces a linear map $A : \R^l \rightarrow \R^n$, hence given $P \in P(n,m)$, we can define $P_A \in P(l,m)$ by setting $P_A(x)=P(Ax)$, $x \in \R^l$. Moreover, if $P \in P_2(n,m)$, then $P_A \in P_2(l,m)$. Let us define the set
\[ \Theta(l,n,m_0)=\{(P,A,Q) \in P_2(n,m_0) \times L_1(n,l) \times \Sigma(l,m_0) \; | \; P_A=Q \}. \]
Then we define $\Upsilon(l,n,m_0)$ to be the projection of $\Theta(l,n,m_0)$ on the first factor $P_2(n,m_0)$, and finally
\[ \Upsilon(n,m_0)=\bigcup_{l=1}^{n-1}\Upsilon(l,n,m_0). \]  
Theorem~\ref{proputile2} is a straightforward consequence of the following two properties of the set $\Upsilon(n,m_0)$. 

\begin{proposition}\label{propNekho}
The set $\Upsilon(n,m_0)$ is a semi-algebraic subset of $P_2(n,m_0)$ of codimension at least one.
\end{proposition} 

\begin{proposition}\label{propNekho2}
The complement of $SS(n,m_0)$ in $P_2(n,m_0)$ is contained in $\Upsilon(n,m_0)$. 
\end{proposition} 

The second proposition is true for any $m \geq 2$ and not just for $m=m_0$, but this will not be needed. 



Let us now give the proof of Proposition~\ref{propNekho} and Proposition~\ref{propNekho2}, following the arguments in \cite{Nek73}.

\begin{proof}[Proof of Proposition~\ref{propNekho}]
The set $P_2(n,m_0)$ is a real vector space hence it is algebraic, $L_1(n,l)$ is obviously an algebraic subset of $L(n,l)$ whereas, by Theorem~\ref{thmNekho}, $\Sigma(l,m_0)$ is a semi-algebraic subset of $P_2(l,m_0)$. Moreover, for $(P,A,Q) \in P_2(n,m_0) \times L_1(n,l) \times \Sigma(l,m_0)$, the equality $P_A=Q$ corresponds to a system of algebraic equations in the coefficients of $P$, $A$ and $Q$. This implies that $\Theta(l,n,m_0)$ is a semi-algebraic subset of $P_2(n,m_0) \times L(n,l) \times P_2(l,m_0)$. Now since the projection of a semi-algebraic subset is a semi-algebraic subset, $\Upsilon(l,n,m_0)$ is a semi-algebraic subset of $P_2(n,m_0)$. Then, as a finite union of semi-algebraic subsets is semi-algebraic, $\Upsilon(n,m_0)$ is a semi-algebraic subset of $P_2(n,m_0)$. We need to prove that the codimension of $\Upsilon(n,m_0)$ in $P_2(n,m_0)$ is at least one; to do this it is sufficient to prove that the codimension of $\Upsilon(l,n,m_0)$ in $P_2(n,m_0)$ is at least one for any $1 \leq l \leq n-1$. So let us fix $1 \leq l \leq n-1$. Given $(A,Q) \in L(n,l) \times P_2(l,m_0)$, we define $\Theta_{A,Q}(l,n,m_0)$ to be the intersection of $\Theta(l,n,m_0)$ with the set
\[ \{(P',A',Q') \in P_2(n,m_0) \times L(n,l) \times P_2(l,m_0) \; | \; A'=A, \; Q'=Q \}.  \]
If $(A,Q) \in L_1(n,l) \times \Sigma(l,m_0)$, it is easy to see that
\[ \mathrm{dim}\Theta_{A,Q}(l,n,m_0)=\mathrm{dim}P_2(n,m_0)-\mathrm{dim}P_2(l,m_0) \]  
and therefore
\begin{eqnarray*}
\mathrm{dim}\Theta(l,n,m_0) & = & \mathrm{dim}\Theta_{A,Q}(l,n,m_0) + \mathrm{dim}L_1(n,l) + \mathrm{dim}\Sigma(l,m_0) \\
& = & \mathrm{dim}P_2(n,m_0)-\mathrm{dim}P_2(l,m_0) + \mathrm{dim}L_1(n,l) + \mathrm{dim}\Sigma(l,m_0) \\
& = & \mathrm{dim}P_2(n,m_0)+ \mathrm{dim}L_1(n,l) - \mathrm{codim}\Sigma(l,m_0) \\
& = & \mathrm{dim}P_2(n,m_0)+ nl - [m_0/2]
\end{eqnarray*}
where in the last equality we used the fact that $\mathrm{dim}L_1(n,l)=\mathrm{dim}L(n,l)=nl$ and Theorem~\ref{thmNekho}. Now given $P \in P_2(n,m_0)$, we define $\Theta_{P}(l,n,m_0)$ to be the intersection of $\Theta(l,n,m_0)$ with the set
\[ \{(P',A',Q') \in P_2(n,m_0) \times L(n,l) \times P_2(l,m_0) \; | \; P'=P\}.  \]
Recall that if $GL(l)$ denotes the group of square invertible matrix of size $l$, with real coefficients, then $GL(l)$ acts freely on $L_1(n,l)$ (the quotient space is nothing but the Grassmannian $G(l,n)$, that is, the space of all $l$-dimensional subspaces of $\R^n$). It is then easy to see that $GL(l)$ acts freely on $\Theta_{P}(l,n,m_0)$, therefore the dimension of an orbit of this action equals the dimension of $GL(l)$, which is $l^2$, and hence, 
\[ \mathrm{dim}\Theta_{P}(l,n,m_0)=l^2. \]
Since $\Upsilon(l,n,m_0)$ is the projection of $\Theta(l,n,m_0)$ on the first factor $P_2(n,m_0)$, we have
\begin{eqnarray*}
\mathrm{dim}\Upsilon(l,n,m_0) & \leq & \mathrm{dim}\Theta(l,n,m_0)-l^2 \\
& \leq & \mathrm{dim}P_2(n,m_0)+ nl - [m_0/2] -l^2 \\
& \leq & \mathrm{dim}P_2(n,m_0)- [m_0/2] +l(n-l) \\
& \leq & \mathrm{dim}P_2(n,m_0)- [m_0/2] +[n^2/4] \\
& \leq & \mathrm{dim}P_2(n,m_0)-1
\end{eqnarray*}
where the last inequality follows from the definition of $m_0$. This proves that $\Upsilon(l,n,m_0)$ has codimension at least one in $P(n,m_0)$ for any $1 \leq l \leq n-1$, therefore $\Upsilon(n,m_0)$ has codimension at least one in $P(n,m_0)$ and this concludes the proof.   
\end{proof}

\begin{proof}[Proof of Proposition~\ref{propNekho2}]
To prove that the complement of $SS(n,m_0)$ in $P_2(n,m_0)$ is contained in $\Upsilon(n,m_0)$, we will prove that the complement of $\Upsilon(n,m_0)$ in $P_2(n,m_0)$ is contained in $SS(n,m_0)$. So we fix $P_0 \in P_2(n,m_0) \setminus \Upsilon(n,m_0)$ and $1 \leq l \leq n-1$. We denote by $O(n,l)$ the subset of $L_1(n,l)$ consisting of matrices whose columns are orthonormal vectors for the Euclidean scalar product. Recalling that the Grassmannian $G(l,n)$ is the quotient of $L_1(n,l)$ by $GL(l)$, it is also the quotient of $O(n,l)$ by the group $O(l)$ of orthogonal matrices of $\R^l$. Therefore given any $\Lambda_0 \in G(l,n)$, there exist an open neighborhood $B_{\Lambda_0}$ of $\Lambda_0$ in $G(l,n)$ and a continuous map $\Psi : B_{\Lambda_0} \rightarrow O(n,l)$ such that, if $\pi : O(n,l) \rightarrow G(l,n)$ denotes the canonical projection, then $\pi \circ \Psi$ is the identity. Let us now consider the continuous map
\[ F : P_2(n,m_0) \times B_{\Lambda_0} \rightarrow P_2(l,m_0), \quad F(P,\Lambda)=P_{\Psi(\Lambda)}.   \]
Since $P_0$ does not belong to $\Upsilon(n,m_0)$, by definition of the latter set it comes that $F(P_0,\Lambda)$ does not belong to $\Sigma(l,m_0)$ and therefore, by Theorem~\ref{thmNekho}, $F(P_0,\Lambda) \in SE(l,m_0)$ for any $\Lambda \in B_{\Lambda_0}$. Hence, by definition of $SE(l,m_0)$, there exist a neighborhood $U_l$ of $F(P_0,\Lambda)$ in $P_2(l,m)$ and positive constants $C_l', \delta_l'$ such that for any $Q \in U_l$, the inequality
\[ \max_{0 \leq \eta \leq \xi}\;\min_{||y||=\eta}||\nabla Q(y)||>C_l'\xi^{m_0-1} \]
holds true for all $0 < \xi \leq\delta_l'$. Now by continuity of $F$, we can find a neighborhood $V_l$ of $P_0$ in $P_2(n,m_0)$ and an open neighborhood $B_{\Lambda_0}' \subseteq B_{\Lambda_0}$ of $\Lambda_0$ in $G(l,n)$ such that $F(V_l \times B_{\Lambda_0}')$ is contained in $U_l$. So for any $P \in V_l$ and any $\Lambda \in B_{\Lambda_0}'$, we have 
\[ \max_{0 \leq \eta \leq \xi}\;\min_{||y||=\eta}||\nabla F(P,\Lambda)(y)||>C_l'\xi^{m_0-1} \]
for all $0 < \xi \leq\delta_l'$. Now since the columns of the matrix $\Psi(\Lambda)$ form an orthonormal basis of $\Lambda$, setting $x=\Psi(\Lambda)y$, $x \in \Lambda$, $||x||=||y||$ and hence 
\[ \min_{||y||=\eta}||\nabla F(P,\Lambda)(y)||=\min_{||x||=\eta, \; x \in \Lambda}||\Pi_\Lambda\nabla P(x)||=\min_{||x||=\eta, \; x \in \Lambda}||\nabla P_\Lambda(x)|| \]
where $\Pi_\Lambda$ is the orthogonal projection onto $\Lambda$, and $P_\Lambda$ is the restriction of $P$ to $\Lambda$. Therefore, for any $P \in V_l$ and any $\Lambda \in B_{\Lambda_0}'$, we have   
\[ \max_{0 \leq \eta \leq \xi}\;\min_{||x||=\eta, \; x \in \Lambda}||\nabla P_\Lambda(x)||>C_l'\xi^{m_0-1} \]
for all $0 < \xi \leq\delta_l'$. To conclude, since the Grassmannian $G(l,n)$ is compact, it can be covered by a finite number of neighborhoods of the form $B_{\Lambda_0}'$, $\Lambda_0 \in G(l,n)$, and hence one can certainly find positive constants $C_l, \delta_l$ such that for any $P \in V_l$ and any $\Lambda \in G(l,n)$, the inequality  
\[ \max_{0 \leq \eta \leq \xi}\;\min_{||x||=\eta, \; x \in \Lambda}||\nabla P_\Lambda(x)||>C_l\xi^{m_0-1} \]
holds true for all $0 < \xi \leq\delta_l$. This means that $P_0 \in SS(n,m_0)$, and this finishes the proof.   
\end{proof}

\section{Birkhoff normal forms with estimates}\label{s3}

The goal of this section is to give the proof of Proposition \ref{propBirkhoffutile}  using the work of Delshams and Gutiérrez (\cite{DG96b}). 

Given $l \in \N$ and $P$ a homogeneous polynomial in $\xi$ of degree $l$, if $P(\xi)=\sum_{|\nu|=l}P_\nu\xi^\nu$, we define the norm
\begin{equation}\label{normpol}
||P||:= \sum_{|\nu|=l}|P_\nu|.
\end{equation}
By our analyticity assumption on the Hamiltonian $H$ in \eqref{Hamintro}, we have the following expansion at the origin
\[ H(\xi)=\sum_{l \geq 2}H_l(\xi)=i\sum_{j=1}^n \alpha_j \xi_j \xi_{n+j}+\sum_{l \geq 3}H_l(\xi) \]
and there exist positive constants $c$ and $d$, which depends only on $n$, $R$ and $||H||_R$ such that for any integer $l \geq 2$, 
\begin{equation}\label{analytic}
||H_l|| \leq c^{l-2}d.
\end{equation}
Using Cauchy formula one easily proves that
\[ ||H_l||\leq (2R)^{-l}(e(2n+1))^l||H||_R \]
and therefore one can choose
\begin{equation}\label{cd}
c:=(2R)^{-1}e(2n+1), \quad d:=(2R)^{-2}(e(2n+1))^2||H||_R.
\end{equation}

Given any function $f$ that can be written as $f=\sum_{k}P_k$, with each $P_k$ homogeneous of degree $k$ in $\xi$, one easily check that
\begin{equation}\label{estnormpol}
\sup_{\xi \in \mathcal{B}_\rho}|f(\xi)| \leq \sum_{k}||P_k||\rho^k,
\end{equation}
and, if $g=\sum_{k, \, k\;\mathrm{even}}Q_k$, with each $Q_k$ homogeneous of degree $k/2$ in $I(\xi)$, then
\begin{equation}\label{estnormpol2}
\sup_{I \in \mathcal{D}_\rho}|g(I)| \leq \sum_{k, \, k\;\mathrm{even}}||Q_k||(\rho^2/2)^{k/2}.
\end{equation} 
Moreover, the above estimates hold true if $f$ is replaced by a tensor-valued function. Recall the definition of $\Psi_\a$ given in~\eqref{fonctionpsi}. Recall that we also defined for any integer $j \geq 3$, $\psi_\alpha^j=\prod_{i=3}^{j}\Psi_\alpha(i)$
and for convenience, we set $\psi_\alpha^2:=1$. We can finally state the main technical proposition of \cite{DG96b}.

\begin{proposition}[Delshams-Gutiérrez]\label{propBirkohffDG} 
Let $H$ be as in \eqref{Hamintro} with 
$\a$ as in \eqref{fonctionpsi} and consider an integer $K \geq 4$. If we define 
\begin{equation*}
\rho_K:=(548ncd K\Psi(K))^{-1},
\end{equation*}
then there exists a real-analytic symplectic transformation $\Phi^K= \text{Id} + O(\xi^2)$ defined on $\mathcal{B}_{\rho_K}$ such that $H \circ \Phi^K$ is in Birkhoff normal form up to order $K$, that is
\[ H \circ \Phi^K(\xi)=\alpha \cdot I(\xi)+\sum_{k \;\mathrm{even}, \: 4 \leq k \leq K}h_k(I(\xi))+\sum_{k \geq K+1}f_k(\xi)\]
where $h_k$ is a homogeneous polynomial of degree $k/2$ in $I(\xi)$, $f_k$ a homogeneous polynomial of degree $k$ in $\xi$, with the following estimates:
\begin{equation*}
||h_k|| \leq 6^{-1}(6cd)^{k-2}(k-2)!\psi_\alpha^{k-1}, \quad k\; \mathrm{even}, \; 4 \leq k \leq K;
\end{equation*}
\begin{equation*}
||f_k|| \leq 20d^2(20cd)^{k-2}(K-3)!(K-2)^{k-K+2}\psi_\alpha^{K-1}\Psi_\alpha(K)^{k-K+2}, \quad k \geq K+1.
\end{equation*}
\end{proposition}

This is exactly the statement of Proposition 1 in \cite{DG96b}, to which we refer for the proof. We will now arrange these estimates in a way that will be more convenient for us.

\begin{proposition}\label{propBirkohff}
Let $H$ be as in \eqref{Hamintro} with 
$\a$ as in \eqref{fonctionpsi}. Given an integer $p \geq 2$ and $K \geq 2p$, we have the following estimates on the homogeneous polynomials of Proposition \ref{propBirkohffDG}:
\begin{equation}\label{estpolh}
||h_k|| \leq \beta(p)\rho_K^{-(k-2p)}, \quad 2p \leq k \leq K;
\end{equation}
where
\begin{equation}\label{cm}
\beta(p):=6^{-1}(6cd)^{2p-2}(2p-2)!\psi_\alpha^{2p-1},
\end{equation}
and, given an integer $0 \leq q \leq K-4$, we have
\begin{equation}\label{estpolf}
||f_k|| \leq \tilde{\beta}(q) \rho_K^{-(k-q-1)}, \quad k \geq K+1.
\end{equation} 
where
\begin{equation*}
\tilde{\beta}(q):=c^{-1}d(20cd)^q(q+2)!\psi_\alpha^{q+2}.
\end{equation*}
\end{proposition}

The proof of Proposition~\ref{propBirkohff} is straightforward from Proposition~\ref{propBirkohffDG}. Now from these estimates on the homogeneous parts of $h^m$ and $f^K$, we will deduce the estimates  of Proposition \ref{propBirkhoffutile}.

\begin{proof}[Proof of Proposition~\ref{propBirkhoffutile}]
Recall that
\[ h^m(I(\xi)) =\sum_{k \;\mathrm{even}, \: 4 \leq k \leq K}h_k(I(\xi))\]
where each $h_k$ is homogeneous of degree $k/2$. For $p \geq 2$ and $k \geq 2p$, $\nabla^p h_k$ is a tensor-valued homogeneous polynomial of degree $(k-2p)/2$ and one can easily check (see \cite{DG96b}, estimates $(24)$), that
\[ ||\nabla^p h_k|| \leq (k/2)^p  ||h_k||.  \] 
Using this inequality, inequality~\eqref{estnormpol2} and the estimate~\eqref{estpolh} we get
\begin{eqnarray*}
||\nabla^p h^m||_\rho & \leq & \sum_{k \;\mathrm{even}, \: 2p \leq k \leq K}||\nabla^p h_k||(\rho^2/2)^{(k-2p)/2} \\
& \leq & \sum_{k \;\mathrm{even}, \: 2p \leq k \leq K}(k/2)^p|| h_k||(\rho^2/2)^{(k-2p)/2} \\
& \leq & \beta(p)\sum_{k \;\mathrm{even}, \: 2p \leq k \leq K}(k/2)^p \rho_K^{-(k-2p)}(\rho^2/2)^{(k-2p)/2} \\
& \leq & \beta(p) \sum_{k \;\mathrm{even}, \: 2p \leq k \leq K}(k/2)^p(1/2)^{(k-2p)/2} (\rho/\rho_K)^{k-2p} \\
& \leq & b(p)
\end{eqnarray*}
since the sum can be bounded by the corresponding series which is convergent. The same bound applies to $||\nabla^j h^m||_\rho$ for any $j$ such that $2 \leq j \leq p$, hence 
\[ ||\nabla^2 h^m||_{p,\rho}=\max_{2 \leq j \leq p}||\nabla^j h^m||_\rho \leq b(p).\]
Concerning 
\[ f^K(\xi) =\sum_{k \geq K+1}f_k(\xi), \]
since $\nabla f_k$ is a vector-valued homogeneous polynomial of degree $k-1$, we have
\[ ||\nabla f_k|| \leq k ||f_k|| \]
and so, using this inequality together with inequality~\eqref{estnormpol} and the estimate~\eqref{estpolf} we obtain
\begin{eqnarray*}
||\nabla f^K||_\rho & \leq & \sum_{k \geq K}||\nabla f_k||\rho^{k-1} \leq \sum_{k \geq K} k||f_k||\rho^{k-1} \\
& \leq & \tilde{\beta}(q)\sum_{k \geq K} \rho_K^{-(k-q-1)}\rho^{k-1} \leq \tilde{\beta}(q)\rho^q \sum_{k \geq K} (\rho/\rho_K)^{k-q-1} \\
& \leq & \tilde{b}(q)\rho^q e^{K}.
\end{eqnarray*}
This concludes the proof.
\end{proof}

\section{Proof of the normal form statement} \label{app}

\subsection{Technical estimates}\label{tech}

We first derive technical estimates for real-analytic vector fields defined on certain domains in $\C^{2n}$. These estimates are stated and proved for Hamiltonian vector fields, even though the Hamiltonian character plays absolutely no role here.

For $0 \leq j \leq n-1$, recall that $\omega_j \in \R^n \setminus\{0\}$ are $T_j$-periodic vectors, and that $\omega_{-1}=0 \in \R^n$. We write $l_{\omega_j}(z)=\omega_j \cdot I(z)$, for $1 \leq j \leq n-1$, and we define the complex vector space
\[ \tilde{\Lambda}_j = \{ v \in \C^n \; | \; v \cdot \omega_{-1}=v \cdot \omega_{0}=\cdots=v \cdot \omega_{j-1}=0\}. \]
Then we consider three sequences of positive real numbers $r_j$, $\xi_j$ and $s_j$, a sequence of points $z_j \in B_{r_j}$ and we let $\tilde{\lambda}_j=I(z_j)+\tilde{\Lambda}_j$ be
the complex affine subspace associated to $\tilde{\Lambda}_j$ passing through $I(z_j)$. The complex domains we consider are given by
\begin{equation*}
\mathcal{V}_{s_j,\xi_j}(z_j)=\{ z \in \C^{2n} \; | \; I(z) \in\tilde{\lambda}_j, \;  |I(z)-I(z_j)| <s_j, \; ||z||< r_j+\xi_j \}.
\end{equation*}
We fix $0 < \sigma_j < s_j$ and $0 < \rho_j < \xi_j$, and a real-analytic Hamiltonian vector field $X_{\chi_j}$ defined on $\mathcal{V}_{s_j,\xi_j}(z_j)$. Throughout this section, we will make the following two assumptions:
\begin{equation}\label{ass}
\sigma_j \leq (r_j+\xi_j)\rho_j, \quad \{l_{\omega_{-1}},\chi_j\}=\{l_{\omega_{-0}},\chi_j\}=\cdots=\{l_{\omega_{j-1}},\chi_j\}=0.
\end{equation}

\begin{lemma}\label{tech1}
Assume that~\eqref{ass} is satisfied. Then $X_{\chi_j}^t : \mathcal{V}_{s_j-\sigma_j,\xi_j-\rho_j}(z_j) \rightarrow \mathcal{V}_{s_j,\xi_j}(z_j)$ is a well-defined symplectic real-analytic embedding for all $|t|\leq \tau_j=(r_j+\xi_j)^{-1}\sigma_j||X_{\chi_j}||_{s_j,\xi_j}^{-1}$, with the estimate $||X_{\chi_j}^t-\mathrm{Id}||_{s_j-\sigma_j,\xi_j-\rho_j} \leq |t|||X_{\chi_j}||_{s_j,\xi_j}$.
\end{lemma}

\begin{proof}
Let $z \in \mathcal{V}_{s_j-\sigma_j,\xi_j-\rho_j}(z_j)$ and $z(t)=X_{\chi_j}^t(z)$ for small $|t|$, and let $|s|\leq |t|$. Since 
\[ \{l_{\omega_{l}},\chi_j\}(z(s))=\omega_l \cdot (\{I_1,\chi_j\}(z(s)), \dots ,\{I_n,\chi_j\}(z(s))):=\omega_l \cdot \{I,\chi_j\}(z(s))\] 
for $-1 \leq l \leq j-1$, the second part of~\eqref{ass} implies that 
\[ \omega_{-1}\cdot \{I,\chi_j\}(z(s))=\omega_{0}\cdot \{I,\chi_j\}(z(s))=\cdots=\omega_{j-1}\cdot \{I,\chi_j\}(z(s))=0, \]
so $\{I,\chi_j\}(z(s)) \in \tilde{\Lambda}_j$ which implies that
\[ \frac{d}{ds}I(z(s))=\{I,\chi_j\}(z(s)) \in \tilde{\Lambda}_j \]
and therefore
\[ I(z(t))=I(z)+\int_{0}^t \frac{d}{ds}I(z(s))ds =I(z)-I(z_j)+I(z_j)+\int_{0}^t \frac{d}{ds}I(z(s))ds \in \tilde{\lambda}_j. \]
Then, using the first part of~\eqref{ass}, for 
\[ |t|\leq \min\{\rho_j, (r_j+\xi_j)^{-1}\sigma_j\}||X_{\chi_j}||_{s_j,\xi_j}^{-1}=(r_j+\xi_j)^{-1}\sigma_j||X_{\chi_j}||_{s_j,\xi_j}^{-1}=\tau_j, \]
we have 
\[ ||z(t)-z|| \leq |t|||X_{\chi_j}||_{r_j,\xi_j} \leq \rho_j \]
and, using Cauchy-Schwarz inequality, 
\[ |I(z(t))-I(z)| \leq 2^{-1}||z(t)+z||||z(t)-z|| \leq (r_j+\xi_j) ||z(t)-z|| \leq (r_j+\xi_j)|t|||X_{\chi_j}||_{s_j,\xi_j} \leq \sigma_j.  \]
This proves that $X_{\chi_j}^t : \mathcal{V}_{s_j-\sigma_j,\xi_j-\rho_j}(z_j) \rightarrow \mathcal{V}_{s_j,\xi_j}(z_j)$ is a well-defined symplectic real-analytic embedding for $|t|\leq \tau_j$, 
with the estimate $||X_{\chi_j}^t-\mathrm{Id}||_{s_j-\sigma_j,\xi_j-\rho_j} \leq |t|||X_{\chi_j}||_{s_j,\xi_j}$. 
\end{proof}

\begin{lemma}\label{tech2}
Assume that~\eqref{ass} is satisfied, and let $X_f$ be a real-analytic Hamiltonian vector field defined on $\mathcal{V}_{s_j,\xi_j}(z_j)$. Then, for $|t|\leq \tau_j/3=(3(r_j+\xi_j))^{-1}\sigma_j||X_{\chi_j}||_{s_j,\xi_j}^{-1}$, we have
\[||(X_{\chi_j}^t)^*X_f||_{s_j-\sigma_j,\xi_j-\rho_j}\leq \left(1+3(r_j+\xi_j) \sigma_j^{-1}|t|||X_{\chi_j}||_{s_j,\xi_j}\right) ||X_f||_{s_j-2\sigma_j/3,\xi_j-2\rho_j/3}\]
and therefore
\[ ||(X_{\chi_j}^t)^*X_f||_{s_j-\sigma_j,\xi_j-\rho_j}\leq 2 ||X_f||_{s_j-2\sigma_j/3,\xi_j-2\rho_j/3}. \]
\end{lemma}

\begin{proof}
Let $|t|\leq \tau_j/3=(3(r_j+\xi_j))^{-1}\sigma_j||X_{\chi_j}||_{s_j,\xi_j}^{-1}$. We have the following expression
\[ (X_{\chi_j}^t)^*X_f=(DX_{\chi_j}^{-t} \circ X_{\chi_j}^{t}).(X_f \circ X_{\chi_j}^t)=\left(DX_{\chi_j}^{-t} \circ X_{\chi_j}^{t}-\mathrm{Id}\right).(X_f \circ X_{\chi_j}^t) + X_f \circ X_{\chi_j}^t. \]
Lemma~\ref{tech1} implies that $X_{\chi_j}^t : \mathcal{V}_{s_j-\sigma_j,\xi_j-\rho_j}(z_j) \rightarrow \mathcal{V}_{s_j-2\sigma_j/3,\xi_j-2\rho_j/3}(z_j)$ hence
\[ ||DX_{\chi_j}^{-t} \circ X_{\chi_j}^{t}-\mathrm{Id}||_{s_j-\sigma_j,\xi_j-\rho_j} \leq ||DX_{\chi_j}^{-t}-\mathrm{Id}||_{s_j-2\sigma_j/3,\xi_j-2\rho_j/3}. \]
We claim that
\[ ||DX_{\chi_j}^{-t}-\mathrm{Id}||_{s_j-2\sigma_j/3,\xi_j-2\rho_j/3} \leq 3(r_j+\xi_j) \sigma_j^{-1} ||X_{\chi_j}^{-t}-\mathrm{Id}||_{s_j-\sigma_j/3,\xi_j-\rho_j/3}  \]
while obviously, using Lemma~\ref{tech1},
\[ ||X_{\chi_j}^{-t}-\mathrm{Id}||_{s_j-\sigma_j/3,\xi_j-\rho_j/3}=||X_{-\chi_j}^{t}-\mathrm{Id}||_{s_j-\sigma_j/3,\xi_j-\rho_j/3} \leq |t|||X_{-\chi_j}||_{s_j,\xi_j}=|t|||X_{\chi_j}||_{s_j,\xi_j}. \]
Assuming this claim, using the expression for $(X_{\chi_j}^t)^*X_f$ and putting all the estimates together, we arrive at
\[ ||(X_{\chi_j}^t)^*X_f||_{s_j-\sigma_j,\xi_j-\rho_j} \leq \left(1+3(r_j+\xi_j)\sigma_j^{-1}|t|||X_{\chi_j}||_{s_j,\xi_j}\right) ||X_f \circ X_{\chi_j}^t||_{s_j-\sigma_j,\xi_j-\rho_j},  \] 
therefore
\[ ||(X_{\chi_j}^t)^*X_f||_{s_j-\sigma_j,\xi_j-\rho_j} \leq \left(1+3(r_j+\xi_j) \sigma_j^{-1}|t|||X_{\chi_j}||_{s_j,\xi_j}\right) ||X_f||_{s_j-2\sigma_j/3,\xi_j-2\rho_j/3}  \]
and also
\[ ||(X_{\chi_j}^t)^*X_f||_{s_j-\sigma_j,\xi_j-\rho_j} \leq 2 ||X_f||_{s_j-2\sigma_j/3,\xi_j-2\rho_j/3} \]
since $|t|\leq \tau_j/3=(3(r_j+\xi_j))^{-1}\sigma_j|X_{\chi_j}|_{s_j,\xi_j}^{-1}$. It remains to prove the claim. Let $F=X_{\chi_j}^{-t}-\mathrm{Id}$, $z\in \mathcal{V}_{s_j-2\sigma_j/3,\xi_j-2\rho_j/3}(z_j)$ and $v \in \C^{2n}$ a unit vector. The map 
\[ \xi \in \C \mapsto F_{z,v}(\xi)=F(z+\xi v) \in \C^{2n} \]
is holomorphic for $|\xi|\leq (3(r_j+\xi_j))^{-1}\sigma_j\leq \rho_j/3$, with $z+\xi v \in \mathcal{V}_{s_j-\sigma_j/3,\xi_j-\rho_j/3}(z_j)$. The usual Cauchy's estimate implies that
\[ ||DF(z)||=\sup_{||v||=1}||DF(z).v||=\sup_{||v||=1}||F_{z,v}'(0)|| \leq 3(r_j+\xi_j)\sigma_j^{-1} \sup_{|\xi|\leq (3(r_j+\xi_j))^{-1}\sigma_j}||F_{z,v}(\xi)||  \]
hence
\[  ||DF(z)|| \leq 3(r_j+\xi_j)\sigma_j^{-1}||F||_{s_j-\sigma_j/3,\xi_j-\rho_j/3}   \]
and the claim follows since $z\in \mathcal{V}_{s_j-2\sigma_j/3,\xi_j-2\rho_j/3}(z_j)$ was arbitrary.
\end{proof}

\begin{lemma}\label{tech3}
Assume that~\eqref{ass} is satisfied, and let $X_f$ be a real-analytic Hamiltonian vector field defined on $\mathcal{V}_{s_j,\xi_j}(z_j)$. Then
\[ ||[X_f,X_{\chi_j}]||_{s_j-2\sigma_j/3,\xi_j-2\rho_j/3} \leq 9(r_j+\xi_j)\sigma_j^{-1}||X_{\chi_j}||_{s_j,\xi_j}||X_f||_{s_j,\xi_j}.  \]
\end{lemma}

\begin{proof}
We have the expression
\[ [X_f,X_{\chi_j}]= \left.\frac{d}{dt}(X_{\chi_j}^t)^*X_f \right\vert_{t=0} \]
so for $z \in \mathcal{V}_{s_j-\sigma_j,\xi_j-\rho_j}(z_j)$, let us define the holomorphic map
\[  t \in \C \mapsto F_z(t)=(X_{\chi_j}^t)^*X_f(z) \in \C^{2n} \]
for $|t|\leq \tau_j/3=(3(r_j+\xi_j))^{-1}\sigma_j||X_{\chi_j}||_{s_j,\xi_j}^{-1}$. By Cauchy's estimate
\[ ||[X_f,X_{\chi_j}](z)||=||F'_z(0)||\leq 3\tau_j^{-1}||F_z(t)|| \leq 3\tau_j^{-1} ||(X_{\chi_j}^t)^*X_f||_{s_j-\sigma_j,\xi_j-\rho_j} \] 
and by Lemma~\ref{tech2}
\[  3\tau_j^{-1} ||(X_{\chi_j}^t)^*X_f||_{s_j-\sigma_j,\xi_j-\rho_j}\leq 6\tau_j^{-1} ||X_f||_{s_j-2\sigma_j/3,\xi_j-2\rho_j/3}\leq 6(r_j+\xi_j)\sigma_j^{-1}||X_{\chi_j}||_{s_j,\xi_j}||X_f||_{s_j,\xi_j}.  \]
Since $z \in \mathcal{V}_{s_j-\sigma_j,\xi_j-\rho_j}(z_j)$ was arbitrary, this proves that
\[ ||[X_f,X_{\chi_j}]||_{s_j-\sigma_j,\xi_j-\rho_j} \leq 6(r_j+\xi_j)\sigma_j^{-1}||X_{\chi_j}||_{s_j,\xi_j}||X_f||_{s_j,\xi_j}  \]
and the lemma follows by simply replacing $\sigma_j$ and $\rho_j$ by respectively $2\sigma_j/3$ and $2\rho_j/3$.  
\end{proof}

\begin{lemma}\label{tech4}
Assume that~\eqref{ass} is satisfied, and let $X_k$ be a real-analytic Hamiltonian vector field defined on $\mathcal{V}_{r_j,\xi_j}(z_j)$, which is integrable, that is $k$ is a function of $I(z)$ alone. Then
\[ ||[X_k,X_{\chi_j}]||_{s_j-2\sigma_j/3,\xi_j-2\rho_j/3}\leq 9 \rho_j^{-1}||X_{\chi_j}||_{s_j,\xi_j}||X_k||_{s_j,\xi_j}.  \]
\end{lemma}

\begin{proof}
Here we write
\[ [X_k,X_{\chi_j}]=-[X_{\chi_j},X_k]= -\left.\frac{d}{dt}(X_k^t)^*X_{\chi_j} \right\vert_{t=0} \]
and we observe that since $X_k$ is integrable, if we let $z(t)=X^t_k(z)$, then $I(z(t))=I(z)$. This implies that $X_k^t : \mathcal{D}_{s_j,\xi_j-\rho_j}(z_j)\rightarrow \mathcal{D}_{s_j,\xi_j}(z_j)$ is a well-defined symplectic real-analytic embedding for all $|t|\leq \tau_j'=\rho_j||X_k||_{s_j,\xi_j}^{-1}$. The conclusion follows easily by repeating all the previous arguments with $\tau_j'$ instead of $\tau_j$.
\end{proof}

\subsection{Proof of the Proposition~\ref{normal}}

Proposition~\ref{normal} will be proved by iterating $m$ times an averaging procedure, which is classical in the case $j=0$, but more involved in the general case. 

\begin{proof}[Proof of Proposition~\ref{normal}]
Let us fix $0 \leq j \leq n-1$, and set $\varepsilon_j:=2^j\varepsilon$. The integer $m \geq 1$ being given, for $0 \leq i \leq m$ we define
\[ \varepsilon_j^i:=2^{-i}\varepsilon_j, \quad \gamma_j^i:=(1-2^{-i})2\varepsilon_j, \quad s_j^i:=3s_j-is_j/m, \quad \xi_j^i:=3\xi_j-i\xi_j/m. \]
Then we claim that for each $0 \leq i \leq m$, there exists a real-analytic symplectic embedding $\Phi_j^i :  \mathcal{V}_{s_j^{i},\xi_j^{i}}(z_j) \rightarrow \mathcal{V}_{3s_j,3\xi_j}(z_j)$ such that 
\[ (H_j-f_j) \circ \Phi_j^i =(h+g_j) \circ \Phi_j^i= h +g_j^i +f_j^i \] 
with 
\[ \{l_{\omega_{-1}},g_j^i\}=\{l_{\omega_{0}},g_j^i\}=\cdots=\{l_{\omega_{j}},g_j^i\}=0, \quad \{l_{\omega_{-1}},f_j^i\}=\{l_{\omega_{0}},f_j^i\}=\cdots=\{l_{\omega_{j-1}},f_j^i\}=0 \] 
and with the estimates
\begin{equation*}
||X_{g_j^i}||_{s_j^{i},\xi_j^{i}} \leq \gamma_j^i, \quad ||X_{f_j^i}||_{s_j^{i},\xi_j^{i}} \leq \varepsilon_j^i, \quad ||\Phi_j^i - \mathrm{Id}||_{s_j^{i},\xi_j^{i}} \leq T_j\gamma_j^i.
\end{equation*}  
Let us prove the claim by induction on $0 \leq i \leq m$.

For $i=0$, letting $\Phi_j^0$ be the identity, $g_j^0:=0$ and $f_j^0:=g_j$, there is nothing to prove. Then assume that the statement holds true for some $0 \leq i \leq m-1$, and let $H_j^i=(H_j-f_j) \circ \Phi_j^i = h +g_j^i +f_j^i$. We define the functions
\[ [f_j^i]_j:=T_j^{-1}\int_0^{T_j} f_j^i \circ X_{\omega_j}^tdt, \quad \chi^i_j:=T_j^{-1}\int_0^{T_j} t(f_j^i-[f_j^i]_j) \circ X_{\omega_j}^tdt \]
whose associated Hamiltonian vector fields are given by 
\[ X_{[f_j^i]_j}=T_j^{-1}\int_0^{T_j} (X_{\omega_j}^t)^* f_j^i dt, \quad X_{\chi_j^i}=T_j^{-1}\int_0^{T_j} t(X_{\omega_j}^t)^*(f_j^i-[f_j^i]_j)dt \]
with, using our inductive assumption, the following obvious estimates
\begin{equation}\label{estim1}
||X_{[f_j^i]_j}||_{s_j^i,\xi_j^i} \leq ||X_{f_j^i}||_{s_j^i,\xi_j^i} \leq \varepsilon_j^i, \quad ||X_{\chi_j^i}||_{s_j^i,\xi_j^i} \leq T_j ||X_{f_j^i}||_{s_j^i,\xi_j^i} \leq T_j \varepsilon_j^i. 
\end{equation}
It is clear that 
\begin{equation}\label{pourlesg}
\{[f_j^i]_j,l_{\omega_j}\}=0. 
\end{equation}
For $-1 \leq l \leq j-1$, observe that $\{l_{\omega_l},l_{\omega_j}\}=0$, so $l_{\omega_l} \circ X_{\omega_j}^t=l_{\omega_l}$ and hence 
\begin{eqnarray*}
\{l_{\omega_l},[f_j^i]_j\} & = & T_j^{-1}\int_0^{T_j} \{l_{\omega_l},f_j^i \circ X_{\omega_j}^t\}dt \\
 & = & T_j^{-1}\int_0^{T_j} \{l_{\omega_l}\circ X_{\omega_j}^t,f_j^i \circ X_{\omega_j}^t\}dt \\
 & = & T_j^{-1}\int_0^{T_j} \{l_{\omega_l},f_j^i \} \circ X_{\omega_j}^t dt 
\end{eqnarray*}
where the last equality follows from the symplectic character of $X_{\omega_j}^t$. Using our inductive assumption, this implies
\begin{equation}\label{pourlesg2}
\{l_{\omega_{-1}},[f_j^i]_j\}=\{l_{\omega_{0}},[f_j^i]_j\}=\cdots=\{l_{\omega_{j-1}},[f_j^i]_j\}=0,
\end{equation} 
and by a completely similar argument, we also get
\begin{equation}\label{pourlechi}
\{l_{\omega_{-1}},\chi_j^i\}=\{l_{\omega_{0}},\chi_j^i\}=\cdots=\{l_{\omega_{j-1}},\chi_j^i\}=0.
\end{equation}
Now set
\[ \sigma_j:=s_j^{i}-s_j^{i+1}=s_j/m, \quad \rho_j:=\xi_j^{i}-\xi_j^{i+1}=\xi_j/m. \]
Since $\xi_j^i \geq 2\xi_j$, using the first inequality of~\eqref{thr1} we have
\[ \sigma_j=s_j/m \leq (r_j+2\xi_j)\xi_j/m \leq (r_j+\xi_j^i)\xi_j/m=(r_j+\xi_j^i)\rho_j \] 
and therefore, using also~\eqref{pourlechi}, we can apply Lemma~\ref{tech1}: 
\[ X_{\chi_j^i}^t : \mathcal{V}_{s_j^{i+1},\xi_j^{i+1}}(z_j)=\mathcal{V}_{s_j^{i}-\sigma_j,\xi_j^{i}-\rho_j}(z_j) \rightarrow \mathcal{V}_{s_j^{i},\xi_j^{i}}(z_j)\] 
is a well-defined symplectic real-analytic embedding for all 
\[ |t|\leq \tau_j=(r_j+\xi_j^i)^{-1}\sigma_j||X_{\chi_j^i}||_{s_j^{i},\xi_j^{i}}^{-1}=(m(r_j+\xi_j^i))^{-1}s_j||X_{\chi_j^i}||_{s_j^{i},\xi_j^{i}}^{-1},\] 
with the estimate $||X_{\chi_j^i}^t-\mathrm{Id}||_{s_j^{i+1},\xi_j^{i+1}} \leq |t|||X_{\chi_j^i}||_{s_j^{i},\xi_j^{i}}$. Moreover, as $\xi_j^i \leq 3\xi_j$, using the second estimate of~\eqref{estim1} and the second inequality of~\eqref{thr1}, we have
\begin{equation}\label{tau}
\tau_j \geq (m(r_j+3\xi_j)T_j\varepsilon_j^i)^{-1}s_j=2^i(m(r_j+3\xi_j)T_j2^j\varepsilon)^{-1}s_j \geq (m(r_j+3\xi_j)T_j2^j\varepsilon)^{-1}s_j \geq 216 
\end{equation}
so $\tau_j>1$ and hence $X_{\chi_j^i}^1 : \mathcal{V}_{s_j^{i+1},\xi_j^{i+1}}(z_j) \rightarrow \mathcal{V}_{s_j^{i},\xi_j^{i}}(z_j)$ is well-defined, with
\begin{equation}\label{estdist}
||X_{\chi_j^i}^1-\mathrm{Id}||_{s_j^{i+1},\xi_j^{i+1}} \leq ||X_{\chi_j^i}||_{s_j^{i},\xi_j^{i}} \leq T_j\varepsilon_j^i.
\end{equation}
It is easy to check, using an integration by parts, that $\{\chi_j^i,l_{\omega_j}\}=f_j^i-[f_j^i]_j$, and this equality, together with Taylor's formula with integral remainder gives
\[ (h +g_j^i +f_j^i) \circ X_{\chi_j^i}^1=h + g_j^i + [f_j^i]_j + \tilde{f}_j^i\]
with
\[ \tilde{f}_j^i=\int_0^1 \{(h-l_{\omega_j})+g_j^i+f_{j,t}^i,\chi_j^i\} \circ X^t_{\chi_j^i}dt, \quad f_{j,t}^i=tf_j^i+(1-t)[f_j^i]_j.\] 
We set $\Phi_j^{i+1}=\Phi_j^i \circ X_{\chi_j^i}^1$, $g_j^{i+1}=g_j^i + [f_j^i]_j$ and $f_j^{i+1}=\tilde{f}_j^i$ so that
\[ (H_j-f_j) \circ \Phi_j^{i+1}=H_j^i \circ X_{\chi_j^i}^1= (h +g_j^i +f_j^i) \circ X_{\chi_j^i}^1=h + g_j^{i+1} + f_j^{i+1}. \]
First observe that $\Phi_j^{i+1} :  \mathcal{V}_{s_j^{i+1},\xi_j^{i+1}}(z_j) \rightarrow \mathcal{D}_{3s_j,3\xi_j}(z_j)$ is a real-analytic symplectic embedding, and using~\eqref{estdist} together with our inductive assumption, we have the estimate
\[ ||\Phi_j^{i+1} - \mathrm{Id}||_{s_j^{i+1},\xi_j^{i+1}} \leq ||\Phi_j^i - \mathrm{Id}||_{s_j^{i},\xi_j^{i}} + ||X_{\chi_j^i}^1-\mathrm{Id}||_{s_j^{i+1},\xi_j^{i+1}} \leq T_j(\gamma_j^i+\varepsilon_j^i)=T_j\gamma_j^{i+1}. \] 
Then 
\[ \{l_{\omega_{-1}},g_j^{i+1}\}=\{l_{\omega_{0}},g_j^{i+1}\}=\cdots=\{l_{\omega_{j}},g_j^{i+1}\}=0 \]
follows from the definition of $g_j^{i+1}$, the inductive assumption,~\eqref{pourlesg} and~\eqref{pourlesg2}. Moreover, using the first estimate of~\eqref{estim1} and our inductive assumption,
\[ ||X_{g_j^{i+1}}||_{s_j^{i+1},\xi_j^{i+1}} \leq ||X_{g_j^{i}}||_{s_j^i,\xi_j^i} + ||X_{[f_j^i]_j}||_{s_j^i,\xi_j^i} \leq \gamma_j^i+\varepsilon_j^i=\gamma_j^{i+1}. \]
For $-1 \leq l \leq j-1$, we already know that
\[ \{l_{\omega_l},g_j^i\}=\{l_{\omega_l},f_j^i\}=\{l_{\omega_l},[f_j^i]_j\}=\{l_{\omega_l},\chi_j^i\}=0 \] 
which implies $\{l_{\omega_l},f_{j,t}\}=0$ whereas $\{l_{\omega_l},h-l_{\omega_j}\}=0$ is obvious. These equalities, together with Jacobi identity, imply that
\[ \{l_{\omega_l},\{(h-l_{\omega_j})+g_j^i+f_{j,t}^i,\chi_j^i\}\}=0 \] 
and therefore
\[ \{l_{\omega_l},f_j^{i+1}\}=\int_0^1 \{l_{\omega_l},\{(h-l_{\omega_j})+g_j^i+f_{j,t}^i,\chi_j^i\}\} \circ X^t_{\chi_j^i}dt. \]
It follows that
\[ \{l_{\omega_{-1}},f_j^{i+1}\}=\{l_{\omega_{0}},f_j^{i+1}\}=\cdots=\{l_{\omega_{j-1}},f_j^{i+1}\}=0. \]
To complete the proof of the claim, it remains to estimate
\[ X_{f_j^{i+1}}=\int_0^1 (X^t_{\chi_j^i})^*[X_{h-l_{\omega_j}}+X_{g_j^i}+X_{f_{j,t}^i},X_{\chi_j^i}] dt. \]
First, 
\[ ||X_{f_j^{i+1}}||_{s_j^{i+1},\xi_j^{i+1}} \leq \sup_{0 \leq t \leq 1}||(X^t_{\chi_j^i})^*[X_{h-l_{\omega_j}}+X_{g_j^i}+X_{f_{j,t}^i},X_{\chi_j^i}]||_{s_j^{i+1},\xi_j^{i+1}}  \]
and since $\tau>3$ by~\eqref{tau}, we can apply Lemma~\ref{tech2} to get
\begin{equation}\label{bout0}
||X_{f_j^{i+1}}||_{s_j^{i+1},\xi_j^{i+1}} \leq 2 ||[X_{h-l_{\omega_j}}+X_{g_j^i}+X_{f_{j,t}^i},X_{\chi_j^i}]||_{s_j^{i}-2\sigma_j/3,\xi_j^{i}-2\rho_j/3}. 
\end{equation}
Using Lemma~\ref{tech3}, we have
\[ ||[X_{f_{j,t}^i},X_{\chi_j^i}]||_{s_j^{i}-2\sigma_j/3,\xi_j^{i}-2\rho_j/3} \leq 9 (r_j+\xi_j^i)\sigma_j^{-1}||X_{f_{j,t}^i}||_{s_j^{i},\xi_j^{i}}||X_{\chi_j^i}||_{s_j^{i},\xi_j^{i}}  \]
and since
\[ ||X_{f_{j,t}^i}||_{s_j^{i},\xi_j^{i}} \leq \varepsilon_j^i, \quad ||X_{\chi_j^i}||_{s_j^{i},\xi_j^{i}} \leq T_j \varepsilon_j^i,  \]
we get
\begin{equation}\label{bout1}
||[X_{f_{j,t}^i},X_{\chi_j^i}]||_{s_j^{i}-2\sigma_j/3,\xi_j^{i}-2\rho_j/3} \leq (9(r_j+3\xi_j)mT_j\varepsilon_js_j^{-1})\varepsilon_j^i
\end{equation}
since $\xi_j^i \leq 3\xi_j$ and $\varepsilon_j^i \leq \varepsilon_j$. Similarly, since $\gamma_j^i \leq 2\varepsilon_j$, 
\begin{equation}\label{bout2}
||[X_{g_j^i},X_{\chi_j^i}]||_{s_j^{i}-2\sigma_j/3,\xi_j^{i}-2\rho_i/3} \leq (9(r_j+3\xi_j)mT_j\gamma_j^is_j^{-1})\varepsilon_j^i \leq (18(r_j+3\xi_j)mT_j\varepsilon_js_j^{-1})\varepsilon_j^i. 
\end{equation}
Concerning the last bracket, let us first prove that 
\[ [X_{h-l_{\omega_j}},X_{\chi_j^i}]=[X_{h_{\tilde{\lambda}_j}-l_{\omega_j}},X_{\chi_j^i}] \]
where we recall that $\tilde{\lambda}_j=I(z_j)+\tilde{\Lambda}_j$ and $h_{\tilde{\lambda}_j}$ is the restriction of $h$ to $\tilde{\lambda}_j$. To do this, it is sufficient to prove $\{h-l_{\omega_j},\chi_j^i\}=\{h_{\tilde{\lambda}_j}-l_{\omega_j},\chi_j^i\}$, which is equivalent to $\{h,\chi_j^i\}=\{h_{\tilde{\lambda}_j},\chi_j^i\}$. For any $z \in \mathcal{V}_{s_j^{i},\xi_j^{i}}(z_j)$, we have
\[ \{h,\chi_j^i\}(z)=\nabla h (I(z))\cdot \{I,\chi_j^i\}(z). \]
But for any $-1 \leq l \leq j-1$, we know that
\[ \{l_{\omega_l},\chi_j^i\}(z)={\omega_l}\cdot \{I,\chi_j^i\}(z)=0 \]
which means that $\{I,\chi_j^i\}(z) \in \tilde{\Lambda}_j$. Therefore, recalling that $\tilde{\Pi}_j$ denotes the orthogonal projection onto $\tilde{\Lambda}_j$, it comes that
\[ \{h,\chi_j^i\}(z)=\tilde{\Pi}_j(\nabla h (I(z)))\cdot \{I,\chi_j^i\}(z)=\nabla h_{\tilde{\lambda}_j} (I(z))\cdot \{I,\chi_j^i\}(z)=\{h_{\tilde{\lambda}_j},\chi_j^i\}(z) \]
and therefore $\{h,\chi_j^i\}=\{h_{\tilde{\lambda}_j},\chi_j^i\}$. Then, for any $z \in \mathcal{V}_{s_j^{i},\xi_j^{i}}(z_j)$, we can estimate
\begin{eqnarray*}
||\nabla h_{\tilde{\lambda}_j} (I(z))-\omega_j|| & = & ||\tilde{\Pi}_j(\nabla h (I(z)))-\omega_j|| \\
& \leq & ||\tilde{\Pi}_j(\nabla h (I(z)))-\tilde{\Pi}_j(\nabla h (I(z_j)))||+||\Pi_j(\nabla h (I(z_j)))-\omega_j|| \\
& \leq & ||\nabla h (I(z))-\nabla h (I(z_j))||+||\Pi_j(\nabla h (I(z_j)))-\omega_j|| \\
& \leq & F||I(z)-I(z_j)||+s_j \leq F\sqrt{n}|I(z)-I(z_j)|+s_j \\
& \leq & F\sqrt{n} s_j^i+s_j \leq (3F\sqrt{n}+1) s_j
\end{eqnarray*}  
since $s_j^i \leq 3s_j$, and where we used the fact that
\[\sup_{z \in \mathcal{V}_{3s_j,3\xi_j}(z_j)}||\nabla^2 h(I(z))|| \leq F.\]
From this, we deduce that
\begin{eqnarray*}
||X_{h_{\tilde{\lambda}_j}-l_{\omega_j}}||_{s_j^{i},\xi_j^{i}} & \leq & \sup_{z \in \mathcal{V}_{s_j^i,\xi_j^i}(z_j)}|\nabla (h_{\tilde{\lambda}_j}-l_{\omega_j})(I(z))|||z||| \leq  (3F\sqrt{n}+1)s_j(r_j+\xi_j^i) \\
& \leq & (3F\sqrt{n}+1)s_j(r_j+3\xi_j)  
\end{eqnarray*}
and using Lemma~\ref{tech4}, we get
\begin{eqnarray*}
||[X_{h-l_{\omega_j}},X_{\chi_j^i}]||_{s_j^{i}-2\sigma_j/3,\xi_j^{i}-2\rho_j/3} & = & ||[X_{h_{\tilde{\lambda}_j}-l_{\omega_j}},X_{\chi_j^i}]||_{s_j^{i}-2\sigma_j/3,\xi_j^{i}-2\rho_j/3} \\
& \leq & 9\rho_j^{-1} ||X_{h_{\tilde{\lambda}_j}-l_{\omega_j}}||_{s_j^{i},\xi_j^{i}}||X_{\chi_j^i}||_{s_j^{i},\xi_j^{i}}
\end{eqnarray*} 
hence
\begin{eqnarray}\label{bout3} \nonumber
||[X_{h-l_{\omega_j}},X_{\chi_j^i}]||_{s_j^{i}-2\sigma_j/3,\xi_j^{i}-2\rho_j/3} & \leq & 9\rho_j^{-1}(3F\sqrt{n}+1)s_j(r_j+3\xi_j)T_j\varepsilon_j^i \\ 
& = & (9(3F\sqrt{n}+1)\xi_j^{-1}(r_j+3\xi_j)mT_js_j)\varepsilon_j^i. 
\end{eqnarray}
Putting the estimates~\eqref{bout0},~\eqref{bout1},~\eqref{bout2} and~\eqref{bout3} together, and recalling that $\varepsilon_j=2^j\varepsilon$, we arrive at
\[ ||X_{f_j^{i+1}}||_{s_j^{i+1},\xi_j^{i+1}} \leq (2^j54(r_j+3\xi_j)mT_j\varepsilon s_j^{-1}+18(3F\sqrt{n}+1)\xi_j^{-1}(r_j+3\xi_j)mT_js_j)\varepsilon_j^i. \]
Using the second and third inequality of~\eqref{thr1}, we obtain
\[ ||X_{f_j^{i+1}}||_{s_j^{i+1},s_j^{i+1}} \leq \varepsilon_j^i/2=\varepsilon_j^{i+1}.\]
This finishes the proof of the claim.

Now let us define $\Phi_j=\Phi_j^m$, $g_j^+=g_j^m$ and $f_j^+=f_j^m + f_j \circ \Phi_j$. Since $s_j^m=2s_j$ and $\xi_j^m=2\xi_j$, $\Phi_j$ is a real-analytic symplectic embedding 
\[ \Phi_j :  \mathcal{V}_{2s_j,2\xi_j}(z_j) \rightarrow \mathcal{V}_{3s_j,3\xi_j}(z_j)\] 
such that $H_j \circ \Phi_j = h +g_j^+ +f_j^+$. We already know that $\{l_{\omega_{-1}},g_j^+\}=\{l_{\omega_{0}},g_j^+\}=\cdots=\{l_{\omega_{j}},g_j^+\}=0$, and the estimates
\[ ||X_{g_j^+}||_{2s_j,2\xi_j} \leq \gamma_j^m \leq 2\varepsilon_j= 2^{j+1}\varepsilon, \quad ||\Phi_0 - \mathrm{Id}||_{2s_j,2\xi_j} \leq T_j\gamma_j^m \leq 2^{j+1} T_j \varepsilon.\]
To conclude the proof of the proposition, it remains to estimate $X_{f_j^+}$. First recall that
\begin{equation}\label{boutf1}
||X_{f_j^m}||_{2s_j,2\xi_j} \leq \varepsilon_j^m=2^{-m}\varepsilon_j=2^j2^{-m}\varepsilon.
\end{equation}
Then, $f_j \circ \Phi_j=f_j \circ \Phi_j^m=f_j \circ X_{\chi_j^1}^1 \circ \cdots X_{\chi_j^m}^1$ and so $X_{f_j \circ \Phi_j}=(X_{\chi_j^m}^1)^*\cdots (X_{\chi_j^1}^1)^*X_{f_j}$, where
\[ ||X_{\chi_j^i}^1-\mathrm{Id}||_{s_j^{i+1},\xi_j^{i+1}} \leq ||X_{\chi_j^i}||_{s_j^{i},\xi_j^{i}} \leq T_j\varepsilon_j^i=2^{-i}T_j\varepsilon_j, \]
for each $1 \leq i \leq m$.
Applying Lemma~\ref{tech2} inductively yields
\begin{eqnarray*}
||X_{f_j \circ \Phi_j}||_{2s_j,2\xi_j} & = & ||(\Phi_j)^*X_{f_j}||_{s_j^{m},\xi_j^{m}} \leq \prod_{i=0}^{m-1}\left( 1+3(r_j+\xi_j^i) \sigma_j^{-1} ||X_{\chi_j^i}||_{s_j^{i},\xi_j^{i}}\right) ||X_{f_j}||_{2s_j,2\xi_j} \\
& = & \prod_{i=0}^{m-1}\left( 1+3(r_j+3\xi_j) s_j^{-1} m ||X_{\chi_j^i}||_{s_j^{i},\xi_j^{i}}\right) ||X_{f_j}||_{2s_j,2\xi_j} \\
& \leq & \prod_{i=0}^{m-1}( 1+2^{-i}3(r_j+3\xi_j) s_j^{-1} m T_j\varepsilon_j) ||X_{f_j}||_{2s_j,2\xi_j} \\
& \leq & \exp\left(\sum_{i=0}^{m-1} 2^{-i}3(r_j+3\xi_j) s_j^{-1} m T_j\varepsilon_j\right) ||X_{f_j}||_{2s_j,2\xi_j} \\
& \leq & \exp(6(r_j+3\xi_j) s_j^{-1} m T_j\varepsilon_j) ||X_{f_j}||_{2s_j,2\xi_j}.
\end{eqnarray*}
The second condition of~\eqref{thr1} implies in particular that $\exp(6(r_j+3\xi_j) s_j^{-1} m T_j\varepsilon_j) \leq 2$  and therefore
\begin{equation}\label{boutf2}
||X_{f_j \circ \Phi_j}||_{2s_j,2\xi_j} \leq 2 ||X_{f_j}||_{2s_j,2\xi_j} \leq 2(j2^{j-1})2^{-m}\varepsilon=j2^{j}2^{-m}\varepsilon.
\end{equation}
From~\eqref{boutf1} and~\eqref{boutf2} we get
\[ ||X_{f_j^+}||_{2s_j,2\xi_j} \leq 2^j2^{-m}\varepsilon+j2^{j}2^{-m}\varepsilon=(j+1)2^j2^{-m}\varepsilon, \]
and this ends the proof.
\end{proof}

\addcontentsline{toc}{section}{References}
\bibliographystyle{amsalpha}
\bibliography{EllipticSteep17}

\end{document}